\documentclass[UTF-8,reqno]{amsart}
\usepackage{enumerate}
\usepackage{mhequ}
\usepackage[margin=1.2in]{geometry}
\usepackage{amssymb,url,color, booktabs,nccmath}
\usepackage{mathrsfs}
\usepackage{enumitem}
\usepackage{graphicx}
\usepackage{tikz}
\usepackage{amsthm}
\usepackage{color}
\usepackage[colorlinks=true]{hyperref}
\hypersetup{
	linkcolor=blue,          
	citecolor=red,        
	filecolor=blue,      
	urlcolor=cyan
}

\definecolor{darkergreen}{rgb}{0.0, 0.5, 0.0}

\numberwithin{equation}{section}

\newcommand{\be}{\begin{eqnarray}}
	\newcommand{\ee}{\end{eqnarray}}
\newcommand{\ce}{\begin{eqnarray*}}
	\newcommand{\de}{\end{eqnarray*}}
\newtheorem{theorem}{Theorem}[section]
\newtheorem{lemma}[theorem]{Lemma}
\newtheorem{remark}[theorem]{Remark}
\newtheorem{definition}[theorem]{Definition}
\newtheorem{proposition}[theorem]{Proposition}
\newtheorem{Examples}[theorem]{Example}
\newtheorem{corollary}[theorem]{Corollary}

\newtheorem*{theorem*}{Theorem}
\newtheorem*{remark*}{Remark}

\newcommand{\assign}{:=}
\newcommand{\cdummy}{\cdot}
\newcommand{\comma}{{,}}
\newcommand{\mathd}{\mathrm{d}}

\newcommand{\nocomma}{}

\def\eps{\varepsilon}

\def\d{\hh}

\def\[{{\Big[}}
\def\]{{\Big]}}
\def\<{{\langle}}
\def\>{{\rangle}}
\def\({{\Big(}}
\def\){{\Big)}}

\def\bx{{\mathbf{x}}}

\def\dif{{\mathord{{\rm d}}}}

\def\={&\!\!=\!\!&}

\def\mT{{\mathbb T}}

\def\1{{\mathbf{1}}}

\def\geq{\geqslant}
\def\leq{\leqslant}

\def\eps{\varepsilon}

\def\d{\hh}

\def\[{{\Big[}}
\def\]{{\Big]}}
\def\<{{\langle}}
\def\>{{\rangle}}
\def\({{\Big(}}
\def\){{\Big)}}

\def\bx{{\mathbf{x}}}

\def\dif{{\mathord{{\rm d}}}}

\def\={&\!\!=\!\!&}
\def\bt{\begin{theorem}}
	\def\et{\end{theorem}}
\def\bl{\begin{lemma}}
	\def\el{\end{lemma}}
\def\br{\begin{remark}}
	\def\er{\end{remark}}
\def\bx{\begin{Examples}}
	\def\ex{\end{Examples}}
\def\bd{\begin{definition}}
	\def\ed{\end{definition}}
\def\bp{\begin{proposition}}
	\def\ep{\end{proposition}}
\def\bc{\begin{corollary}}
	\def\ec{\end{corollary}}

\def\geq{\geqslant}
\def\leq{\leqslant}

 \def\R{\mathbb R}
 \def\R{\mathbb R}    
\def\N{\mathbb N}  
   
\def\<{\langle} \def\>{\rangle}

\def\hh{{h}}

\allowdisplaybreaks

\begin{document}

	\title{The fluctuation behaviour  of the stochastic point vortex model with common noise}
	
	\author{Yufei Shao}
	\address[Y.Shao]{School of Mathematics and Statistics, Beijing Institute of Technology, Beijing 100081, China
	}
	\email{yufeishao@bit.edu.cn}

	\author{Xianliang Zhao}
	\address[X. Zhao]{ Academy of Mathematics and Systems Science,
		Chinese Academy of Sciences, Beijing 100190, China; Fakult\"at f\"ur Mathematik, Universit\"at Bielefeld, D-33501 Bielefeld, Germany}
	\email{xzhao@math.uni-bielefeld.de}

	\begin{abstract}
		This article studies the fluctuation behaviour of the stochastic point vortex model with common noise. Using the martingale method combined with a localization argument, we prove that the sequence of fluctuation processes converges in distribution to the unique probabilistically strong solution of a linear stochastic evolution equation. In particular, we establish the strong convergence from the stochastic point vortex model with common noise to the conditional McKean-Vlasov equation.
	\end{abstract}

	
	\date{\today}
	\thanks{Y.S. is grateful
		to the financial supports by  National Key R\&D Program of China (No. 2022YFA1006300) and the financial supports of the NSFC (No. 12271030, No.12426205). X.Z. is supported by the DFG through the CRC 1283 “Taming uncertainty and profiting from randomness and low regularity in analysis, stochastics and their applications”.}
	\maketitle
	
	\setcounter{tocdepth}{2}

	\tableofcontents

	\section{Introduction}
In this article, 	we investigate the fluctuation behavior of  the following weakly interacting particle system  with common noise  on torus $\mathbb{T}^2=[-\pi,\pi]^2,$ 
\begin{equation}\label{eqt:vortex}
	X^N_i(t)=X_i(0)+\frac{1}{N}\sum_{j\neq i}\int_0^tK(X^N_i(s)-X^N_j(s))\mathd s+\sqrt{2} B_i(t)+\int_0^t\sigma(X^N_i(s))\circ\mathd W_s,\quad t\in [0,T].
\end{equation}
This system is commonly referred to as the stochastic point vortex model.	Here $T>0,N\in \mathbb{N}$ and $\circ$ denotes the stochastic integral  in the Stratonovich sense. The interaction between particles is defined by   the Biot-Savart kernel $K$ on $\mathbb{T}^2$, namely 
	\begin{align}
		K(x)=-\frac{1}{2\pi}\frac{x^{\perp}}{|x|^2}+K_0(x), \quad x^{\perp }=(x_2,-x_1),\quad x=(x_1,x_2)\in \mathbb{T}^2,
	\end{align}
	where $K_0$ is a smooth correction to periodize $K$ on  torus $\mathbb{T}^2=[-\pi,\pi]^{2}.$ The common  noise shared by all particles is described by the term $\int_0^{\cdot}\sigma(X_i(s))\circ\mathd W_s,$  where $\sigma$ is a smooth and  divergence free vector field and $\{W_t ,t\in[0,T]\}$  represents a 1-dimensional standard Brownian motion.   Additionally,  $\{B_i, i\in \mathbb{N}\}$ are independent 2-dimensional Brownian motions  on torus $\mathbb{T}^2,$ modeling  the individual noise for each particle. The initial positions of the particles are given by a sequence of independent and identically distributed (i.i.d.) random variables $\{X^N_i(0),i\in \mathbb{N}\}$ taking values in $\mathbb{T}^2.$ The identical distribution for initial values $\{X^N_i(0),i\in \mathbb{N}\}$  is denoted by $\mathcal{L}(X(0)).$   We further assume that the initial positions $\{X^N_i(0),i\in \mathbb{N}\},$  the individual noises $\{B_i,i\in \mathbb{N}\}$ and environmental noise $W$  are mutually independent.
	
	    The aim of this paper is to study the asymptotic behavior of the fluctuation process \begin{align}\label{def:eta}
	    	\eta^N_t:=\sqrt{N}(\mu_N(t)-v_t)=\frac{1}{\sqrt{N}}\sum_{i = 1}^N\left( \delta_{X^N_i(t)}-v_t\right),\quad \forall t\in[0,T] ,\end{align} which describes the deviations of the empirical measures of the stochastic point vortex model \eqref{eqt:vortex} 
	     \begin{align}\label{def:empirical measure}
	    	\mu_N(t)\assign\frac{\sum_{i=1}^{N}\delta_{X^N_i(t)}}{N},\quad \forall t\in[0,T]
	    \end{align} 
	   from  the mean field limit $(v_t)_{t\in[0,T]}.$ Here $(v_t)_{t\in[0,T]}$ is the unique probabilistically strong solution to the following stochastic 2-dimensional Navier-Stokes equation \eqref{eqt:mean} on   $ [0,T]\times\mathbb{T}^2$ in the sense of Definitions  \ref{def:mean-strong} and \ref{def:uni-mean} below :
	    \begin{equation}\label{eqt:mean}
	    	\mathd   v=\(\Delta   v-K*  v \cdot \nabla   v\)\mathd t-\sigma\cdot\nabla   v \circ \mathd W_t  ,   \quad   v(0,\cdot)=  v_0.
	    \end{equation} 
	    This study is often referred to as the central limit theorem for interacting particle systems  \eqref{eqt:vortex}.  Compared to the  mean field limit result \eqref{defmeanfieldlimit}, it provides a more precise description of the relationship between the interacting particle system  \eqref{eqt:vortex} and the mean field limit equation \eqref{eqt:mean}, i.e, formally $$\mu_N\overset{\text{d}}{\approx} v+\frac{1}{\sqrt{N}}\eta,$$
	      where  $\overset{\text{d}}{\approx}$ means that the approximation holds in distribution and $\eta$ represents the limiting  process of the sequence  $\{\eta^N\}_{N\in \mathbb{N}}$ in the sense of convergence in distribution. Furthermore, $\eta$ is typically Gaussian-distributed in the absence of environmental noise  $W.$
	 
	      The classical mean field limit for particle systems without common noise, such as \eqref{eqt:vortex} with $\sigma=0,$ has been extensively studied over the past decade. The goal of studying the mean field limit is to analyze the asymptotic independence of particles. This result can be expressed in two forms, which are qualitatively equivalent, as established in  \cite{sznitman1991topics}. One form is that, for every $t\in[0,T],$ the empirical measure of the particle system $\mu_N(t)$ satisfies the weak convergence of measures, \begin{align}\label{defmeanfieldlimit}
	      	\mu_N(t)\rightharpoonup v_t.
	      \end{align}
	     The other form, often referred to as  {\em propagation of chaos}, states that, for fixed $k\in \mathbb{N},t\in [0,T],$
	      \begin{align*}
	      	\mathcal{L}^{N,k}(t)\rightharpoonup v_t^{\otimes k},
	      \end{align*}
	      where $	\mathcal{L}^{N,k}(t)$ is the $k$-marginal of the  particle distribution $\mathcal{L}(X_1^N(t),\cdots, X_N^N(t))$  for \eqref{eqt:vortex} with $\sigma=0.$  
	 The limiting measure $(v_t)_{t\in[0,T]}$ is the  solution to the following  nonlinear Fokker-Planck equation \begin{equation*}
	 	\partial_t  v=\Delta   v-K*  v \cdot \nabla   v,
	 \end{equation*} and coincides with the law of a solution to the related Mckean Vlasov equation. We refer readers to the works \cite{osada1986propagation,sznitman1991topics,fournier2014propagation, duerinckx2016mean,jabin2018quantitative,serfaty2020mean,lacker2023hierarchies,bresch2023mean,feng2023quantitative,guillin2024uniform,wang2024sharp,carrillo2024relative}  for classical results on mean field limit and reference therein. For interacting particle systems \eqref{eqt:vortex} with various kernels $K,$ particular attention has been paid to systems with singular kernels due to their physical relevance. One of the most well-known examples is  {\em the   Biot-Savart kernel} $K,$  which is also the focus of our study in this paper. The corresponding particle system is often referred to as the stochastic point vortex model, which describes the behavior of fluid. 
 Recent advancements in this area have highlighted the relative entropy method, which not only ensures convergence results for particle systems \eqref{eqt:vortex} with singular kernels but also provides quantitative convergence rates. In our study, we address the central limit problem for the stochastic point vortex model, leveraging mean field limit results expressed through the relative entropy framework. Additionally, we use this method to derive uniform estimates. The global relative entropy method was developed by Jabin and Wang in \cite{jabin2016mean} for second-order systems with bounded kernels and in \cite{jabin2018quantitative} for first-order systems with general  $W^{-1,\infty}$ kernels on torus, including the   Biot-Savart kernel. Feng and Wang \cite{feng2023quantitative} recently extended quantitative particle approximation results for the 2-dimensional Navier-Stokes equations to the whole space. More recently,   Carrillo, Feng, Guo, Jabin and Wang \cite{carrillo2024relative} used the relative entropy method to study the particle  approximation of the   spatially homogeneous Landau equation for Maxwellian molecules. In \cite{lacker2023hierarchies}, Lacker developed a new local relative entropy method, achieving optimal quantitative estimates between $\mathcal{L}^{N,k}(t)$ and $v_t^{\otimes k}.$  Recently, Wang \cite{wang2024sharp} extended this approach to handle particle systems with  singular $W^{-1,\infty}$ kernels.

A key distinction of our model compared to classical particle systems lies in the introduction of environmental noise $W,$ which induces stochasticity in the mean field limit $(v_t)_{t\in[0,T]}.$  Specifically,  the  mean field limit $(v_t)_{t\in[0,T]}$ satisfies a stochastic nonlinear Fokker-Planck equation \eqref{eqt:mean} with transport noise, rather than a deterministic partial differential equation.  Moreover, $(v_t)_{t\in[0,T]}$ coincides with the conditional law of $\bar{X}_i(t),$ which is the solution to  the following conditional McKean-Vlasov equation \eqref{eqt:ncopy},  with respect to  the environmental noise $\{W_t,t\in [0,T]\},$ i.e.,  $v_t(\mathd x)=\mathcal{L}(\bar{X}_i(t)| \mathcal{F}^{W}_T)(\mathd x)$ in the sense of $\mathcal{P}(\mathbb{T}^2),$\footnote{ We use $\mathcal{P}(\mathbb{T}^2)$   to denote the space of probability measures on $\mathbb{T}^2.$} 
 \begin{equation}\label{eqt:ncopy}
 	\bar{X}_i(t)=X_i(0)+\int_0^tK*  v_s(\bar{X}_i(s))\mathd s+\sqrt{2}B_i(t)+\int_0^t\sigma(\bar{X}_i(s))\circ\mathd W_s.
 \end{equation}
This setup allows us to establish a {\em conditional propagation of chaos} type result.  For fixed $k\in \mathbb{N},$ as the number of particles $N\rightarrow \infty,$  we have
\begin{align*}
	F^{N,k}(t)\rightharpoonup\bar{F}^{N, k}(t)\assign v^{\otimes k}_t.
\end{align*} 
Here $F^{N,k}(t)(\mathd x^N)$ is the $k$-marginal of the conditional distribution $F^N(t)(\mathd x^N)$ of $X^N(t)$  with respect to  the environmental noise $\{W_t,t\in [0,T]\},$ defined by 
\begin{equation}\label{def:F}
	\mathcal{L}(X^N(t)|\mathcal{F}_{T}^{W})(\mathd x^N), 
\end{equation}
and $\bar{F}^{N,k}(t)(\mathd x^N)$ is indeed the $k$-marginal of the conditional distribution $\bar{F}^N(t)(\mathd x^N)$ of $\bar{X}^N(t)$  with respect to  the environmental noise $\{W_t,t\in [0,T]\},$ defined by
\begin{align}\label{def:barF}
\mathcal{L}(\bar{X}^N(t)| \mathcal{F}^{W}_T)(\mathd x^N).
\end{align}
We refer to \cite{shao2024quantitative} for more details.

One motivation for considering a particle system with common noise is its ability to describe environmental influences on particles across various fields. For example, in mathematical finance, such a model reflects the fact that a large financial market should include a common set of assets accessible to all agents (see e.g., \cite{lacker2015stochastic}). Additionally, it can help us study the phenomenon known as  regularization by noise (see e.g., \cite{flandoli2011random}). However, the literature on the mean field limit in the presence of environmental noise remains limited.  We refer readers to the works \cite{kurtz1999particle,coghi2016propagation,rosenzweig2020mean,nguyen2021mean,shao2024quantitative,nikolaev2024quantitative}.  Relative entropy method has also been developed to address the case of environmental noise. In \cite{shao2024quantitative}, we derived a quantitative particle approximation for the stochastic 2-dimensional Navier-Stokes equation. In \cite{nikolaev2024quantitative}, Nikolaev established the relative entropy method for particle systems with a general kernel  $K\in L^{\infty}(\mathbb{R}^d)\cap L^{2}(\mathbb{R}^d)$ and Ito's noise on the whole space. Another interesting result is presented in  \cite{flandoli2021mean,guo2023scaling}, where it is shown that the mean field limit equation can still be a deterministic partial differential equation by rescaling the space covariance of the noise as the number of particles increases.

Recently,  Wang, Zhao and Zhu \cite{wang2023gaussian}  study the limiting behavior of the fluctuation process $\eta^N$ for the interacting particle system \eqref{eqt:vortex} with $\sigma=0$ on torus $\mathbb{T}^d$ by martingale method. They focused on systems with singular kernels satisfying  $\|K\|_{L^\infty}<\infty,$ or $K(x)=-K(-x)$ and $\|xK(x)\|_{L^\infty}< \infty,$  which include important examples such as the Biot-Savart kernel. Using the  It\^o's formula,  the  fluctuation measure $\eta^N$ can be formally represented by the following SPDE
\begin{align}
d \eta^N_t= &\Delta\eta^N_t \mathd t-\nabla\cdot\big(vK\ast\eta^N+\eta^NK\ast v+\frac{1}{\sqrt{N}}\eta^NK\ast\eta^N\big)\mathd t\nonumber\\&+\frac{1}{2}\sigma\cdot\nabla(\sigma\cdot\nabla\eta^N_t)\mathd t+ \mathd \mathcal{M}^N_t-(\sigma\cdot\nabla\eta^N_t)\mathd W_t,
\end{align}
where $\mathcal{M}^N_t$ is a continuous stochastic process taking values in $H^{-\alpha-1}(\mathbb{T}^2)$ for every $\alpha>1,$ introduced in Section \ref{sec:compactness} below.
The most challenging part lies in studying the uniform estimates and convergence for the interacting term $$vK\ast\eta^N+\eta^NK\ast v+\frac{1}{\sqrt{N}}\eta^NK\ast\eta^N,$$ due to the singularity of the kernel  $K$ and the fluctaution measures $(\eta^N)_{N\geq 1}$( which live in negative Sobolev spaces). By the  relative entropy method, Wang, Zhao and Zhu \cite{wang2023gaussian}  addressed the challenging issues of uniform estimates and the convergence of interaction term. They showed the convergence from the fluctaution measures $(\eta^N)_{N\geq 1}$ to the solution $\eta$ of the fluctuation SPDE \eqref{LimitSpde}  with $\sigma=0$ below.
Building on the results  of \cite{wang2023gaussian}, this paper extends their framework to incorporate the presence of common noise.  Specifically, we demonstrate that  the sequence of fluctuation measures $(\eta^N)_{N\geq 1}$ for the stochastic point vortex model with common noise  \eqref{eqt:vortex}
converges in distribution to a stochastic process $\eta$ in the space $L^2 ([0, T], H^{-\alpha}) \cap C ([0, T], H^{- \alpha - 2})$  for every $\alpha>1.$  Here $\eta$ is the unique probabilistically  strong solution to the following fluctuation SPDE \eqref{LimitSpde} 
\begin{equation}\label{LimitSpde}
	d \eta_t= \Delta\eta_t \mathd t-\nabla\cdot(v_tK*\eta_t)\mathd t-\nabla\cdot(\eta_t K*v_t)\mathd t+\frac{1}{2}\sigma\cdot\nabla(\sigma\cdot\nabla\eta_t)\mathd t+ \mathd \mathcal{M}_t-(\sigma\cdot\nabla\eta_t)\mathd W_t,
\end{equation}
where  $\{v_t,t\in[0,T]\}$ is the unique  probabilistically  strong solution to the stochastic 2-dimensional Navier-Stokes equation \eqref{eqt:mean} with initial value $v_0\in H^3$ with strictly positive lower bound, i.e., $\inf_{\mathbb{T}^2}v_0>0.$ 
 Compared to the case without environmental noise $W$,  the final fluctuation  equation \eqref{LimitSpde} includes two  noise terms. The first is a multiplicative transport noise $\sigma\cdot \nabla \eta_t\mathd W_t,$ which arises from the environmental noise $W.$  The second is an additive noise $\mathcal{M}_t.$ As shown in \cite{wang2023gaussian},  $\{\mathcal{M}_t,t\in [0,T]\}$  is a continuous Gaussian process  taking values in $H^{-\alpha-1}(\mathbb{T}^2)$ for every $\alpha>1,$ in the absence of environmental noise $W.$ In the presence of environmental noise, it remains a continuous stochastic process taking values in $H^{-\alpha-1}(\mathbb{T}^2)$ for every $\alpha>1,$
  but it is no longer Gaussian.  Instead, its conditional distribution  satisfies  
 	\begin{align}\label{conditional law}
 	\begin{split}
 		\mathbb{E}\[\exp{i\left\langle \varphi,\mathcal{M}_t \right\rangle}\mid \mathcal{F}^{W}_T\]&=\exp\bigg\{-\int_{0}^{t}\langle \mid \nabla\varphi\mid^2,v_s\rangle\mathd s\bigg\},
 		\\ \mathbb{E}\[\exp{i\left\langle \varphi,(\mathcal{M}_{t+r}-\mathcal{M}_{t}) \right\rangle}\mid \mathcal{F}^{W}_T\vee\mathcal{F}^{\mathcal{M}}_t\]&=\exp\bigg\{-\int_{t}^{t+r}\langle \mid \nabla\varphi\mid^2,v_s\rangle\mathd s\bigg\},
 	\end{split}
 \end{align}
   for every $\varphi\in C^{\infty}(\mathbb{T}^2),0\leq t<t+r\leq T,$ where
 $(\mathcal{F}^{W}_t)_{t\in[0,T]}$ is the normal filtration generated by environmental noise $W$ and  $(\mathcal{F}^{\mathcal{M}}_t)_{t\in[0,T]}$ is the normal filtration generated by additive noise $\mathcal{M}.$ Therefore, it can be inferred that under the influence of environmental noise $W,$ the distribution of $\eta$ is no longer Gaussian. Further explanations on the conditional distribution of $\mathcal{M}_t$ can be found in Remark \ref{remark-m}.

\begin{theorem}\label{thm:main}
	Assume  that the identical distribution $\mathcal{L}(X(0))$ on torus $\mathbb{T}^2$ for the i.i.d. initial values $\{X^N_i(0),i\in \mathbb{N}\}$  has a  density $v_0\in H^{3}(\mathbb{T}^2)$ with strictly positive lower bound, i.e., $\inf_{\mathbb{T}^2}v_0>0.$  the sequence of fluctuation measures $\eta^N$ for the stochastic point vortex model \eqref{eqt:vortex} converges in distribution to $\eta$ in the space $$L^2 ([0, T], H^{-
		\alpha}) \cap C ([0, T], H^{- \alpha - 2}),$$  for every $\alpha>1$, where $\eta$ is  the unique probabilistically  strong solution to the fluctuation SPDE \eqref{LimitSpde} in the sense of Definitions  \ref{def:stronglimit} and \ref{def:uni} below.
\end{theorem}
The definition of the solutions to the fluctuation SPDE  \eqref{LimitSpde}  can be found in Section \ref{sec:preliminaries}, Definitions \ref{def:stronglimit} and \ref{def:uni}. During this process, we also obtain the well-posedness of  the fluctuation SPDE  \eqref{LimitSpde}.
\begin{corollary}\label{thm:main1}
	Given a 1-dimensional Brownian motion $\{W_t, t\in[0,T]\}$ and  a continuous stochastic process  $\{\mathcal{M}_t,t\in [0,T]\}$ taking values in $H^{-\alpha-1}(\mathbb{T}^2)$ for every $\alpha>1,$ on probability space $(\Omega,\mathcal{F},\mathbb{P}),$ satisfying  \eqref{conditional law},
	for each $\eta_0\in H^{-\alpha},\forall\alpha>1,\mathbb{P}$-a.s., satisfying   for every $\varphi\in C^{\infty}(\mT^2),$
	\begin{align*}
		\mathcal{L}(	\left\langle \eta_0,\varphi\right\rangle )= \mathcal{N}\(0,\left\langle \varphi^2,v_0\right\rangle-\left\langle \varphi,v_0\right\rangle^2  \),
	\end{align*} 
	there exists a unique probabilistically strong solution to \eqref{LimitSpde}.
	Here $	\mathcal{L}(	\left\langle \eta_0,\varphi\right\rangle )$ denotes the distribution of $\left\langle \eta_0,\varphi\right\rangle ,$ and $\mathcal{N}(0,a)$ denotes the centered Gaussian distribution on $\mathbb{R}$ with variance $a.$ 
\end{corollary}

Similar to the mean-field limit result \eqref{defmeanfieldlimit}, the central limit theorem for the interacting particle system \eqref{eqt:vortex} reflects the asymptotic independence of particles. Specifically, in the absence of environmental noise $W,$  the distribution of the fluctuation measures $\{\eta^N\}_{N\in \mathbb{N}}$ becomes asymptotically Gaussian as $N\rightarrow\infty,$ like i.i.d. random variables. In the presence of environmental noise $W,$ our result shows that the asymptotic independence is reflected in an additive noise $\mathcal{M}_t,$ which from \eqref{conditional law} is “Gaussian conditioned on $\mathcal{F}^W_T.$"  Additionally, the effect of environmental noise $W$ is captured in the conditional distribution of $\mathcal{M}_t$ and the transport noise $\sigma\cdot \nabla \eta_t\mathd W_t,$ analogous to the mean-field limit equation \eqref{eqt:mean}.
  \subsection{Related literatures}   

For the fluctuations of interacting diffusions, which is the focus of this article, one of the earliest results is due to  It\^o \cite{ito1983distribution}, where he showed that  for  the system of 1-dimensional independent and identically distributed Brownian motions,  the limit of the corresponding fluctuations is a Gaussian process.
One common method to study the central limit theorem for interacting particle systems \eqref{eqt:vortex} is the martingale method,  which is also employed in this paper.  A significant contribution in this area was made by Fernandez and M\'el\'eard \cite{fernandez1997hilbertian}, who  studied the fluctuation behaviour of particle systems with well-regularized kernels and multiplicative independent noise on the whole space.   It was shown that  the fluctuation process $\eta^N$, as a weighted Sobolev space-valued random variable,  converges to a Gaussian-distributed limit in the sense of distribution, as $N\rightarrow \infty.$ The requirement for kernels with strong regularity arises for two main reasons. First, the uniform estimates needed to prove tightness rely on a coupling method, which requires at least Lipschitz regularity for the kernels. Second, when identifying the tight limit, the second-order differential operator in the fluctuation equation \eqref{LimitSpde} must be linear continuous in the weighted Sobolev space, necessitating  stricter regularity conditions. It is worth emphasizing that the  approach introduced in \cite{ fernandez1997hilbertian}  has been amplified to study  various interacting models, see \cite{jourdain1998propagation,chevallier2017fluctuations,chen2016fluctuation,lucon2016transition} etc.  Recent work by Wang, Zhao and Zhu \cite{wang2023gaussian} extends the results to the limiting behavior of the fluctuation process $\eta^N$ for interacting particle system \eqref{eqt:vortex} on torus with singular kernel  satisfying $\|K\|_{L^\infty}<\infty,$ or $K(x)=-K(-x)$ and $\|xK(x)\|_{L^\infty}< \infty$ (e.g.,   Biot-Savart kernel).  Using the relative entropy method and a structured observation of interaction term, they addressed the challenging issues of uniform estimates and the convergence of interaction terms when applying the martingale method to particle systems with singular kernels. Building on \cite{wang2023gaussian}, our work applies their framework to resolve singularity issues in the Biot–Savart law.  Another study closely related to our work is  \cite{kurtz2004stochastic}, which employed a martingale method combined with  coupling method to investigate the convergence of the fluctuation process for interacting particle systems with common noise and time-varying random intensities $(\xi_j)_{j\geq 1}$ in the modified Schwartz space.  In their model, the random intensities $\xi_j$ are governed by a stochastic differential equation driven by independent noise $B_j$ and the environmental noise $W.$ The limiting distribution of the fluctuation process is non-Gaussian, consistent with our findings.  Similar to \cite{fernandez1997hilbertian}, their method also required strong regularity assumptions on the interaction kernel $K.$
We also refer to the works \cite{oelschlager1987fluctuation,jourdain1998propagation,chen2024fluctuations} which study the fluctuations in the  moderate mean field regime. 

Additionally, another type result is known as the pathwise central limit theorem. This kind of result studies the limiting behavior of fluctuation processes based on particle trajectories, considering the entire paths of particles. For example, we consider the fluctuation processes of the form $\sqrt{N}(\frac{1}{N}\sum_{i=1}^{N}\delta_{X_i}-\mathcal{L}(X)),$   where $X\in C([0,T],\mathbb{R}^d)$ solves some nonlinear stochastic differential equation.
In this context, Tanaka and Hitsuda \cite{tanaka1981central} first studied a specific case with $K(x)=-\lambda x,\lambda>0.$ Later, Tanaka \cite{tanaka1984limit} extended the analysis to more general kernels $K\in C^2_b,$ using a pathwise construction approach. Sznitmann \cite{sznitman1984nonlinear} removed  the differentiability condition  on test functions and generalized the result to bounded and Lipschitz continuous kernel, using   Girsanov’s formula and the method of $U$-statistics.  Recently, Budhiraja and Wu \cite{budhiraja2016some} studied some general interacting systems with possible common factors, which do not necessarily  have the exchangability property as usual. Their result  follows the strategy by Sznitmann \cite{sznitman1984nonlinear}. More recently,  Chaintron \cite{chaintron2024quasi} further generalized the results to include the case of multiplicative independent noise.

The qualitative central limit theorem for second-order systems is explored in \cite{braun1977vlasov} and \cite{lancellotti2009fluctuations}.  Among these, \cite{braun1977vlasov} as the first to investigate the fluctuation behavior of second-order systems. More recently, significant advancements have been made in \cite{duerinckx2021size} and \cite{bernou2024uniform}. Duerinckx \cite{duerinckx2021size} achieved optimal quantitative fluctuation estimates, while Bernou and Duerinckx \cite{bernou2024uniform} established an uniform-in-time quantitative central limit theorem.
	\subsection{Outline of proof and difficulties}
	The proof in this paper consists of two main steps: Step 1 involves establishing tightness, and Step 2 focuses on identifying the tight limit and proving the pathwise uniqueness of the fluctuation SPDE \eqref{LimitSpde}.

In the first step, we start by obtaining the necessary uniform estimates for the fluctuation process $(\eta^N)_{N\in \mathbb{N}}$ and  the interacting terms $\nabla\cdot [K \ast \mu_N (t) \mu_N (t) - v_t K
	\ast v_t] , $ $\mathcal{K}^N(\varphi)$  defined in \eqref{interact term} below,  using the relative entropy method. The main challenge in this step is to establish the tightness of the laws of $\{\eta^N\}$ based on these estimates.  On one hand, the uniform bound for the relative entropy $\sup_{t \in [0, T]}H(F_t^N|\bar{F}^N_t)(\omega)$  depends on $\omega \in \Omega,$ due to the singularity of the Biot–Savart kernel and the presence of environmental noise $\{W_t, t\in[0,T]\}.$ This prevents us from obtaining the tightness of the laws of $\{\eta^N\}$  directly. However, when $K$ is bounded or there is no common noise, the uniform relative entropy  $\sup_{t \in [0, T]}H(F_t^N|\bar{F}^N_t)$ can be bounded by a deterministic constant. On the other hand, compared to the case without common noise in  \cite{wang2023gaussian}, we must  handle the uniform estimate for the H\"older seminorm of the new transport noise term $\sigma\cdot \nabla \eta^N_t \mathd W_t.$ To address these challenges,    based on \cite{wang2023gaussian},
	 a classical localization method is applied. Specifically, we study tightness through introducing two  sequence of stopping times  to control the mean field limit $v$ and the fluctuation measures $\eta^N.$ 

In the second step, we identify the tight limit $\tilde{\eta}$  in  Proposition \ref{pro:skorokhod} and  establish the existence of solutions to the fluctaution SPDE \eqref{LimitSpde}. The main difficulty in this step is identifying  the additive noise $\{\mathcal{M}_t,t\in [0,T]\}$ in the fluctution SPDE  \eqref{LimitSpde}  through studying the conditional law of $\{\mathcal{M}_t,t\in [0,T]\}$ with respect to the environmental noise $\{W_t, t\in[0,T]\}.$ To overcome this challenge, we establish a strong convergence ( see Proposition \ref{prop:strong-convergence} below) from the stochastic point vortex system \eqref{eqt:vortex} to the conditional Mckean-Vlasov equation \eqref{eqt:ncopy},   and then  derive  the conditional law of the additive noise $\{\mathcal{M}_t,t\in [0,T]\}$,  following the idea in \cite{kurtz2004stochastic}. Pathwise uniqueness for the fluctuation SPDE \eqref{LimitSpde} is then proven using standard SPDE arguments.

\subsubsection*{\bf{Organization of the paper}}

 The paper is structured as follows. Section \ref{sec:preliminaries} provides an introduction to key definitions related to the mean field limit equation \eqref{eqt:mean}, the fluctuation SPDE \eqref{LimitSpde}, and some auxiliary results that will be used in the subsequent proofs. The main sections of the paper are  Section  \ref{sec:tightness} and Section \ref{sec:wellposed}. In Section \ref{sec:tightness}, we study the tightness of laws of $\{\eta^N\}_{N\in \mathbb{N}}$ in the space $\mathcal{X}$ (defined in Definition \eqref{def:H}) and  $\{\mathcal{M}^N\}_{N\in\mathbb{N}}$  in the space $\mathcal{Y}$ (defined in Definition \eqref{def:H}).  In  Section \ref{sec:wellposed}, we establish the well-posedness of the fluctuation SPDE \eqref{LimitSpde} and   convergence of the fluctuation process $\{\eta^N\}_{N\in \mathbb{N}}$ to  the unique probabilistically  strong solution $\eta$  to the fluctuation SPDE \eqref{LimitSpde} in the sense of distribution.

	At the end of this section,  we introduce the basic notation used throughout the paper.

	\begin{enumerate}
		\item Bracket notations: The bracket $\<\cdot, \cdot\>$  denotes integrals when the space and underlying measure are clear from the context. We use a similar bracket $[\cdot ,\cdot]_t$ to denote quadratic variations between local martingales at time $t$. 
		\item  Filtration notations: The notation $\mathcal{F}_1\vee \mathcal{F}_2$ stands for the $\sigma$-algebra generated by $\mathcal{F}_1\cup \mathcal{F}_2$. Given a stochastic process $X_t, t\in [0,T],$ we use $(\mathcal{F}_t^X)_{t\in [0,T]}$ to denote the normal filtration 	generated by $X.$ 
		In particular, we  use $(\mathcal{F}_t^W)_{t\in [0,T]}$ to denote the normal filtration generated by 1-dimensional Brownian motions $\{W_t, t\in [0,T]\}.$ We also use $(\mathcal{F}_t)_{t\in [0,T]}$ to denote the normal filtration generated by
		1-dimensional Brownian motion $W$ and 2-dimensional Brownian motions $(B_i)_{i\geq 1}$ and $( X_i(0))_{i\geq 1}.$
		\item  Distribution notations: Given a  Polish space $E$ and probability space $(\Omega,\mathcal{F},\mathbb{P}),$ we say $Q(\mathd x,\omega)$ is a random meausre on $E,$ if $Q(\mathd x,\omega)$ is a function of two variables $\omega\in \Omega$ and $A\in\mathbb{B}(E),$ satisfies that there exists a null set $N\in \mathcal{F}$ such that $Q(\mathd x,\omega)$ is a measure in $A$ for fixed $\omega\in N^c.$   Given $\sigma$-algebra $\mathcal{F}$, we use $\mathcal{L}(X),$ $\mathcal{L}(X|\mathcal{F})$ to denote the distribution of $X$ and conditional distribution of $X$ with respect to  $\mathcal{F}.$  In particular, for convention, we may denote the distribution by its density function when the distribution has a density function. Given a symmetric probability measure $\rho^N$ on $E^N$ where $E$ is a Polish space, the $k$-marginal $\rho^{N,k}$ is a probability measure on $E^k$ defined by $\int_{E^{N-k}}\rho^N(\mathd x_1\cdots \mathd x_{N}),$ where $k\leq N.$ Finally, we use $\mathcal{P}(E)$ to denote the space of probability measure on $E.$ 
		\item  Independence and Conditional independence: Given three $\sigma$-algebras $\mathcal{F}_1$, $\mathcal{F}_2$, and $\mathcal{F}_3,$  we use $\mathcal{F}_1\perp\mathcal{F}_2|\mathcal{F}_3$ to  indicate that $\mathcal{F}_1$ and $\mathcal{F}_2$ are conditionally independent given $\mathcal{F}_3.$ $\mathcal{F}_1\perp\mathcal{F}_2$ indicates that $\mathcal{F}_1$ and $\mathcal{F}_2$ are mutually independent. 
		\item  Product space and Product function:  Given two measure spaces $(\Omega_1,\mathcal{A}_1,\mu_1)$ and $(\Omega_2,\mathcal{A}_2,\mu_2),$  we denote their product space as $\Omega_1 \times \Omega_2,$ the product $\sigma$-algebra  as $\mathcal{A}_1 \times \mathcal{A}_2,$  and the product measure  as $\mu_1 \times \mu_2.$ For $k\in \mathbb{N},$ we use $(\Omega_1^{\otimes k},\mathcal{A}_1^{\otimes k},\mu_1^{\otimes k})$ to denote the $k$-product measure space for  the measure space $(\Omega_1,\mathcal{A}_1,\mu_1).$ Given a function $f(x),x\in E$ on Polish space $E$ and $k\in \mathbb{N},$ the $k$-tensorized function $f^{\otimes k}(x^k)$ is defined by $\prod_{i=1}^{k}f(x_i),$ where $x^k=(x_1,\cdots,x_k)\in E^k.$
		\item  We will mostly  work on  Sobolev spaces.  The norm of Sobolev space $H^{\alpha}(\mathbb{T}^d)$, $\alpha\in \mathbb{R}$, is defined by
		\begin{equation*}
			\|f\|_{H^{\alpha}}^2:=\sum_{k\in\mathbb{Z}^d}(1+|k|^2)^{\alpha}|\langle f,e_k \rangle|^2,
		\end{equation*}
	 where $e_k\assign e^{\sqrt{-1}k\cdot x}, k\in \mathbb{Z}^d.$
		We also use some results on Besov spaces $B^{\alpha}_{p,q}$ in this paper and  provide a brief introduction about Besov spaces in Section \ref{besov space}.  
	\end{enumerate}

Finally, throughout this paper, we use $C$ to denote universal constants, and we indicate relevant dependencies using subscripts when necessary. We use the notation $a\lesssim b$ if there exists a universal constant
$C > 0$ such that $a\leq Cb.$

	\section{Preliminaries}\label{sec:preliminaries}
	In this section, we introduce the definitions of solutions and collect some auxiliary results.
\subsection{Definitions of solutions}\label{ssec:def}
In this subsection,  we   present several distinct definitions for solutions to the stochastic 2-dimensional Navier-Stokes equation  \eqref{eqt:mean} and the fluctuation SPDE \eqref{LimitSpde}. The  following definitions for solutions to the stochastic 2-dimensional Navier-Stokes equation  \eqref{eqt:mean} are consistent with those in \cite{shao2024quantitative} and the well-posedness of \eqref{eqt:mean} has also been established in \cite{shao2024quantitative}.

 \begin{definition}\label{def:mean}
 	A probabilistically weak solution $(\Omega,\mathcal{F},(\mathcal{G}_t)_{t\in [0,T]},\mathbb{P},W,(  v_t)_{t\in [0,T]})$   to \eqref{eqt:mean} with initial value $  v_0\in H^{3}(\mathbb{T}^2)$ is defined as  a stochastic basis $(\Omega,\mathcal{F},(\mathcal{G}_t)_{t\in [0,T]},\mathbb{P})$ supporting   standard $(\mathcal{G}_t)_{t\in [0,T]}$-Brownian motion $\{W_t,t\in [0,T]\}$ (denoted by $W$) and a continuous $L^2(\mathbb{T}^2)$-valued $\mathcal{G}_t$-adapted stochastic process $(  v_t)_{t\in [0,T]}$  such that 
 	\begin{enumerate}
 		\item For all $t\in [0,T]$, \begin{align}\label{regu-1}
 			\underset{x\in \mathbb{T}^2 }{ess\sup} v_t\leq \underset{x\in \mathbb{T}^2 }{ess\sup}v_0,    \quad\underset{x\in \mathbb{T}^2 }{ess\inf}   v_t\geq  \underset{x\in \mathbb{T}^2 }{ess\inf}   v_0,  \quad\|  v_t\|_{L^2}\leq \|  v_0\|_{L^2},\quad \mathbb{P}-a.s..
 		\end{align}  
 		\item  It holds that 
 		\begin{align}\label{regu-2}
 			&\mathbb{E} \int_0^T \|  v_t\|_{H^4}^2\mathd t <\infty , \quad \mathbb{E}\big[\sup_{t\in [0,T ]}\|  v_t\|_{H^2}^2\big]< \infty.
 		\end{align}
 		\item   For all $\varphi\in C^{\infty}(\mathbb{T}^2)$,  it holds almost surely that for all $t\in [0,T],$ 
 		\begin{align*}
 			\left\langle  \varphi,  v_t\right\rangle  =  \left\langle \varphi,   v_0\right\rangle +\int_0^t \left\langle \Delta\varphi,   v_s\right\rangle \mathd s +\int_0^t\left\langle  \nabla \varphi,  K*  v_s   v_s\right\rangle \mathd s+ \int_0^t  \left\langle \nabla \varphi,    v_s\sigma\right\rangle\circ \mathd W_s .
 		\end{align*}  
 	\end{enumerate}
 	
 \end{definition}
 \begin{definition}\label{def:mean-strong}
 	Given   a independent 1-dimensional Brownian motion $\{W_t,t\in [0,T]\}$  on a probability space $(\Omega,\mathcal{F},\mathbb{P}),$  we say that 
 	$(  v_t)_{t\in [0,T]}$  is a probabilistically strong solution  to \eqref{eqt:mean} with initial value $  v_0\in H^{3}(\mathbb{T}^2)$   if $(\Omega,\mathcal{F},(\mathcal{F}^{W}_t)_{t\in [0,T]},\mathbb{P},(  v_t)_{t\in [0,T]})$  is a  probabilistically weak solution to \eqref{eqt:mean} with initial value $  v_0\in H^{3}(\mathbb{T}^2),$ where $(\mathcal{F}^{W}_t)_{t\in [0,T]}$ is the normal filtration generated by Brownian motions $\{W_t,t\in [0,T]\}.$ 
 \end{definition}
 \begin{definition}\label{def:uni-mean}
 	We say that pathwise uniqueness holds for \eqref{eqt:mean} if for any two probabilistically weak solutions $(  v_t)_{t\in [0,T]}$ and $(\tilde{  v}_t)_{t\in [0,T]}$  on the same stochastic basis $(\Omega,\mathcal{F},(\mathcal{G}_t)_{t\in [0,T]},\mathbb{P}),$ with the same noise $\{W_t,t\in [0,T]\}$ and the same initial data $v_0,$ it satisfies $$\mathbb{P}\(\|  v_t-\tilde{  v}_t\|_{L^2}=0, \forall t\in [0,T] \)=1.$$
 \end{definition}

Before introducing the definitions of solutions to \eqref{LimitSpde}, we first define a   Polish space in which the solution exists. This space is given by
\begin{align}\label{def:d}
	\mathcal{D}\assign	\mathcal{X}\times\mathcal{Y}\times\mathcal{W}
\end{align}  equipped with the metric  $d_{\mathcal{D}}(f,g)\assign (\sum_{i=\mathcal{X},\mathcal{Y},\mathcal{W}}d^2_{i}(f,g))^{\frac{1}{2}},$ 
where \begin{align*}
			\mathcal{X} &\assign  \bigcap_{k\in \mathbb{N}}  \left[ C ([0, T];
	H^{-3-\frac{1}{k}}(\mathbb{T}^2)) \cap L^2([0,T];H^{-1-\frac{1}{k}}(\mathbb{T}^2))\right] ,
	\\	\mathcal{Y}&\assign \bigcap_{k\in \mathbb{N}}C ([0, T];H^{-2-\frac{1}{k}}(\mathbb{T}^2)),
	\\\mathcal{W}&\assign C([0,T];\mathbb{R}),
\end{align*}
endowed with the metrics
\begin{align*}
	d_{	\mathcal{X} }(f,g)&\assign\sum_{k=1}^{\infty}2^{-k}\bigg(1\wedge\(\|f-g\|_{C ([0, T];
		H^{-3-\frac{1}{k}})}+\|f-g\|_{L^2([0,T];H^{-1-\frac{1}{k}})}\)\bigg),
	\\d_{	\mathcal{Y} }(f,g)&\assign\sum_{k=1}^{\infty}2^{-k}\big(1\wedge\|f-g\|_{C ([0, T];
		H^{-2-\frac{1}{k}})}\big).
\end{align*}
Similarly, we  define  Polish space  \begin{align}\label{def:H}
	\mathcal{H}\assign\mathcal{V}\times	\mathcal{X}\times\mathcal{Y}\times\mathcal{W}
\end{align}  equipped with the metric  $d_{\mathcal{H}}(f,g)\assign (\sum_{i=\mathcal{V},\mathcal{X},\mathcal{Y},\mathcal{W}}d^2_{i}(f,g))^{\frac{1}{2}},$  where
		$$\mathcal{V}\assign C([0,T];L^2(\mathbb{T}^2))\cap L^2([0,T];H^4(\mathbb{T}^2))$$
 endowed with the metric $	d_{	\mathcal{V} }(f,g)\assign\|f-g\|_{C ([0, T];
	L^2)}+\|f-g\|_{L^2([0,T];H^{4})}.$
We then give the definitions about the solution to the fluctuation SPDE \eqref{LimitSpde}.
\begin{definition}\label{def:limit spde}
	A probabilistically weak solution $\bigg(\Omega,\mathcal{F},(\mathcal{G}_t)_{t\in [0,T]},\mathbb{P},\(  \eta_t,\mathcal{M}_t,W_t\)_{t\in [0,T]}\bigg)$   to the SPDE \eqref{LimitSpde}  is defined as  a stochastic basis $(\Omega,\mathcal{F},(\mathcal{G}_t)_{t\in [0,T]},\mathbb{P})$ supporting the stochastic process $$\(  \eta_t,\mathcal{M}_t,W_t\)_{t\in [0,T]}$$ valued in $\mathcal{D},$ 
	 	\begin{enumerate}
		\item $W$ is $(\mathcal{G}_t)_{t\in [0,T]}$-1-dimensional Brownian motion.
		\item  $(\mathcal{M}_t)_{t\in [0,T]}$ is a $\mathcal{G}_t$-adapted process belonging to  $C([0,T];H^{- \alpha}(\mathbb{T}^2))$ $\mathbb{P}$-a.s., for every $\alpha > 2,$ and satisfying for every $\varphi\in C^{\infty}(\mathbb{T}^2), 0\leq t<t+r\leq T,$ 
		\begin{align*}
			\begin{split}
			\mathbb{E}\[\exp{i\left\langle \varphi,\mathcal{M}_t \right\rangle}\mid \mathcal{F}^{W}_T\]&=\exp\bigg\{-\int_{0}^{t}\langle \mid \nabla\varphi\mid^2,v_s\rangle\mathd s\bigg\},
			\\ \mathbb{E}\[\exp{i\left\langle \varphi,(\mathcal{M}_{t+r}-\mathcal{M}_{t}) \right\rangle}\mid \mathcal{F}^{W}_T\vee\mathcal{F}^{\mathcal{M}}_t\]&=\exp\bigg\{-\int_{t}^{t+r}\langle \mid \nabla\varphi\mid^2,v_s\rangle\mathd s\bigg\},
		\end{split}
		\end{align*}
		\item $(\eta_t)_{t\in [0,T]}$ is a continuous $H^{- \alpha - 2}(\mathbb{T}^2)$-valued $\mathcal{G}_t$-adapted stochastic process satisfying $\eta\in L^2 ([0, T], H^{-
			\alpha}(\mathbb{T}^2))$ $\mathbb{P}$-a.s.,   for every $\alpha > 1.$
		\item For all $\varphi\in C^{\infty}(\mathbb{T}^2)$,  it holds almost surely that for all $t\in [0,T],$ 
			\begin{align}\label{eqt-limitspde}
			\langle \eta_t, \varphi \rangle =& \langle \eta_0, \varphi \rangle +
			\int^t_0 \langle  \Delta \varphi, \eta_s \rangle \mathd \nocomma s +
			\int^t_0 \langle \nabla \varphi, v_s K \ast \eta_s \rangle \mathd
			\nocomma s + \int^t_0 \langle \nabla \varphi, \eta_s K \ast v_s
			\rangle \mathd \nocomma s \nonumber\\
			&+\langle\mathcal{M}_t ,\varphi\rangle+\frac{1}{2}\int_0^t  \left\langle\sigma\cdot\nabla\(\sigma\cdot\nabla\varphi\),\eta_s\right\rangle \mathd s+\int_{0}^{t}\left\langle \sigma\cdot \nabla\varphi,\eta_s\right\rangle \mathd W_s,
		\end{align}
	where $(  v_t)_{t\in [0,T]}$  is the unique probabilistically strong solution to the mean field equation \eqref{eqt:mean} in the sense of Definitions \ref{def:mean-strong} and \ref{def:uni-mean}.
	\end{enumerate}   
\end{definition}
\begin{remark}\label{remark-m}
Condition (2) in Definition \ref{def:limit spde} specifies the conditional distribution of  $\mathcal{M}$ with respect to the environmental noise $W,$  which uniquely determines the joint distribution of $\mathcal{M}$ and $W.$ Notably, this condition also shows that given the environmental noise information $\mathcal{F}_T^W,$ the distribution of $\mathcal{M}_t$
is similar to a Gaussian distribution.
\end{remark}
\begin{definition} \label{def:stronglimit}
		Given  a 1-dimensional Brownian motion $\{W_t,t\in[0,T]\}$  and a stochastic process $\{\mathcal{M}_t,t\in [0,T]\}$ with values in $\bigcap_{k\in \mathbb{N}}C ([0, T];H^{-2-\frac{1}{k}}(\mathbb{T}^2))$ on a probability space $(\Omega,\mathcal{F},\mathbb{P})$ satisfying for every $\varphi\in C^{\infty}(\mathbb{T}^2), 0\leq t<t+r\leq T,$ 
		\begin{align*}
				\begin{split}
				\mathbb{E}\[\exp{i\left\langle \varphi,\mathcal{M}_t \right\rangle}\mid \mathcal{F}^{W}_T\]&=\exp\bigg\{-\int_{0}^{t}\langle \mid \nabla\varphi\mid^2,v_s\rangle\mathd s\bigg\},
				\\ \mathbb{E}\[\exp{i\left\langle \varphi,(\mathcal{M}_{t+r}-\mathcal{M}_{t}) \right\rangle}\mid \mathcal{F}^{W}_T\vee\mathcal{F}^{\mathcal{M}}_t\]&=\exp\bigg\{-\int_{t}^{t+r}\langle \mid \nabla\varphi\mid^2,v_s\rangle\mathd s\bigg\},
			\end{split}
		\end{align*}   we say that 
	$(\eta_t)_{t\in [0,T]}$ is a probabilistically strong solution  to \eqref{LimitSpde}     if $$\bigg(\Omega,\mathcal{F},(\mathcal{F}^{W,\mathcal{M}}_t)_{t\in [0,T]},\mathbb{P},\(  \eta_t,\mathcal{M}_t,W_t\)_{t\in [0,T]}\bigg)$$    is a  probabilistically weak solution to \eqref{LimitSpde} where $(\mathcal{F}^{W,\mathcal{M}}_t)_{t\in [0,T]}$ is the normal filtration generated by $\{W_t,t\in[0,T]\},$    $\{\mathcal{M}_t,t\in [0,T]\},$ and the initial value $\eta_0.$  
\end{definition}
\begin{definition}\label{def:uni}
	We say that pathwise uniqueness holds for \eqref{LimitSpde} if for any two probabilistically weak solutions $(\eta_t)_{t\in [0,T]}$ and $(\tilde{  \eta}_t)_{t\in [0,T]}$  on the same stochastic basis $(\Omega,\mathcal{F},(\mathcal{G}_t)_{t\in [0,T]},\mathbb{P}),$ with the same noise $\{W_t,t\in[0,T]\}, \{\mathcal{M}_t,t\in [0,T]\}$   and the same initial data $\eta_0 \in \cap_{k\in\mathbb{N}} H^{-1-\frac{1}{k}}(\mathbb{T}^2),$ it satisfies that for every $4>\alpha>3,$	$$\mathbb{P}(\sup_{t \in [0, T]}\|\eta_t-\tilde{\eta}_t\|_{H^{-\alpha}}^2=0)=1.$$
\end{definition}
\subsection{Relative entropy and stochastic point vortex model}
In this section, we collect the  auxiliary results from \cite{jabin2018quantitative}, \cite{shao2024quantitative}  and \cite{fournier2014propagation} for convenience.
We start this section by recalling the definition of relative entropy  associated to any two probability measures $\rho$ and $\eta$ on Polish space $E$.  

The {\it relative entropy} $H(\rho|\eta )$  is defined as 
\begin{align*} 
	H(\rho|\eta)\assign 
	\begin{cases}
		\displaystyle{	\int_{E} \log\frac{d\rho}{d\eta}\mathd \rho }   &\rho\ll\eta; \vspace{5pt}\\
		+\infty, &otherwise.
	\end{cases} 
\end{align*}  
Here $\frac{d\rho}{d\eta}$ represents the Radon–Nikodym derivative of $\rho$ with respect to $\eta$.  

 The following lemma is used to derive the uniform estimates for the stochastic point vortex model \eqref{eqt:vortex}, i.e.,  Lemma \ref{lemma uni mu}, in Section \ref{sec:tightness}.
\begin{lemma}[{\cite[Lemma 1]{jabin2018quantitative}}]\label{lemma jw0} 
	For  any  two probability densities $\rho_N,\bar{\rho}_N$ on $\mathbb{T}^{2N},N\geq 1$ and any function $\phi \in L^{\infty} \big(\mathbb{T}^{2N} \big) $, one has that  for any constant $b>0,$
	$$\int_{\mathbb{T}^{2N}}\phi \rho_N \mathd x^{N}\leq \frac{1}{b N} \bigg(H(\rho_N|\bar{\rho}_N)+\log \int_{\mathbb{T}^{2N}}\bar{\rho}_N\exp\{bN\phi\}\mathd x^{N}\bigg).$$
\end{lemma}
The following two results from \cite{shao2024quantitative} establish the well-posedness of the stochastic 2-dimensional Navier-Stokes equation \eqref{eqt:mean} and provide a quantitative conditional propagation of chaos result for the stochastic point vortex model \eqref{eqt:vortex}.
\begin{lemma}[{\cite[Theorem 3.1]{shao2024quantitative}}]\label{thm:spde}
	Given  1-dimensional standard Brownian motion  $\{W_{t},t\in [0,T]\}$ on  probability space $(\Omega,\mathcal{F},\mathbb{P}),$
	for each $  v_0\in H^{3}(\mathbb{T}^2)$,	there exists a unique probabilistically strong solution $(  v_t)_{t\in [0,T]}$ to \eqref{eqt:mean} in the sense of Definitions \ref{def:mean-strong} and \ref{def:uni-mean}. 
\end{lemma}
\begin{lemma}[{\cite[Theorem 1.1]{shao2024quantitative}}]\label{thm:entropy}
	Assume that the probability measure $v_0=\mathcal{L}(X(0)) $  on $\mathbb{T}^2$ has a  density $v_0\in H^{3}(\mathbb{T}^2)$ and  $\underset{x\in \mathbb{T}^2 }{\inf}v_0>0.$ Then,
	it holds that
	\begin{align*}
		H(F_t^N|\bar{F}^N_t)\leq \exp\bigg(C_{0}\int_{0}^{t}(\|  v_s\|_{H^4}^2 +1)\mathd s\bigg)\quad\forall t\in [0,T]\quad\mathbb{P}-a.s.,
	\end{align*}
	where $C_{0}$ is a positive deterministic constant depending on $\|v_0\|_{L^2(\mathbb{T}^2)}$ and $\underset{x\in \mathbb{T}^2 }{\inf}v_0,$  $F_t^N$ and $\bar{F}^N_t$ are random measures on torus $\mathbb{T}^{2N}$ defined in \eqref{def:F} and \eqref{def:barF}, and $(  v_t)_{t\in [0,T]}$ is the unique probabilistically strong solution to the stochastic 2-dimensional Navier-Stokes  equation \eqref{eqt:mean} with the initial data $  v_0$  in the sense of Definition \ref{def:mean-strong} and \ref{def:uni-mean}. 
\end{lemma}
The following lemma will be applied in the proof of Proposition \ref{prop:strong-convergence}, in which  the strong convergence between the  stochastic point vortex model \eqref{eqt:vortex} and the conditional Mckean-Vlasov equation \eqref{eqt:ncopy} is established.
\begin{lemma}[{\cite[Lemma 3.3]{fournier2014propagation}}]\label{fournier3.3}
	For any $r\in(0,2)$ and $\beta>\frac{r}{2},$ there exists a constant $C_{r,\beta}>0$ depends on $r,\beta$ such that for any probability measure $\rho$ on $\mathbb{T}^2\times \mathbb{T}^2$
	with finite Fisher information $I(\rho)$,
	\begin{align*}\int_{\mathbb{T}^{2}}\frac{1}{|x_1-x_2|^{r}}\rho (\mathd x_1\mathd x_2)\leq C_{r,\beta}(I^{\beta}(\rho)+1).
	\end{align*}
Here the Fisher information $I(\rho)$ on $\mathbb{T}^2$ is defined by $\int_{\mathbb{T}^{2}}\frac{|\nabla\rho|^2}{\rho}\mathd x.$ 
\end{lemma}
\subsection{Besov space}\label{besov space}
In this subsection, we collect useful results related to Besov spaces.  We use $(\Delta_{i})_{i\geq -1}$ to denote the Littlewood-Paley blocks for a dyadic partition of unity. Besov spaces  $B^{\alpha}_{p,q}$ 
on the torus with  $\alpha\in \mathbb{R}$ and $1\leq p,q\leq \infty$,  are defined as the completion of $C^{\infty}$ with respect to the norm
\begin{equation*}
	\|f\|_{B^{\alpha}_{p,q}}\assign \( \sum_{n\geq -1}\left(2^{n\alpha q }\|\Delta_{n}f\|_{L^p}^q \right) \) ^{\frac{1}{q}}.
\end{equation*}

We remark that  $B^{\alpha}_{2,2}$ coincides with the Sobolev space $H^{\alpha}$, $\alpha\in \mathbb{R}$. We say $f\in C^\alpha$, $\alpha\in \mathbb{N}$, if $f$ is $\alpha$-times differentiable. For $\alpha \in \mathbb{R}\setminus \mathbb{N} $, we set $C^{\alpha}=B^{\alpha}_{\infty,\infty}$. We will often write $\|\cdot\|_{C^{\alpha}}$ instead of $\|\cdot\|_{B^{\alpha}_{\infty,\infty}}$. In the case $\alpha\in \mathbb{R}^+\setminus \mathbb{N}$, $C^{\alpha}$ coincides with the usual H\"older space. We use $C^{\infty}$ to denote the space of infinitely differentiable functions on $\mathbb{T}^2$.

We quote the following results about  Besov spaces.
\begin{lemma}[{\cite[Proposition 4.6]{triebel2006theory}}]\label{lemma embedding}Let $\alpha\in \mathbb{R}$, $\beta\in \mathbb{R}$ and $p_1,p_2,q_1,q_2\in [1,\infty]$. Then the embedding
	\begin{equation*}
		B^{\alpha}_{p_1,q_2}\hookrightarrow B^{\beta}_{p_2,q_2}
	\end{equation*}
	is compact if and only if,
	\begin{equation*}
		\alpha-\beta>d\(\frac{1}{p_1}-\frac{1}{p_2}\)_{+}.
	\end{equation*}
\end{lemma}
\begin{lemma}
	\label{lemma triebel}
	(i) Let $\alpha,
	\beta \in \mathbb{R}$ and $p, p_1, p_2, q \in [1, \infty]$ be such that
	$\frac{1}{p} = \frac{1}{p_1} + \frac{1}{p_2}$. The bilinear map $(u, v)
	\mapsto u \nocomma v$ extends to a continuous map from $B_{p_1, q}^{\alpha}
	\times B_{p_2, q}^{\beta}$ to $B_{p \comma q}^{\alpha \wedge \beta}$ if
	$\alpha + \beta > 0$ (cf. {\cite[Corollary 2]{mourrat2017global}}).
	
	(ii) (Duality.) Let $\alpha\in (0,1)$, $p,q\in[1,\infty]$, $p'$ and $q'$ be their conjugate exponents, respectively. Then the mapping  $(u, v)\mapsto \<u,v\>=\int uv \dif x$  extends to a continuous bilinear form on $B^\alpha_{p,q}\times B^{-\alpha}_{p',q'}$, and one has $|\<u,v\>|\lesssim \|u\|_{B^\alpha_{p,q}}\|v\|_{B^{-\alpha}_{p',q'}}$ (cf.  \cite[Proposition~7]{mourrat2017global}).
\end{lemma}

\begin{lemma}[{{\cite[Corollary
			2.86]{chemin2011fourier}} }]\label{lemma:infity} For any positive real number $\alpha$ and any $p,q\in [1,\infty]$, it holds that
	\begin{align*}
		\|fg\|_{B^{\alpha}_{p,q}}\lesssim\|f\|_{L^{\infty}}\|g\|_{B^{\alpha}_{p,q}}+\|f\|_{B^{\alpha}_{p,q}}\|g\|_{L^{\infty}},
	\end{align*}
	with the proportional constant independent of $f$ and $g$.
\end{lemma}

\begin{lemma}[{\cite[Theorem 2.1 and 2.2]{kuhn2021convolution}}]
	\label{lemma convolution}Let $ \alpha,\beta \in \mathbb{R}, q ,q_1,q_2\in(0, \infty]$ and $p, p_{1}, p_{2} \in[1, \infty]$
	be such that
	$$
	1+\frac{1}{p}=\frac{1}{p_{1}}+\frac{1}{p_{2}},\quad \frac{1}{q}\leq\frac{1}{q_1}+\frac{1}{q_2}.
	$$
	\begin{enumerate}
		\item 	If $f \in B_{p_{1}, q}^{\alpha}$ and $g \in L^{p_{2}},$ then $f * g \in B_{p, q}^{\alpha}$ and
		$$
		\|f * g\|_{B_{p, q}^{\alpha}}\lesssim\| f\|_{B_{p_{1}, q}^{\alpha}} \cdot\left\|g \right\|_{L^{p_{2}}},
		$$

		\item	If $f\in B^{\alpha}_{p_1,q_1}$ and $g\in B^{\beta}_{p_2,q_2}$,
		then $f*g\in B^{\alpha+\beta}_{p,q}$ and
		\begin{align*}
			\|f * g\|_{B_{p, q}^{\alpha+\beta}}\lesssim\| f\|_{B_{p_{1}, q_1}^{\alpha}} \cdot\left\|g \right\|_{B^{\beta}_{p_2,q_2}},
		\end{align*}
		with the proportional constant independent of $f$ and $g$.
	\end{enumerate}

\end{lemma}
						
\section{Uniform estimates}\label{sec:tightness}	
The goal of this section is to prove the tightness of  laws of the fluctuation measures $(\eta^N)_{N\in\mathbb{N}}$  and the tightness of laws of the additive noise $(\mathcal{M}^N)_{N\in\mathbb{N}}.$ Recall that for $t\in[0,T],$ $\eta^N_t=\sqrt{N}(	\mu_N (t) - v_t)$ and $\mu_N(t) = \frac{1}{N} \sum_{i=1}^N \delta_{X_i(t)}.$ The additive noise term $(\mathcal{M}^N)_{N\in\mathbb{N}}$ given in Lemma  \ref{lemma m} below, satisfies for all $t
\in [0, T]$ and $\varphi\in C^\infty,$ 
	$\left\langle \mathcal{M}_t^N, \varphi \right\rangle	 = \frac{\sqrt{2}}{\sqrt{N}}  \sum_{i = 1}^N \int^t_0 \nabla \varphi (X_{i}) \cdot
	\mathd B_s^i,\quad \mathbb{P}$-a.s..
 
 Applying Itô's formula to the interacting particle system \eqref{eqt:vortex}, we derive  the following SPDE representation for the fluctuation measures $(\eta^N_{t})_{t\in[0,T]}$ i.e., for every $\varphi\in C^{\infty}(\mathbb{T}^2),$
\begin{align}\label{spde-etaN}
	\left\langle \varphi, \eta_t^N \right\rangle&=	\left\langle \varphi, \eta_0^N \right\rangle+\int_0^t \left\langle \Delta\varphi, \eta_s^N\right\rangle \mathd s
	+\frac{1}{2}\int_0^t  \left\langle\sigma\cdot\nabla\(\sigma\cdot\nabla\varphi\),\eta_s^N\right\rangle \mathd s\nonumber
	\\&+\sqrt{N}\int_{0}^{t}\left\langle  \nabla \varphi,  \mu_N(s)K*\mu_N(s)  \right\rangle \mathd s-\sqrt{N}\int_{0}^{t}t\left\langle  \nabla \varphi,  v_sK*v_s  \right\rangle \mathd s\nonumber
	\\&+\sqrt{2}\int_{0}^{t}\frac{1}{\sqrt{N}}\sum_{i=1}^{N}\nabla \varphi(X^i_s)\cdot\mathd B_i(s)+\int_{0}^{t}\left\langle \sigma\cdot \nabla\varphi,\eta_s^N\right\rangle \mathd W_s.
\end{align}
For simplicity, we define the following interacting terms $	\mathcal{K}_t^N :C^{\infty}(\mathbb{T}^2)\rightarrow \mathbb{R}$
\begin{align}\label{interact term}
	\mathcal{K}_t^N (\varphi)  \assign  & \sqrt{N}  \langle \nabla \varphi, K \ast
	\mu_N (t) \mu_N (t) \rangle - \sqrt{N}  \langle \nabla \varphi, v_t K
	\ast v_t  \rangle . 
\end{align}	
To establish the tightness of the fluctuation measures $\{\eta^N\}_{N\in \mathbb{N}}$ and the  additive noise term $\{\mathcal{M}^N\}_{N\in \mathbb{N}}$, we  first obtain the uniform estimates for $\eta_t^N,\nabla\cdot [K \ast \mu_N (t) \mu_N (t) - v_t K
\ast v_t],\mathcal{K}_t^N(\varphi)$ in Section \ref{ssec:relative}   following the estimates in \cite{wang2023gaussian}.  We also  establish an additional  estimate Lemma \ref{lemma: eta unit} in Section \ref{sec:compactness}, by exploiting the structure of \eqref{spde-etaN}.  The proof of tightness is more complicated than the case without environmental noise $W,$ since we have to deal with the new transport noise term $\sigma\cdot \nabla \eta^N_t \mathd W_t$  and the uniform bound for the relative entropy $\sup_{t \in [0, T]}H(F_t^N|\bar{F}^N_t)(\omega)$  in  Lemma \ref{thm:entropy}  depends on $\omega \in \Omega.$  To address these challenges, we use a classical localization argument.
\subsection{Estimates on the relative entropy }\label{ssec:relative}
In this section, we apply the relative entropy method to obtain uniform estimates essential for the subsequent proof of tightness. 
 
 The core idea behind the relative entropy method is to employ the Donsker-Varadhan variational formula, which gives Lemma \ref{lemma jw0}, to decompose the target integral into two terms. One term is the relative entropy, which is bounded almost surely ( i.e., Lemma \ref{thm:entropy}), while the other is an exponential-type integral that can be controlled using  estimates {\cite[Theorem 4]{jabin2018quantitative}} and  {\cite[Lemma 2.3]{wang2023gaussian}}. 
 Compared to  \cite{wang2023gaussian},  in the environmental noise case, the components $\{\bar{X}_i,i\in \mathbb{N}\}$ of  the limiting nonlinear SDE \eqref{eqt:ncopy} are no longer independent but conditionally independent and identically distributed, i.e., for $i\neq j$
\begin{align*}
	\mathcal{L}(\bar{X}_i(t)| \mathcal{F}^{W}_T)(\mathd x)=\mathcal{L}(\bar{X}_j(t)| \mathcal{F}^{W}_T)(\mathd x)=v_t(\mathd x),\quad \mathbb{P}-a.s. 
\end{align*}
in the sense of $\mathcal{P}(\mathbb{T}^2)$ and
\begin{align*}
	\bar{X}_i(t)\perp\bar{X}_j(t)|\mathcal{F}^{W}_T.
\end{align*} At this time,  two target measures in Lemma \ref{lemma jw0} are considered as $F^N(t)(\mathd x^N)$ and $\bar{F}_N(t)(\mathd x^N)$ defined in \eqref{def:F} and \eqref{def:barF}.
\begin{lemma}
		\label{lemma uni mu}For each $\alpha > 1$, there exist  constants
	$C_{\alpha}$ and $C$ such that  for all $N\in \mathbb{N},$
	\begin{enumerate}
		\item \begin{align*}
			&\mathbb{E}\[ \| \mu_N (t) - v_t \|^2_{H^{- \alpha}}\mid \mathcal{F}_T^{W}\] \\\leqslant&
			\frac{C_{\alpha}}{N} (\nocomma H (F_N (t)| \bar{F}_N(t)) + 1),\quad \forall t \in [0, T], \quad \mathbb{P}-a.s.,
		\end{align*}
	\item \begin{align*}
		&\mathbb{E} \[\| \nabla \cdummy [K \ast \mu_N (t) \mu_N (t) - v_t K
		\ast v_t] \|^2_{H^{- \alpha}}\mid \mathcal{F}_{T}^{W}\] \\\leqslant& \frac{C_{\alpha}}{N}
		(\nocomma H (F_N(t) | \bar{F}_N(t)) + 1) , \quad \forall t \in [0, T], \quad \mathbb{P}-a.s.,
	\end{align*}
\item 	\begin{align*}
	&\mathbb{E} \[|\langle \varphi K \ast (\mu_N (t) - v_t), \mu_N (t)
	- v_t \rangle| \mid \mathcal{F}_{T}^{W}\]\\\leqslant& \frac{C}{ N}  (H(F_N(t) |
	\bar{F}_N(t)) + 1) ,\quad \forall t \in [0, T], \quad \mathbb{P}-a.s.,\nonumber
\end{align*}
	\end{enumerate}
	where the random measures $F^N(t)(\mathd x^N)$ and $\bar{F}^N(t)(\mathd x^N)$  on torus $\mathbb{T}^{2N}$ are defined by \eqref{def:F}, \eqref{def:barF}, and  $v_t$ is the unique probabilistically strong solution to the mean field limit equation \eqref{eqt:mean} in the sense of Definitions \ref{def:mean-strong} and \ref{def:uni-mean}.
\end{lemma}
\begin{proof}
	We now obtain estimates concerning the fluctuation measures, based on the results in  \cite{wang2023gaussian}. For example, we express the conditional expectation $	\mathbb{E}\[ \| \mu_N (t) - v_t \|^2_{H^{- \alpha}}\mid \mathcal{F}_T^{W}\]$ through the conditional law of particles $X^N(t)$ with respect to the environmental noise $\mathcal{F}_T^{W},$
	\begin{align*}
		\mathbb{E}\[ \| \mu_N (t) - v_t \|^2_{H^{- \alpha}}\mid \mathcal{F}_T^{W}\]=\int_{\mathbb{T}^{dN}} \| \mu_N  - v_t \|^2_{H^{- \alpha}} F^N(t) (\mathd x^N),
	\end{align*}
	where $F^N(t)(\mathd x^N)=\mathcal{L}(X^N(t)|\mathcal{F}_{T}^{W})(\mathd x^N)$ is a random measure on torus $\mathbb{T}^{2N}.$ 
	Through Donsker-Varadhan variational formula i.e. Lemma \ref{lemma jw0}, we have 
	\begin{equation*}
		\begin{split}
			&\mathbb{E}\[ \| \mu_N (t) - v_t \|^2_{H^{- \alpha}}\mid \mathcal{F}_T^{W}\]\\
			\leqslant& \frac{1}{\kappa N}  \left( \nocomma H(F_N(t) | \bar{F}_N(t)) +
			\log\int_{\mathbb{T}^{2N}} \exp\({\kappa N\|\mu_N-v_t\|_{H^{-\alpha}}^2} \) \bar{F}_N(t) (\mathd x^N)
			\right) , 
		\end{split}
	\end{equation*}
	where  $\bar{F}_N(t)(\mathd x^N)=\mathcal{L}((\bar{X}^N(t))| \mathcal{F}^{W}_T)(\mathd x^N)=v_t^{\otimes
		N}(\mathd x^N)$ is a random measure on torus $\mathbb{T}^{2N}.$ 
	The rest of the proof follows by the same analysis  in \cite[Lemma 2.6-Lemma 2.9]{wang2023gaussian}.
\end{proof}
From Lemma \ref{thm:entropy},  we know that it holds almost surely that for all  \begin{align*}
	\frac{1}{2e^{C_0T}}\big(\nocomma H (F_N (t)| \bar{F}_N(t)) + 1\big)\exp\bigg\{-\int_{0}^{t}C_0\|v_s\|^2_{H^4}\mathd s\bigg\}\leq 1,
\end{align*} 
where $C_{0}$ is a positive deterministic constant depending on $\|v_0\|_{L^2(\mathbb{T}^2)}$ and $\underset{x\in \mathbb{T}^2 }{\inf}v_0.$
By  choosing $m=2e^{C_0T}+C_0>1,$ which depends on $\|v_0\|_{L^2(\mathbb{T}^2)},$$\underset{x\in \mathbb{T}^2 }{\inf}v_0$ and $T,$  we then obtain
\begin{align*}
	\frac{1}{m}\exp\bigg\{-\int_{0}^{t}m\|v_s\|^2_{H^4}\mathd s\bigg\}\big(\nocomma H (F_N (t)| \bar{F}_N(t)) + 1\big)\leq 1,\mathbb{P}-a.s..
\end{align*}
We then deduce  the following result by Lemma \ref{lemma uni mu}.
\begin{corollary}	\label{lemma uni mu2}
For each $t\in[0,T],$  define the weight term for $f\in L^2([0,T];H^4),$ $\mathcal{R}_t(f)=\frac{1}{m}\exp\bigg\{-\int_{0}^{t}m\|f_s\|^2_{H^4}\mathd s\bigg\},$ where the deterministic constant $m>1$   depends on $\|v_0\|_{L^2(\mathbb{T}^2)},$$\underset{x\in \mathbb{T}^2 }{\inf}v_0$ and $T.$ 	Then, for each $\alpha > 1$, there exist
	  constants
	$C_{\alpha}$ and $C$ such that  for all $N\in \mathbb{N},$
	\begin{enumerate}
		\item 	\begin{align*}
			\sup_{t \in [0, T]}\mathbb{E}\[\mathcal{R}_t(v) \|\mu_N (t) - v_t \|^2_{H^{- \alpha}}\]&\leq \frac{C_{\alpha}}{N},
		\end{align*}
	\item 
	\begin{align*}
	\sup_{t \in [0, T]}\mathbb{E}&\[\mathcal{R}_t(v)
	\| \nabla \cdummy [K \ast \mu_N (t) \mu_N (t) - v_t K
	\ast v_t] \|^2_{H^{- \alpha}}\]\leq \frac{C_{\alpha}}{N},		
\end{align*}
		\item  \begin{align*}
			\sup_{t \in [0, T]}\mathbb{E}&\[\mathcal{R}_t(v)\mid \langle \varphi K \ast (\mu_N (t) - v_t), \mu_N (t)
			- v_t \rangle\mid\]\leq \frac{C}{N}.
		\end{align*}		
	\end{enumerate}
\end{corollary}
\begin{remark}
	Since the initial values are i.i.d. random variables,  the classical central limit theorem allows us to  infer that   for each $\varphi\in C^{\infty}(\mT^2)$, we have
	\begin{align*}
		\left\langle \eta^N_0,\varphi\right\rangle =\frac{1}{\sqrt{N}}\sum_{i=1}^N\[\varphi(X_i(0))-\left\langle \varphi,\mu \right\rangle \]\xrightarrow{N\rightarrow \infty} \mathcal{N}\(0,\left\langle \varphi^2,\mu\right\rangle-\left\langle \varphi,\mu\right\rangle^2  \),
	\end{align*}
in the sense of distribution, where $\mathcal{N}(0,a)$ denotes the centered Gaussian distribution on $\mathbb{R}$ with variance $a$.  Using  Lemma \ref{lemma uni mu} which gives the tightness of  laws of $\{\eta^N(0)\}$  in $H^{-\alpha},$ we then conclude that for every $\alpha>1,$ $\eta^N_0$ converges in distribution to some $\eta_0$  in $H^{-\alpha}.$ 
\end{remark}

\subsection{Tightness}\label{sec:compactness}
In this section, we will  prove the  tightness of laws of $\{\mathcal{L}(v,\eta^N,\mathcal{M}^N,W),N\in \mathbb{N}\}$ on the Polish space  
	$\mathcal{H}=\mathcal{V}\times	\mathcal{X}\times\mathcal{Y}\times\mathcal{W},$ which is given by \eqref{def:H}.
We begin by introducing the following pathwise representation of the additive noise part in the decomposition of \eqref{spde-etaN}, given by
\begin{align*}
	\frac{\sqrt{2}}{\sqrt{N}}  \sum_{i = 1}^N \int^t_0 \nabla \varphi (X_{i}) \cdot
	\mathd B_s^i ,
\end{align*}
for each $\varphi \in C^{\infty} (\mathbb{T}^2),$
along with its corresponding estimate.  The proof is provided in \cite{wang2023gaussian}.
\begin{lemma}\label{lemma m}
	For each $N$, there exists a progressively measurable process
	$\mathcal{M}^N$ with values in $H^{- \alpha}$, for every $\alpha > 2$, such that
	\begin{enumerate}
		\item   For all $t
		\in [0, T]$ and $\varphi \in C^{\infty}(\mathbb{T}^d),$ it holds $\mathbb{P}$-a.s., 
		\begin{align*}
		\left\langle \mathcal{M}_t^N, \varphi \right\rangle	 = \frac{\sqrt{2}}{\sqrt{N}}  \sum_{i = 1}^N \int^t_0 \nabla \varphi (X_{i}) \cdot
			\mathd B_s^i ,
		\end{align*}
	
		\item For every $\alpha > 2$,  $\theta'\in (0,\frac{1}{2}),$ there exists  constants $C_{T,\alpha,\theta'}$ and $ C_{T,\alpha}$ such that   \begin{align*}
			&\sup_N \mathbb{E} (\| \mathcal{M}^N
			\|_{C^{\theta'} ([0, T], H^{- \alpha})}^2) \leq C_{T,\alpha,\theta'},
			\\&  \sup_{N}\mathbb{E}\[\sup_{t \in [0, T]}\| \mathcal{M}^N_t
			\|^2_{H^{- \alpha}}\]\leq C_{T,\alpha}.
		\end{align*} 
	\end{enumerate}
	 
	 Furthermore, for every $\alpha > 2,$ the sequence  $(\mathcal{M}^N)_{N \in \mathbb{N}}$ is tight in the space $C([0,T],H^{-\alpha}).$
\end{lemma}
Before proceeding, we introduce the following stopping times  $\{\tau_R, R>0\}.$
\begin{align*}
	\tau_R\assign \inf\bigg \{0<t\leq T: \mathcal{R}_t^{-1}(v)= m\exp\bigg\{\int_{0}^{t}m\|v_s\|^2_{H^4}\mathd s\bigg\}> R\bigg\},
\end{align*} (with the convention $\inf \emptyset=T$)
and define \begin{align*}
	\eta_{R}^N(t)\assign\eta^N(t\wedge
	\tau_{R}),\quad t\in [0,T].
\end{align*} 
Then, we obtain the uniform estmates on $\eta^N$ before the stopping time $\tau_{R}$ through Corollary \ref{lemma uni mu2}.

\begin{lemma}\label{lemma: eta unit}
	 For every $\alpha > 3 $ and $R>0,$   there exists  a constant $C_{\alpha,\sigma,T,R}$ such that
	\begin{align*}
		\sup_{N}\mathbb{E}\sup_{t \in [0, T]}\|\eta_{R}^N(t)\|^2_{H^{-\alpha}}\leq C_{\alpha,\sigma,T,R}.
	\end{align*}
\end{lemma}
\begin{proof}
		Notice that  for each $N\in \mathbb{N},$ it holds $\mathbb{P}$-a.s. that for every $t\in [0,T],$
		\begin{align*}
			\eta^N_{t\wedge\tau_R}-\eta^N_0&=-\sqrt{N}\int_{0}^{t\wedge \tau_{R}}\nabla\cdot (\mu_N(s)K*\mu_N(s))  \mathd s+\sqrt{N}\int_{0}^{t\wedge \tau_{R}}\nabla \cdot  (v_sK*v_s )  \mathd s
			\\&+\int_0^{t\wedge \tau_{R}} \Delta\eta_s^N\mathd s
			+\frac{1}{2}\int_0^{t\wedge \tau_{R}} \sigma\cdot\nabla\(\sigma\cdot\nabla\eta_s^N\) \mathd s
			+\mathcal{M}^N_{t\wedge \tau_{R}}-\int_{0}^{t\wedge \tau_{R}} \sigma\cdot \nabla\eta_s^N\mathd W_s.
		\end{align*}
	Since $\mathcal{R}^{-1}_t(v)\leq R$ before the stopping time $\tau_R,$
		we then have 
		\begin{align*}
			\sup_{t \in [0, T]}\big\|\eta_{t\wedge \tau_{R}}^N\big\|^2_{H^{-\alpha}}\lesssim \|\eta_0^N\|^2_{H^{-\alpha}}+\sum_{i=1}^{5}J_i,
		\end{align*}
		where \begin{align*}
			J_1&\assign\int_0^T R\mathcal{R}_s(v) \|\Delta\eta_s^N\|^2_{H^{-\alpha}} \mathd s,
			\\J_2&\assign\int_0^TR\mathcal{R}_s(v) \|\sigma\cdot\nabla\(\sigma\cdot\nabla\eta_s^N\)\|^2_{H^{-\alpha}} \mathd s,
			\\J_3&\assign\int_{0}^{T}R\mathcal{R}_s(v)N\|\nabla \(\mu_N(s)K*\mu_N(s)-v_sK*v_s  \)\|^2_{H^{-\alpha}}\mathd s,
			\\J_4&\assign\sup_{t \in [0, T]}\|\mathcal{M}_t^N\|^2_{H^{-\alpha}},
			\\J_5&\assign\sup_{t \in [0, T]}\|\int_{0}^{t\wedge \tau_{R}}\sigma\cdot \nabla\eta_s^N \mathd W_s\|^2_{H^{-\alpha}}.
		\end{align*}
	
	Observe that Lemma \ref{lemma triebel} implies $\|\sigma\cdot\nabla\(\sigma\cdot\nabla\eta_s^N\)\|^2_{H^{-\alpha}}\leq C_\sigma\|\eta_s^N\|^2_{H^{-\alpha+2}}.$
		By applying Corollary \ref{lemma uni mu2}, we then have 
		\begin{align*}
			&\sup_{N}\mathbb{E}[J_2]\leq C_{\sigma}RT\sup_{t \in [0, T]}\mathbb{E}\[ \mathcal{R}_t(v) \|\eta_t^N\|^2_{H^{-\alpha+2}} \]\leq RTC_{\alpha,\sigma}.
		\end{align*}
	Similarly, by Corollary \ref{lemma uni mu2}, we have 
	\begin{align*}
		&\sup_{N}\mathbb{E}[J_1+J_3]\leq RTC_{\alpha,\sigma}.
	\end{align*}
	Recall that we have already established  $\sup_{N}\mathbb{E}[J_4]\leq C_{T,\alpha}$ in Lemma \ref{lemma m}.
		For $J_5,$ applying  Burkholder-Davis-Gundy's inequality,  we have 
		\begin{align*}
			\mathbb{E}[J_5]&\leq\mathbb{E}\int_{0}^{T\wedge \tau_{R}}\|\sigma\cdot \nabla\eta_s^N \|^2_{H^{-\alpha}}\mathd s\\&\leq C_{\sigma}RT\sup_{t \in [0, T]}\mathbb{E}\[\mathcal{R}_t(v) \|\eta_t^N\|^2_{H^{-\alpha+1}} \]\leq C_{\alpha,\sigma,T,R}.
		\end{align*}
		Summerizing the estimates above, the proof is then completed.
	\end{proof}
For $\alpha>3,$ 
	we  define the following stopping times  $\{\tau_{ \alpha,M,R}, M,R>0\},$ $$\tau^N_{ \alpha,M,R}\assign \inf \bigg\{0<t\leq T: \sup_{s\in [0, t]}\|\eta_s^{N}\|^2_{H^{-\alpha}}> M\bigg\}\wedge\tau_R,$$
(with the convention $\inf \emptyset=T$)	and the associated stopped process \begin{align*}
		\eta_{ \alpha,M,R}^N(t)\assign\eta^N(t\wedge
		\tau^N_{ \alpha,M,R}),\quad t\in [0,T].
	\end{align*} 
	Based on  Lemma \ref{lemma: eta unit}, we derive a uniform estimate for the H\"older semi-norm  of $\eta^N.$

	\begin{lemma}\label{lemma:Holdermu}
		\label{lemma equi mu} For every
		$\alpha > 3$ and $ \theta\in (0,\frac{1}{2})$, there exists a constant $C_{\alpha,\sigma,\theta,R,M,T}$ such that for every $R>0$ and $M>0,$
		\begin{align}
			\sup_{N}\mathbb{E}\[\|\eta^N_{ \alpha,M,R}\|_{C^{\theta} ([0, T], H^{- \alpha})}\]\leq C_{\alpha,\sigma,\theta,R,M,T},
		\end{align}
		where
		\begin{align*}
			\| f_t \|_{C^{\theta} ([0, T], H^{- \alpha})}\assign \sup_{ 0 \leq s <  t \leq T } \frac{\| f_t - f_s \|_{H^{-
						\alpha}}}{(t - s)^{\theta}}.
		\end{align*}
	\end{lemma}
	\begin{proof}
	Notice that  for each $N\in \mathbb{N},$ it holds $\mathbb{P}$-a.s. that for every $t\in [0,T],$
	\begin{align*}
		\eta^N_{t\wedge\tau^N_{ \alpha,M,R}}-\eta^N_0&=-\sqrt{N}\int_{0}^{t\wedge \tau^N_{ \alpha,M,R}}\nabla\cdot (\mu_N(s)K*\mu_N(s))  \mathd s+\sqrt{N}\int_{0}^{t\wedge \tau^N_{ \alpha,M,R}}\nabla \cdot  (v_sK*v_s )  \mathd s
		\\&+\int_0^{t\wedge \tau^N_{ \alpha,M,R}} \Delta\eta_s^N\mathd s
		+\frac{1}{2}\int_0^{t\wedge \tau^N_{ \alpha,M,R}} \sigma\cdot\nabla\(\sigma\cdot\nabla\eta_s^N\) \mathd s
		\\&+\mathcal{M}^N_{t\wedge \tau^N_{ \alpha,M,R}}-\int_{0}^{t\wedge \tau^N_{ \alpha,M,R}} \sigma\cdot \nabla\eta_s^N\mathd W_s.
	\end{align*}
Thus, $\{\| \eta_t^N - \eta^N_s \|_{H^{- \alpha}}$, $0 \leqslant s < t < T\}$ thus can be 
		controlled by the following relation
		\begin{equation}
			\| \eta_{ \alpha,M,R}^N(t) - \eta^N_{ \alpha,M,R}(s) \|_{H^{- \alpha}} \lesssim  \sum_{i = 1}^7 J_{s,
				t}^i \label{tight mu 2},
		\end{equation}
		where $J_{s, t}^i$, $i = 1, \ldots, 5$, are defined as
		\begin{align}
			J_{s, t}^1 \assign & \left\|  \int^t_s I_{[0,\tau^N_{ \alpha,M,R}]}(r)\Delta \eta_r^N \mathd r
			\right\|_{H^{- \alpha}}, \quad\quad\quad
			J^2_{s, t} \assign  \left\| \int^t_sI_{[0,\tau^N_{ \alpha,M,R}]}(r) \mathcal{K}^N_r \mathd r
			\right\|_{H^{- \alpha}}, \nonumber\\
			J^3_{s, t} \assign & \left\| \int^t_s I_{[0,\tau^N_{ \alpha,M,R}]}(r)\sigma\cdummy\nabla (\sigma\cdummy\nabla\eta_r^N) \mathd r
			\right\|_{H^{- \alpha}}, \quad
			J^4_{s, t} \assign  \left\| \mathcal{M}^N_{t\wedge\tau^N_{ \alpha,M,R}}-\mathcal{M}^N_{s\wedge\tau^N_{ \alpha,M,R}}\right\|_{H^{- \alpha}}, \nonumber\\
			J^5_{s, t} \assign & \left\| \int_{s}^{t}I_{[0,\tau^N_{ \alpha,M,R}]}(r)\sigma\cdummy\nabla\eta_r^N \mathd W_r \right\|_{H^{-
					\alpha}} . \nonumber
		\end{align}
		where $\mathcal{K}^N_t = \sqrt{N} \nabla \cdummy [K \ast \mu_N (t)
		\mu_N (t) - K \ast v_tv_t].$
		
			For the drift terms $J_{s,
		t}^i,i=1,2,3,$	 it is sufficient  to prove that $$\sup_N \mathbb{E} \left( \sup_{0 \leqslant s < t \leqslant T} \frac{J_{s,
				t}^i}{(t - s)^{\frac{1}{2}}} \right)^2<C.$$
		Indeed, for every $\theta\in (0,\frac{1}{2}),$ we have the following estimate.
\begin{align*}
	\mathbb{E} \left( \sup_{0 \leqslant s < t \leqslant T} \frac{J_{s,
			t}^i}{(t - s)^{\theta}} \right)\leq&	\mathbb{E} \left( \sup_{0 \leqslant s < t \leqslant T} \frac{J_{s,
			t}^i}{(t - s)^{\frac{1}{2}}} \right)T^{\frac{1}{2}-\theta}
		\leq\[\mathbb{E} \left( \sup_{0 \leqslant s < t \leqslant T} \frac{J_{s,
				t}^i}{(t - s)^{\frac{1}{2}}} \right)^2\]^{\frac{1}{2}}T^{\frac{1}{2}-\theta}.
\end{align*}
		For $J_{s,
		t}^1,$	 note that $\mathcal{R}^{-1}_t(v)=m\exp\bigg\{\int_{0}^{t}m\|v_s\|^2_{H^4}\mathd s\bigg\}\leq R$ before the stopping time $\tau_{ \alpha,M,R},$ and we apply H\"older's inequality to derive the following estimate
		\begin{align}\label{holder-1}
		\begin{split}
				\sup_N \mathbb{E} \left( \sup_{0 \leqslant s < t \leqslant T} \frac{J_{s,
					t}^1}{(t - s)^{\frac{1}{2}}} \right)^2 \leq&\sup_N \mathbb{E}\[ \int^T_0 I_{[0,\tau^N_{ \alpha,M,R}]}(t)\| \Delta
			\eta_t^N \|^2_{H^{- \alpha}} \mathd t \]
			\\\leq& \sup_N \mathbb{E}\[\int^T_0 R\mathcal{R}_t(v)\| \Delta
			\eta_t^N \|^2_{H^{- \alpha}} \mathd t \]
			\\\leq& RT\sup_N \sup_{t \in [0, T]}\mathbb{E}\[\mathcal{R}_t(v)\| \Delta
			\eta_t^N \|^2_{H^{- \alpha}}\]\leq RTC_{\alpha},
		\end{split}
		\end{align}
		where the final inequality follows from  Corollary \ref{lemma uni mu2}.
			Using a similar approach,  we have
		\begin{align}\label{holder-2}
			\begin{split}
				\sup_N \mathbb{E} \left( \sup_{0 \leqslant s < t \leqslant T} \frac{J_{s,
						t}^2}{(t-s)^{\frac{1}{2}}} \right)^2 \leq &  RT\sup_N \sup_{t \in [0, T]} \mathbb{E} \[\mathcal{R}_t(v)\|\mathcal{K}^N_t
				\|^2_{H^{- \alpha}}\] \leq RTC_{\alpha}, 
			\end{split}
		\end{align}
		where we used Corollary \ref{lemma uni mu2} to reach the final inequality, and
		\begin{align}\label{holder-3}
			\begin{split}
				\sup_N \mathbb{E} \left( \sup_{0 \leqslant s < t \leqslant T} \frac{J_{s,
						t}^3}{(t-s)^{\frac{1}{2}}} \right)^2 \leq & RT\sup_N \sup_{t \in [0, T]} \mathbb{E} \[\mathcal{R}_t(v)\|\sigma\cdummy\nabla (\sigma\cdummy\nabla\eta_t^N)
				\|^2_{H^{- \alpha}}\]
				\\\leq &C_{\sigma}R T\sup_N \sup_{t \in [0, T]} \mathbb{E} \[\mathcal{R}_t(v)\|\eta_t^N
				\|^2_{H^{- \alpha+2}}\] \leq RTC_{\alpha,\sigma}. 
			\end{split}
		\end{align}
		where we applied Lemma \ref{lemma triebel} to get the second inequality and Lemma \ref{lemma uni mu2} to get the final inequality.
	
		Followng the same approach in Lemma \ref{lemma m}, i.e., \cite[Lemma 3.2]{wang2023gaussian}, we conclude that  for any $\theta\in(0,\frac{1}{2}),$
		\begin{align}\label{holder-4}
			\sup_N \mathbb{E} \bigg(\sup_{0 \leqslant s < t \leqslant T}\frac{J^4_{s, t}}{|t-s|^{\theta}}\bigg) \leq C_{\alpha, \theta}. 
		\end{align}

	For $J_{s,
			t}^5$, applying  Burkholder-Davis-Gundy's inequality {\cite[Theorem 2.3.8]{breit2018stochastically}} gives, for any $\theta' >1,$
		\begin{align*}
			\begin{split}
				\sup_N \mathbb{E} (|J^5_{s, t}|^{2 \theta'}) \leq& C_{\theta'} \sup_N\mathbb{E}\[\int_{s}^{t}I_{[0,\tau^N_{ \alpha,M,R}]}(r)\|\sigma\cdummy\nabla\eta_{R}^N
				\|^{2}_{H^{-
						\alpha}} \mathd r\]^{\theta'}
				\\\leq&  C_{\theta',\sigma} \sup_N\mathbb{E}\[\int_{s}^{t}I_{[0,\tau^N_{ \alpha,M,R}]}(r)\sup_{s\in [0, r]}\|\eta_s^{N}\|^2_{H^{-\alpha}}\mathd r\]^{\theta'}\leq  C_{\theta',\sigma,M}\mid t-s\mid^{\theta'},
			\end{split}
		\end{align*}
		where in the final inequality we use the fact that for $\mathbb{P}$-a.s.,
			$\sup_{s\in [0, t]}\|\eta_R^{N}(s)\|^2_{H^{-\alpha}}\leq  M$
		up to the stopping time $\tau_{ \alpha,M,R}.$
		By applying the Kolmogorov continuity theorem {\cite[Theorem 2.3.11]{breit2018stochastically}}, we then deduce that for any $\theta\in(0,\frac{1}{2}),$
		\begin{align}\label{holder-5}
			\sup_N \mathbb{E} \bigg(\sup_{0 \leqslant s < t \leqslant T}\frac{J^5_{s, t}}{|t-s|^{\theta}}\bigg) \leq C_{ \theta,\sigma,M}. 
		\end{align}   The result follows by combining inequalities \eqref{holder-1}-\eqref{holder-5}.
	\end{proof}
	
	Then, we  demonstrate the tightness of the fluctuation measures $(\eta^N)_{N \geqslant 1}$ in the Polish space $$\mathcal{X}= \left\{\bigcap_{k\in \mathbb{N}}  \left[ C ([0, T],
	H^{-3-\frac{1}{k}}) \cap L^2([0,T],H^{-1-\frac{1}{k}})\right] \right\}.$$
	\begin{lemma}
		\label{lemma tight mu}The sequence of laws of $(\eta^N)_{N \geqslant 1}$ is
		tight in the space $\mathcal{X}.$
	\end{lemma}
	\begin{proof}
		First, we claim that it is sufficient to prove the tightness of the sequence of laws of  $(\eta^N)_{N \geqslant 1}$ in the space $$\mathcal{X}_k\assign C ([0, T],
		H^{-3-\frac{1}{k}}) \cap L^2([0,T],H^{-1-\frac{1}{k}})$$ for each fixed $k\in \mathbb{N}.$  	Indeed, if  the sequence of  laws  of $(\eta^N)_{N \geqslant 1}$ is tight in the space $\mathcal{X}_k$ for each $k,$ then
		for any $ \vartheta>0,$  we can choose compact sets $A_k^{ \vartheta}$ in $ C ([0, T],
		H^{- 3-\frac{1}{k}}) \cap L^2([0,T],H^{-1-\frac{1}{k}})$ for each $k\in \mathbb{N}$ such that
		\begin{align*}
			\mathbb{P}(\eta^N\notin A_k^{ \vartheta})<  \vartheta 2^{-k},\quad \forall N\in \mathbb{N}.
		\end{align*}
		The set $A^{ \vartheta}$ in $\mathcal{X}$ defined by
		\begin{align*}
			A^{ \vartheta}\assign \bigcap_{k\in \mathbb{N}}A_k^{ \vartheta}
		\end{align*}
		is   compact 
		and satisfies
		\begin{align*}
			\mathbb{P}\(\eta^N
			\notin A^{ \vartheta}\)\leq \sum_{k\in\mathbb{N}}\mathbb{P}(\eta^N\notin A_k^{ \vartheta})< \vartheta,\quad \forall N\in \mathbb{N},
		\end{align*}
		which implies  $(\eta^N)_{N \geqslant 1}$ is
		tight in the space $\mathcal{X}.$

	Next, we prove that the sequence $(\eta^N)_{N \geqslant 1}$ is
		tight in the space $ C ([0, T],
		H^{- \alpha-2}) \cap L^2([0,T],H^{-\alpha}),$ for every
		$\alpha > 1 .$
		For any $\delta>0$  and $\alpha>\alpha' > 1 ,$ we define  
		\begin{align*}
			K^{\delta}\assign&\biggl\{ \eta\big|  \sup_{t \in [0, T]}\|\eta_{t}\|_{H^{-\alpha'-2}}\leq \frac{1}{\delta }, \int_{0}^{T}\|\eta_{t}\|^2_{H^{-\frac{2\alpha+2}{4}}}\mathd t\leq \frac{1}{\delta },\|\eta\|_{C^{\frac{1}{8}}([0,T],H^{-\alpha-2})}\leq \frac{1}{\delta }\biggr\}
		\end{align*}
		which is   a compact subset of  $ C ([0, T],H^{- \alpha-2}) \cap L^2([0,T],H^{-\alpha})$ as established by \cite[corollary 1.8.4]{breit2018stochastically} and Arzela-Ascoli theorem \cite[Theorem 7.17]{kelley2017topology}.
			Recall the stopping times and stopped processes, 
		\begin{align*}
			\tau_R&= \inf \bigg\{0<t\leq T:\mathcal{R}^{-1}_t(v)= m\exp\bigg\{\int_{0}^{t}m\|v_s\|^2_{H^4}\mathd s\bigg\}>R\bigg\},
			\\\tau^N_{ \alpha+2,M,R}&= \inf \bigg\{0<t\leq T: \sup_{s\in [0, t]}\|\eta_R^{N}(s)\|^2_{H^{-\alpha-2}}> M\bigg\}\wedge\tau_R,
		\end{align*} 
	and \begin{align*}
		\eta_R^N(t)&=\eta^N(t\wedge
		\tau_R),\quad t\in [0,T],			
		\\\eta_{ \alpha+2,M,R}^N(t)&=\eta^N(t\wedge
		\tau^N_{ \alpha+2,M,R}),\quad t\in [0,T],
	\end{align*}
		we can directly conclude that for every $R>0$ and $M>0,$
		\begin{align*}
			&\mathbb{P}\bigg(\eta^N
			\notin K^\delta\bigg)\\\leq& \mathbb{P}\bigg(\eta^N=\eta_{R}^N=\eta_{ \alpha+2,M,R}^N
			\notin K^\delta, m\exp\bigg\{\int_{0}^{T}m\|v_s\|^2_{H^4}\mathd s\bigg\}\leq R,\sup_{t\in [0, T]}\|\eta_R^{N}(t)\|^2_{H^{-\alpha-2}}\leq M\bigg)\\+&\mathbb{P}\bigg(m\exp\bigg\{\int_{0}^{T}m\|v_s\|^2_{H^4}\mathd s\bigg\}\geq R\bigg)
			+\mathbb{P}\bigg(\sup_{t\in [0, T]}\|\eta_R^{N}(t)\|^2_{H^{-\alpha-2}}\geq M\bigg).
		\end{align*}
	Given $\eps>0,$  since  
	$\mathcal{R}^{-1}_T(v)=m\exp\bigg\{\int_{0}^{T}m\|v_s\|^2_{H^4}\mathd s\bigg\}<\infty,
	\mathbb{P}$-a.s.,  we can thus choose a sufficiently large $R_0>0$ such that $$\mathbb{P}\(m\exp\bigg\{\int_{0}^{T}m\|v_s\|^2_{H^4}\mathd s\bigg\}\geq R_0\)\leq\frac{\eps }{4}.$$
	By Chebyshev's inequality, we have\begin{align*}
		\mathbb{P}\(\sup_{t\in [0, T]}\|\eta_{R_0}^{N}(t)\|^2_{H^{-\alpha-2}}\geq M\)\leq\frac{\mathbb{E}\[\sup_{t\in [0, T]}\|\eta_{R_0}^{N}(t)\|^2_{H^{-\alpha-2}}\]}{M}.
	\end{align*}
Using Lemma \ref{lemma: eta unit},	we can  choose a sufficiently large $M_0>0$ such that for all $N\in\mathbb{N},$ \begin{align*}
	\mathbb{P}\bigg(\sup_{t\in [0, T]}\|\eta_{R_0}^{N}(t)\|^2_{H^{-\alpha-2}}\geq M_0\bigg)\leq \frac{\eps }{4}.
\end{align*}
We then conclude that 
\begin{align*}
		\mathbb{P}\bigg(\eta^N
	\notin K^\delta\bigg)&\leq \mathbb{P}\bigg(\eta^N=\eta_{R_0}^N=\eta_{\alpha+2,M_0,R_0}^N
	\notin K^\delta, m\exp\bigg\{\int_{0}^{T}m\|v_s\|^2_{H^4}\mathd s\bigg\}\leq R_0,\\&\text{and }\sup_{t\in [0, T]}\|\eta_{R_0}^{N}(t)\|^2_{H^{-\alpha-2}}\leq M_0\bigg)+\frac{\eps}{2}.
	\\&\leq\mathbb{P}\bigg(\sup_{t \in [0, T]}\|\eta_{R_0}^N(t)\|_{H^{-\alpha'-2}}\geq \frac{1}{\delta}\bigg)
	+\mathbb{P}\bigg(\|\eta_{\alpha+2,M_0,R_0}^N\|_{C^{\frac{1}{8}}([0,T],H^{-\alpha-2})}\geq \frac{1}{\delta }\bigg)
	\\&+\mathbb{P}\bigg(R_0\int_{0}^{T}\mathcal{R}_t(v)\|\eta^N_{t}\|^2_{H^{-\frac{2\alpha+2}{4}}}\mathd t\geq \frac{1}{\delta }\bigg)+\frac{\eps}{2}.
\end{align*}
		 Applying Chebyshev's inequality again, along with Corollary \ref{lemma uni mu2}, Lemma \ref{lemma: eta unit} and Lemma \ref{lemma:Holdermu}, we can select a sufficiently small  $\delta_0>0$ such that $\mathbb{P}\(\eta^N
		\notin K^{\delta_0}\)<\eps,$ which completes the proof.
	\end{proof}
	
	It is well known that every probability measure on a Polish space is tight. Then, we have the tightness of laws of $v$ and $W.$ Combining Lemma \ref{lemma m} and Lemma \ref{lemma tight mu}, we conclude the following result.
	\begin{lemma}\label{tightness}
		The law of $\( v,\eta^N,\mathcal{M}^N, W\)$ with values in $\mathcal{H}$ as defined in \eqref{def:H} is tight.
	\end{lemma}

		By Skorohod's representation, we obtain the following result.
		\begin{proposition}\label{pro:skorokhod}
			There exists a subsequence of $(v,\eta^N,
			\mathcal{M}^N,W)_{N \geqslant 1}$, still denoted by $(v,\eta^{N}, \mathcal{M}^{N},W)$ for simplicity,
			and a probability space $(\tilde{\Omega}, \tilde{\mathcal{F}},
			\tilde{\mathbb{P}})$ with $\mathcal{H}$-valued random variables
			$(\tilde{v}^N,\tilde{\eta}^{N_{}}, \tilde{\mathcal{M}}^N,\tilde{W}^N)_{N \geqslant 1}$ and
			$(\tilde{v},\tilde{\eta}, \tilde{M},\tilde{W})$ such that
			\begin{enumerate}
				\item For each $N \in \mathbb{N}$, the law of $(\tilde{v}^N,\tilde{\eta}^{N_{}}, \tilde{\mathcal{M}}^N,\tilde{W}^N)$ coincides with the law of $(v,\eta^{N},
				\mathcal{M}^{N},W)$.
				\item The sequence of $\mathcal{H}$-valued random variables $(\tilde{v}^N,\tilde{\eta}^{N_{}}, \tilde{\mathcal{M}}^N,\tilde{W}^N)_{N \geqslant 1}$ converges to $(\tilde{v},\tilde{\eta}, \tilde{M},\tilde{W})$ in $\mathcal{H}$  $\tilde{	\mathbb{P}}$-a.s..
			\end{enumerate}
		\end{proposition}  
		\begin{remark}
			We emphasize that $(\tilde{v}^N,
			\tilde{W}^N)_{N \geqslant 1}$ are different random variables, but they share the same law on the new probability space $(\tilde{\Omega}, \tilde{\mathcal{F}},
			\tilde{\mathbb{P}}).$  
		\end{remark}
	
\section{Well posedness of the fluctuation SPDE}\label{sec:wellposed}
This section aims to establish the well-posedness of the fluctuation SPDE \eqref{LimitSpde} and the convergence from the fluctuation measures $(\tilde{\eta}^N)_{N\geq 1}$ to the probabilistically strong solution  $\tilde{\eta}$ of the fluctuation SPDE \eqref{LimitSpde}.

\subsection{Identification of the limiting points}\label{sec:identifi}
	In this subsection, we identify the limiting points of the fluctuation process  $(\tilde{\eta}^N)_{N\geq 1}$  as  probabilistically weak solutions to the fluctuation SPDE \eqref{LimitSpde}.  The proof proceeds in several steps. First, we identify the joint law of $(\tilde{M},\tilde{W})$ in Lemma \ref{lem:unilaw}. Next, we study the convergence of the interacting term $\tilde{ \mathcal{K}}_t^N (\varphi)$ in Lemma \ref{lemma limit no}. In the third step, we show that the limit process $\tilde{v},$ given in Propositon \ref{pro:skorokhod}, is the unique probabilistically strong solution to the mean field limit equation \eqref{eqt:mean}, as shown in Proposition \ref{prop:w}. Finally, we prove the convergence of the transport noise term $\sigma\cdot \nabla \tilde{\eta}_t \mathd \tilde{W}_t$  in Proposition \ref{prop:transport}.
	
	There are two  challenges in this section. Firstly, unlike in the case without environmental noise $W,$ we must address not only the convergence of the fluctuation measures $(\tilde{\eta}^N)_{N\geq 1}$  and the additive noise $(\tilde{\mathcal{M}}^N)_{N\geq 1},$ but also the convergence of the new multiplicative noise term $\sigma \cdot \nabla\tilde{\eta}^N \mathd \tilde{W}^{N}_t,$ as well as   the convergence of the random mean field limit $\tilde{v}^N.$ 	Furthermore, we  identifying the joint law of $(\tilde{M},\tilde{W})$ by  studying  the conditional law of additive noise $\{\tilde{\mathcal{M}}_t, t\in [0,T]\}$ with respect to the environmental noise $\mathcal{F}^W_T,$ which is the main challenge in this section.  To address this, we establish the following strong convergence from the interacting particle system \eqref{eqt:vortex} to the conditional Mckean-Vlasov equation \eqref{eqt:ncopy}.

		\begin{proposition}\label{prop:strong-convergence}
		For $i\in \mathbb{N}$ and $t\in[0,T],$ we have \footnote{The well-posedness of stochastic point  vortex model \eqref{eqt:vortex} and conditional Mckean-Vlasov equation \eqref{eqt:ncopy}  has been given in \cite[Lemma 3.4, Proposition 4.1]{shao2024quantitative}. } 
		\begin{align}\label{pathwise-convergence}
			\lim\limits_{N\rightarrow \infty}	\mathbb{E}\big| X_i^N(t)-\bar{X}_i(t)\big|^2=0,
		\end{align}
	where $X^N$ is the unique probabilistically strong solution to the stochastic point  vortex model  \eqref{eqt:vortex} with i.i.d. initial values $X^N(0)=(X_1(0),\cdots,X_N(0))$ and $\bar{X}_i$ is the unique probabilistically strong solution to the  conditional Mckean-Vlasov equation \eqref{eqt:ncopy} with initial value $X_i(0).$
	\end{proposition}
	\begin{proof}
		
		Consider the difference between the particle system \eqref{eqt:vortex} and the conditional Mckean-Vlasov equation \eqref{eqt:ncopy},
		\begin{align*}
			X^N_i(t)-\bar{X}_i(t)&= \int_0^{t}\(K*\mu_N(s)(X^N_i(s))-K*v_s(\bar{X}_i(s))\)\mathd s\\&+\int_0^{t}\(\sigma(X^N_i(s))-\sigma(\bar{X}_i(s))\)\circ\mathd W_s
			\\&= \int_0^{t}\(K*\mu_N(s)(X^N_i(s))-K*v_s(\bar{X}_i(s))\)\mathd s\\&+\frac{1}{2}\int_{0}^{t}\((\sigma\cdot\nabla\sigma)(X^N_i(s))-(\sigma\cdot\nabla\sigma)(\bar{X}_i(s))\)\mathd s+\int_0^{t}\(\sigma(X^N_i(s))-\sigma(\bar{X}_i(s))\) \mathd W_s.
		\end{align*} 
	where we convert the Stratonovich integral  into It\^o's integral in the last equality.
		
	In our previous work \cite[Lemma 3.2]{shao2024quantitative}, we showed that $v\in C([0,T];H^2),\mathbb{P}$-a.s..	We now define the stopping times $\{\Theta_R,R>0\}$ (with the convention $\inf \emptyset=T$) as follows.
		\begin{align*}
			\Theta_R\assign \inf\bigg\{0<t\leq T| H_t(v)\assign\sup_{s\in [0, t]}\|v_s\|_{H^2}+m\exp\bigg\{\int_{0}^{t}m\|v_s\|^2_{H^4}\mathd s\bigg\}>R\bigg\}.
		\end{align*}
		Applying It\^o's formula, we obtain
		\begin{align*}
			\mid 	X^N_i(t\wedge \Theta_R)-\bar{X}_i(t\wedge \Theta_R)\mid^2 &=\int_{0}^{t\wedge \Theta_{R}}(X^N_i(s)-\bar{X}_i(s))\(K*\mu_N(s)(X^N_i(s))-K*v_s(\bar{X}_i(s))\)\mathd s
			\\&+\frac{1}{2}\int_{0}^{t\wedge\Theta_R} (X^N_i(s)-\bar{X}_i(s))\((\sigma\cdot\nabla\sigma)(X^N_i(s))-(\sigma\cdot\nabla\sigma)(\bar{X}_i(s))\)\mathd s
			\\&+\int_0^{t\wedge \Theta_R}(X^N_i(s)-\bar{X}_i(s))\(\sigma(X^N_i(s))-\sigma(\bar{X}_i(s))\) \mathd W_s
			\\&+\int_0^{t\wedge \Theta_R}\(\sigma(X^N_i(s))-\sigma(\bar{X}_i(s))\)^2 \mathd s.
		\end{align*}
		Direct computation yields the following identity
		\begin{align*}
				\mid 	X^N_i(t\wedge \Theta_R)-\bar{X}_i(t\wedge \Theta_R)\mid^2 &=\int_{0}^{t\wedge \Theta_R}(X^N_i(s)-\bar{X}_i(s))\(K*\mu_N(s)(X^N_i(s))-K*v_s(X^N_i(s))\)\mathd s
			\\&+\int_{0}^{t\wedge \Theta_R}(X^N_i(s)-\bar{X}_i(s))\(K*v_s(X^N_i(s))-K*v_s(\bar{X}_i(s))\)\mathd s
			\\&+\frac{1}{2}\int_{0}^{t\wedge\Theta_R} (X^N_i(s)-\bar{X}_i(s))\((\sigma\cdot\nabla\sigma)(X^N_i(s))-(\sigma\cdot\nabla\sigma)(\bar{X}_i(s))\)\mathd s
			\\&+\int_0^{t\wedge \Theta_R}(X^N_i(s)-\bar{X}_i(s))\(\sigma(X^N_i(s))-\sigma(\bar{X}_i(s))\) \mathd W_s
			\\&+\int_0^{t\wedge \Theta_R}\(\sigma(X^N_i(s))-\sigma(\bar{X}_i(s))\)^2 \mathd s.
		\end{align*}
	By  Sobolev embedding theorem, we deduce that \begin{align}\label{lip:kvs}
		\|K* v_s\|_{C^1}&\leq\|K* v_s\|_{H^3}=\|\nabla^{\perp} (-\Delta)^{-1}  v_s\|_{H^3}\leq \|v_s\|_{H^2}.
	\end{align}
Here we used the fact the   Biot-Savart kernel $K=\nabla^{\perp}G,$ where $G$ is the Green function of $-\Delta$ on torus, as given in \cite{flandoli2011full}.
	Through  the compactness of torus, together with the Lipschitz property of $\sigma,\sigma\cdot \nabla \sigma,K*v_s$ and the estimate for $K*v_s$ given in \eqref{lip:kvs}, we obtain the following estimate.
		\begin{align*}
		\mid 	X^N_i(t\wedge \Theta_R)-\bar{X}_i(t\wedge \Theta_R)\mid^2 &\leq  C\int_0^{t\wedge \Theta_R}\big|K*\mu_N(s)(X^N_i(s))-K*v_s(X^N_i(s))\big|\mathd s
		\\&+\int_{0}^{t}(C_\sigma+R)\mid 	X^N_i(s\wedge \Theta_R)-\bar{X}_i(s\wedge \Theta_R)\mid^2 \mathd s
		\\&+\int_0^{t\wedge \Theta_R}(X^N_i(s)-\bar{X}_i(s))\(\sigma(X^N_i(s))-\sigma(\bar{X}_i(s))\) \mathd W_s.
	\end{align*}
		Taking expectation of both sides, we then conclude that 
		\begin{align*}
			\mathbb{E}\[\mid 	X^N_i(t\wedge \Theta_R)-\bar{X}_i(t\wedge \Theta_R)\mid^2\]\leq &(C_\sigma+R) \int_{0}^{t} \mathbb{E}\[\mid 	X^N_i(s\wedge \Theta_R)-\bar{X}_i(s\wedge \Theta_R)\mid^2\] \mathd s
			\\&+C\int_{0}^{T}\mathbb{E}\[I_{\{s\leq\Theta_R\}}\mid K*\mu_N(s)(X^N_i(s))-K*v_s(X^N_i(s))\mid\]\mathd s.
		\end{align*}

		We now deal with the interacting  term $K*\mu_N(s)(X^N_i(s))-K*v_s(X^N_i(s)).$
		As done in \cite[Section 3.2]{flandoli2011full}, we first regularize $K$ by introducing  smooth periodic functions $K_{\eps}$ such that $K_{\eps}(x)= K(x)$ for any $|x|>\eps$ and they satisfy \begin{align}\label{biot-property}
			|K(x)|+|K_\eps(x)|\lesssim \frac{1}{|x|},\quad \forall  \eps >0.
		\end{align}
		We decompose the interaction term into three parts.
		\begin{align*}
			K*\mu_N(s)(X^N_i(s))-K*v_s(X^N_i(s))=\sum_{i=1}^{3}J_i,
		\end{align*}
		where \begin{align*}
			J_1(s)&\assign	K*\mu_N(s)(X^N_i(s))-K_{\eps}*\mu_N(s)(X^N_i(s)),
			\\J_2(s)&\assign K_\eps*\mu_N(s)(X^N_i(s))-K_\eps*v_s(X^N_i(s)),
			\\J_3(s)&\assign K_\eps*v_s(X^N_i(s))-K*v_s(X^N_i(s)).
		\end{align*}
	 Recall that $v$ is a continuous $L^2$-valued $\mathcal{F}^W_t$-adapted process, we then have 
		\begin{align*}\int_{0}^{T}\mathbb{E}\[I_{\{s\leq\Theta_R\}}J_1(s)\]\mathd s=&\int_{0}^{T}\mathbb{E}\[I_{\{s\leq\Theta_R\}}\mathbb{E}\[\frac{1}{N}\sum_{j\neq i}(K-K_\eps)(X^N_i(s)-X^N_j(s))|\mathcal{F}^W_T\]\]\mathd s
			\\\leq& \frac{N-1}{N}\int_{0}^{T}\mathbb{E}\[I_{\{s\leq\Theta_R\}}\mathbb{E}\[|K-K_\eps|(X^N_1(s)-X^N_2(s))|\mathcal{F}^W_T\]\]\mathd s,
		\end{align*}
		where the second inequality follows from the symmetry of the random measure $\mathcal{L}(X^N(t)|\mathcal{F}^W_T)(\mathd x^N)$ on $\mathbb{T}^{2N}.$ 
		The upper bound \eqref{biot-property} for the   Biot-Savart kernel $K$ and its regularized version $K_\eps$ yields \begin{align*}\int_{0}^{T}\mathbb{E}\[I_{\{s\leq\Theta_R\}}J_1(s)\]\mathd s\lesssim& \int_{0}^{T}\mathbb{E}\[I_{\{s\leq\Theta_R\}}\int_{\mathbb{T}^2}\frac{1}{|x_1-x_2|}I_{\{|x_1-x_2|\leq \eps \}}F^{N;2}(s)(\mathd x_1,\mathd x_2) \]\mathd s
			\\\leq& \mathbb{E}\[\eps^{\frac{1}{2}}\int_{\mathbb{T}^2}I_{\{|x_1-x_2|\leq \eps \}}\frac{1}{|x_1-x_2|^{\frac{3}{2}}}F^{N;2}(s)(\mathd x_1,\mathd x_2)\]\mathd s,
		\end{align*}
	where $F^{N;2}(t)(\mathd x_1,\mathd x_2)$ is the $2$-marginal of the random measure $F^N(t)(\mathd x^N)=\mathcal{L}(X^N(t)|\mathcal{F}^W_T)(\mathd x^N)$ on $\mathbb{T}^{2N}.$
Using the property of Fisher information, i.e., Lemma \ref{fournier3.3}, we conclude that 
	\begin{align*}\int_{0}^{T}\mathbb{E}\[I_{\{s\leq\Theta_R\}}J_1(s)\]\mathd s\lesssim&\eps^{\frac{1}{2}}\mathbb{E}\int_{0}^{T}\[C(I(F^{N;2}(s))+1)\]\mathd s
		\\\leq&\eps^{\frac{1}{2}}\mathbb{E}\int_{0}^{T}\[C(\frac{2}{N}I(F^{N}(s))+1)\]\mathd s\leq C_{v_0}\eps^{\frac{1}{2}}.
	\end{align*}
Here we used the sub-additivity of Fisher information, i.e., $I(F^{N,k}(t))\leq \frac{k}{N}I(F^N(t)), \forall k\leq N,$ to derive the second inequality. For the final inequality, we applied  the boundness about the Fisher information $I(F^N(t))$ for $F^N(t)(\mathd x^N),$  i.e., $\int_{0}^{T}I(F^N(t))\mathd t\leq C_{v_0},$ given in \cite[Lemma 4.2]{shao2024quantitative}.
		
	Moving on to $J_2,$ by the Sobolev embedding theorem, we find that  for $\alpha>1,$
		\begin{align*}
			\int_{0}^{T}\mathbb{E}\[I_{\{s\leq\Theta_R\}}J_2(s)\]\mathd s=&
			\int_{0}^{T}	\mathbb{E}\[I_{\{s\leq\Theta_R\}}\mathbb{E}\[K_\eps*\mu_N(s)(X^N_i(s))-K_\eps*v_s(X^N_i(s))|\mathcal{F}^W_T\]\]\mathd s
			\\\leq&\int_{0}^{T}\mathbb{E}\[I_{\{s\leq\Theta_R\}}\mathbb{E}\[\|K_\eps*\mu_N(s)-K_\eps*v_s\|_{L^\infty}|\mathcal{F}^W_T\] \mathd s
			\\\lesssim&\int_{0}^{T}\mathbb{E}\[I_{\{s\leq\Theta_R\}}\mathbb{E}\[\|K_\eps*\mu_N(s)-K_\eps*v_s\|_{H^{2}}|\mathcal{F}^W_T\] \mathd s.
		\end{align*}
		Applying   Lemma \ref{lemma convolution} yields 
			\begin{align*}
			\int_{0}^{T}\mathbb{E}\[I_{\{s\leq\Theta_R\}}J_2(s)\]\mathd s\lesssim&
			\int_{0}^{T}\mathbb{E}\[I_{\{s\leq\Theta_R\}}\mathbb{E}\[\|K_\eps\|_{B^{2+\alpha}_{1,2}}\|\mu_N(s)-v_s\|_{H^{-\alpha}}|\mathcal{F}^W_T\] \mathd s
			\\\leq&\|K_\eps\|_{B^{2+\alpha}_{1,2}}\int_{0}^{T}\mathbb{E}\[I_{\{s\leq\Theta_R\}}R\mathcal{R}_s(v)\|\mu_N(s)-v_s\|_{H^{-\alpha}}\] \mathd s
			\\\leq& \|K_\eps\|_{B^{2+\alpha}_{1,2}}T\frac{C_{R,\alpha}}{N},
		\end{align*} where the last inequality is derived using Corollary \ref{lemma uni mu2}.
		
		Finally, using the upper bound \eqref{biot-property} of   Biot-Savart kernel $K$ and regularized version   Biot-Savart kernel $K_\eps$ again, we have
		\begin{align*}
			\int_{0}^{T}\mathbb{E}\[I_{\{s\leq\Theta_R\}}J_3(s)\]\mathd s&\leq\int_{0}^{T}\mathbb{E}\[|K_\eps*v_s(X^N_i(s))-K*v_s(X^N_i(s))|\]\mathd s
			\\&\lesssim \int_{0}^{T}\mathbb{E}\[\int_{\mathbb{T}^2}\frac{1}{|X^N_i(s)-y|}I_{\{|X^N_i-y|\leq \eps \}}v_s(y)\mathd y\]\mathd s
			\\&\leq C_{v_0}\eps^{\frac{1}{2}} \int_{0}^{T}\mathbb{E}\[\int_{\mathbb{T}^2}\frac{1}{|X^N_i(s)-y|^{\frac{3}{2}}}\mathd y\]\mathd s\leq C_{v_0,T}\eps^{\frac{1}{2}} .
		\end{align*}
		Here the third inequality follows from the regularity of $v$ in Definition \eqref{def:mean}, i.e., $\|v_t\|_{L^{\infty}}\leq\|v_0\|_{L^{\infty}},\forall t\in[0,T],\mathbb{P}$-a.s..
		In summary,	we  conclude that 
		\begin{align*}
			\mathbb{E}\[	\mid 	X^N_i(t\wedge \Theta_{R})-\bar{X}_i(s\wedge \Theta_{R})\mid^2\]\leq&(C_\sigma+R) \int_{0}^{t} \mathbb{E}\[	\mid 	X^N_i(s\wedge \Theta_{R})-\bar{X}_i(s\wedge \Theta_{R})\mid^2\] \mathd s
			\\& +C_{v_0,T}\eps^{\frac{1}{2}}+\|K_\eps\|_{B^{2+\alpha}_{1,2}}T\frac{C_{R,\alpha}}{N}.
		\end{align*}
	 Choosing some $\alpha_0>1$	and applying Gronwall's lemma, 
		we get \begin{align*}
			\mathbb{E}\[	\mid 	X^N_i(t\wedge \Theta_{R})-\bar{X}_i(t\wedge \Theta_{R})\mid^2\]\leq& \(C_{v_0,T}\eps^{\frac{1}{2}}+\|K_\eps\|_{B^{2+\alpha_0}_{1,2}}T\frac{C_{R,\alpha_0}}{N}\)\exp\{(C_\sigma+R)T\}.
		\end{align*}
	
		Let $N\rightarrow \infty$ and then $\eps\rightarrow 0,$
		we have $\lim\limits_{N\rightarrow \infty}\mathbb{E}\[\sup_{s \in [0, T]}	\mid 	X^N_i(s\wedge \Theta_{R})-\bar{X}_i(s\wedge \Theta_{R})\mid^2\]=0,$ for every $R>0.$
			Since the compactness of $\mathbb{T}^2$ ensures that $\sup_N\mathbb{E}[\mid X_i^N(t)-\bar{X}_i(t)\mid^2]<\infty,$
		we arrive at the desired result.
		\begin{align*}
			&\lim\limits_{N\rightarrow \infty}\mathbb{E}[\mid X_i^N(t)-\bar{X}_i(t)\mid^2]
			\\=&\lim\limits_{N\rightarrow \infty}\sum_{n=0}^{\infty}\mathbb{E}\[\mid 	X^N_i(t\wedge \Theta_{n+1})-\bar{X}_i(t\wedge \Theta_{n+1})\mid^2I{\{n\leq H_T(v)<n+1\}}\]
				\\=&\sum_{n=0}^{\infty}\lim\limits_{N\rightarrow \infty}\mathbb{E}\[\mid 	X^N_i(t\wedge \Theta_{n+1})-\bar{X}_i(t\wedge \Theta_{n+1})\mid^2I{\{n\leq H_T(v)<n+1\}}\]
			\\\leq&	\sum_{n=0}^{\infty}\lim\limits_{N\rightarrow \infty}\mathbb{E}\[\mid 	X^N_i(t\wedge \Theta_{n+1})-\bar{X}_i(t\wedge \Theta_{n+1})\mid^2\]=0.
		\end{align*}
	\end{proof}

		Now, let $(\mathcal{G}_t)_{t \in [0,T]}$ be the natural filtration of the process $(\tilde{v},\tilde{\eta}, \tilde{\mathcal{M}},\tilde{W})$. That means that,  for each $t \in [0,T],$ $\mathcal{G}_t$ is the smallest $\sigma$-algebra such that $\tilde{v}(s) : \tilde{\Omega} \to L^{2},$ $\tilde{\eta}(s) : \tilde{\Omega} \to H^{\alpha}, \alpha=-3-\frac{1}{k}, k\in \mathbb{N},$ $\tilde{\mathcal{M}}(s) : \tilde{\Omega} \to H^{\gamma},\gamma
	=-2-\frac{1}{k}, k\in \mathbb{N},$ and $\tilde{W}(s) : \tilde{\Omega} \to \mathbb{R}$ are measurable for all $s \in [0,t]$. 	 
	Let $\mathcal{N} \assign\{ M \in \tilde{\mathcal{F}} \mid \tilde{\mathbb{P}}(M) = 0 \}$. We will consider the augmented filtration $(\tilde{\mathcal{F}}_t)_{t \in [0,T]}$ which is defined by
	\begin{align*}
		\tilde{\mathcal{F}}_t &\assign \bigcap_{s > t} \sigma\left( \mathcal{G}_s \cup \mathcal{N} \right), ~t\in[0,T]. 
	\end{align*}
	The augmented filtration $(\tilde{\mathcal{F}}_t)_t$ is a normal filtration.
	For $N \in \N,$ we do the same construction to define the natural filtration $(\mathcal{G}_t^N)_{t\in[0,T]}$ and the corresponding augmented filtration $(\tilde{\mathcal{F}}_t^N)_{t\in[0,T]}$ of $(\tilde{v}^N,\tilde{\eta}^N,
	\tilde{\mathcal{M}}^N,\tilde{W}^N).$ 
	
	We now focus on the conditional law of $\{\tilde{\mathcal{M}}_t,t\in [0,T]\}$ with respect to the  environmental noise $\mathcal{F}^{\tilde{W}}_T.$
	The basic idea,  inspired by \cite{kurtz2004stochastic}, is to consider a similar form of “$\{\mathcal{M}^N_t,t\in [0,T]\}$"  derived from the conditionally i.i.d. particles $\{\bar{X}_i\}_{i\geq1}$ of \eqref{eqt:ncopy}, instead of  the  interacting particle system \eqref{eqt:vortex}. This requires  the strong convergence result  of $X^N_i$ to $\bar{X}_i,$ established in Proposition \ref{prop:strong-convergence}.  Using the classical central limit theorem, we ultimately obtain  that  the conditional  law of $\{\tilde{\mathcal{M}}_t,t\in [0,T]\}$  given by \eqref{eqt:lem4.2}.
	\begin{lemma}\label{lem:unilaw} For every $\varphi\in C^{\infty}(\mathbb{T}^2)$ and $ 0\leq t<r+t\leq T,$ it holds $\tilde{\mathbb{P}}$-a.s. that
		\begin{align}\label{eqt:lem4.2}
			\begin{split}
				\tilde{\mathbb{E}}\[\exp\{i\langle \tilde{\mathcal{M}}_t,\varphi\rangle\}\mid \mathcal{F}^{\tilde{W}}_T\]&=\exp\bigg\{-\int_{0}^{t}\langle \mid  \nabla\varphi\mid^2,\tilde{v}_s\rangle\mathd s\bigg\},\\ \mathbb{E}\[\exp{i\left\langle \varphi,(\tilde{\mathcal{M}}_{t+r}-\tilde{\mathcal{M}}_{t}) \right\rangle}\mid \mathcal{F}^{\tilde{W}}_T\vee\mathcal{F}^{\tilde{\mathcal{M}}}_t\]&=\exp\bigg\{-\int_{t}^{t+r}\langle \mid \nabla\varphi\mid^2,\tilde{v}_s\rangle\mathd s\bigg\}.
			\end{split}
		\end{align}
	\end{lemma}
	\begin{proof}
		Given a random variable  $Z$ which can be written as $g(\tilde{W}),$ where $g$ is  bounded and  continuous.  We thus have 
		\begin{align*}
			\tilde{\mathbb{E}}\[\exp\{i\langle \tilde{\mathcal{M}}_t,\varphi\rangle\}Z\]&=\lim\limits_{N\rightarrow \infty}\tilde{\mathbb{E}}\[\exp\{i\langle \tilde{\mathcal{M}}^N_t,\varphi\rangle\}g(\tilde{W}^N)\]
			\\&=\lim\limits_{N\rightarrow \infty}\mathbb{E}\[\exp\{i\langle \mathcal{M}^N_t,\varphi\rangle\}g(W)\]
			\\&=\lim\limits_{N\rightarrow \infty}\mathbb{E}\[(\exp\{i\langle \mathcal{M}^N_t,\varphi\rangle\}-\exp\{i\langle\bar{\mathcal{M}}^N_t,\varphi\rangle\})g(W)\]\\&+\lim\limits_{N\rightarrow \infty}\mathbb{E}\[\exp\{i\langle\bar{\mathcal{M}}^N_t,\varphi\rangle\}g(W)\].
		\end{align*}
		where $\langle\bar{\mathcal{M}}^N_t,\varphi\rangle$ is defined as $\frac{\sqrt{2}}{\sqrt{N}}  \sum_{i = 1}^N \int^t_0 \nabla \varphi (\bar{X}_{i}) \cdot
		\mathd B_s^i ,$ with $X_i^N$ in $\mathcal{M}$ replaced by $\bar{X}_i$ from the conditional Mckean Vlasov equation \eqref{eqt:ncopy}.
		
			Then,
		\begin{align*}
			&\mathbb{E}\bigg|(\exp\{i\langle \mathcal{M}^N_t,\varphi\rangle\}-\exp\{i\langle\bar{\mathcal{M}}^N_t,\varphi\rangle\})g(W)\bigg|\\\lesssim&\bigg[\mathbb{E}\bigg|\frac{\sqrt{2}}{\sqrt{N}}  \sum_{i = 1}^N \int^t_0 \nabla \varphi (X^N_{i}) \cdot
			\mathd B_s^i -\frac{\sqrt{2}}{\sqrt{N}}  \sum_{i = 1}^N \int^t_0 \nabla \varphi (\bar{X}_{i}) \cdot
			\mathd B_s^i \bigg|^2\bigg]^{\frac{1}{2}}
			\\\lesssim &\[\mathbb{E}\int_{0}^{t}|\nabla \varphi (X^N_{1})-\nabla \varphi (\bar{X}_{1})|^2\mathd s\]^{\frac{1}{2}}
		\end{align*}
		where the first inequality uses the fact that $g$ is bounded and H\"older's inequality, while the second inequality follows from Burkholder-Davis-Gundy's inequality and the symmetry of the law of $(X^N,\bar{X}^N).$ 
		Using Proposition \ref{prop:strong-convergence}, we have 
		\begin{align*}
			\lim\limits_{N\rightarrow \infty}\mathbb{E}\bigg|(\exp\{i\langle \mathcal{M}^N_t,\varphi\rangle\}-\exp\{i\langle\bar{\mathcal{M}}^N_t,\varphi\rangle\})g(W)\bigg|=0.
		\end{align*}
		Notice that $\{(\bar{X}_i,B_i)\}_{i\geq 1}$ is conditionally i.i.d. with respect to $\mathcal{F}^{W}_T,$ we thus have
		\begin{align*}
			&\lim\limits_{N\rightarrow \infty}\mathbb{E}\[\exp\{i\langle\bar{\mathcal{M}}^N_t,\varphi\rangle\}g(W)\]\\&=\lim\limits_{N\rightarrow \infty}\mathbb{E}\bigg[\mathbb{E}\bigg[\exp\bigg\{\frac{\sqrt{2}}{\sqrt{N}}\int_{0}^{t}\nabla \varphi (\bar{X}_{1}) \cdot
			\mathd B_s^1\bigg\} \mid \mathcal{F}^{W}_T\bigg]^Ng(W)\bigg]\\&=\mathbb{E}\bigg[\exp\bigg\{-\int_{0}^{t}\langle \mid  \nabla\varphi\mid^2 ,v_s\rangle\mathd s\bigg\} g(W)\bigg]=\tilde{\mathbb{E}}\bigg[\exp\bigg\{-\int_{0}^{t}\langle \mid  \nabla\varphi\mid^2,\tilde{v}_s\rangle\mathd s \bigg\}Z\bigg].
		\end{align*}
		where the second equality is a consequence of  central limit theorem \cite[Theorem 3.4.1]{durrett2019probability}. 
		By Lusin theorem, $Z$ can be extended to be any bounded $\mathcal{F}^{\tilde{W}}_T$ measureable random variables. The proof of the second identity is similar.
	\end{proof}
Before proceeding, let's recall the following abbreviation defined in   Corollary \ref{lemma uni mu2}, for $t\in[0,T]$  and $f\in L^2([0,T];H^4(\mathbb{T}^2)),$
\begin{align}\label{eqdef:R}
\mathcal{R}_t(f)=\frac{1}{m}\exp\bigg\{-\int_{0}^{t}m\|f_s\|^2_{H^4}\mathd s\bigg\},
\end{align} 
where the deterministic constant $m>1$   depends on $\|v_0\|_{L^2(\mathbb{T}^2)},$$\underset{x\in \mathbb{T}^2 }{\inf}v_0$ and $T.$ 
 
 To prepare the proof for the convergence for the  interacting term, 
 \begin{align*}
 	&\int^t_0\tilde{  \mathcal{K}}^N_s
 	(\varphi) - \langle \tilde{v}_s K \ast \tilde{\eta}_s + \tilde{\eta}_s K \ast
 	\tilde{v}_s, \nabla \varphi \rangle \mathd \nocomma s \xrightarrow{N\rightarrow\infty}0,
 \end{align*} we establish the following modified limit about fluctuation measures.
\begin{lemma}
	\label{prop converge}For every $\alpha > 1$, it holds that
	\begin{align}
		\begin{split}
			\tilde{\mathbb{E}} \int^T_0 \big\|	\mathcal{R}_T(\tilde{v}^N) \tilde{\eta}_t^N &-	\mathcal{R}_T(\tilde{v})\tilde{\eta}_t \big\|_{H^{- \alpha}} \mathd t
			\xrightarrow{N \rightarrow \infty} 0. \label{coro 1}
		\end{split}
	\end{align}
\end{lemma}

\begin{proof}
	By H\"older's inequality, we have $$\tilde{\mathbb{E}}\[\int_0^T  	\mathcal{R}_T(\tilde{v})\|\tilde{\eta}_t \|_{H^{- \alpha}} \mathd t\]^2
		\leq \sqrt{T}\[\tilde{\mathbb{E}}\int_0^T  	\mathcal{R}^2_T(\tilde{v})\|\tilde{\eta}_t \|^2_{H^{- \alpha}} \mathd t\]^{\frac{1}{2}}.$$
Recall the definitions about $\mathcal{R}_T(\tilde{v}^N),\mathcal{R}_T(\tilde{v})$ in \eqref{eqdef:R}	and notice that $\mathcal{R}^2_T(\tilde{v}^N)\leq\mathcal{R}_T(\tilde{v}^N)\leq\mathcal{R}_t(\tilde{v}^N),$ we then infer that $\forall \alpha > 1,$
	\begin{align*}
		\tilde{\mathbb{E}}\int_0^T  	\mathcal{R}^2_T(\tilde{v})\|\tilde{\eta}_t \|^2_{H^{- \alpha}} \mathd t&\leq 	\sup_N\tilde{\mathbb{E}}\int_0^T 	\mathcal{R}_T(\tilde{v}^N)\|  \tilde{\eta}_t^N \|^2_{H^{- \alpha}} \mathd t\\&\leqslant
		T\sup_{t \in [0, T]} \sup_N \tilde{\mathbb{E}} \[	\mathcal{R}_T(\tilde{v}^N)\|\tilde{\eta}_t^N \|^2_{H^{- \alpha}}\] 
		\\&\leq T\sup_{t \in [0, T]} \sup_N \mathbb{E} \[	\mathcal{R}_t(v)\|\eta_t^N \|^2_{H^{- \alpha}}\] <
		\infty, \nonumber
	\end{align*}
	where we used Corollary \ref{lemma uni mu2} to get the last inequality.
	
	The above content provides the uniform  integrability of \begin{align*}
	\int_{0}^{T}\big\| \mathcal{R}_T(\tilde{v}^N)\tilde{\eta}_t^N& - \mathcal{R}_T(\tilde{v})\tilde{\eta}_t \big\|_{H^{- \alpha}}\mathd t.
	\end{align*}
	Thus  the  convergence  \begin{align*}
	\int_{0}^{T}\big\|\mathcal{R}_T(\tilde{v}^N)\tilde{\eta}_t^N - \mathcal{R}_T(\tilde{v})\tilde{\eta}_t \big\|_{H^{- \alpha}}\mathd t\xrightarrow{N \rightarrow \infty}0,\quad \tilde{\mathbb{P}}-a.s.
	\end{align*}
	 leads to {\eqref{coro 1}}.
\end{proof}

We now deal with the interacting term.

\begin{lemma}
	\label{lemma limit no}For each $\varphi \in C^{\infty} (\mathbb{T}^2)$, it
	holds that
	\begin{align}
		\tilde{\mathbb{E}} \bigg( \sup_{t \in [0, T]} &\bigg| \int^t_0 \mathcal{R}_T(\tilde{v}^N)\tilde{  \mathcal{K}}^N_s
		(\varphi) - \mathcal{R}_T(\tilde{v})\langle \tilde{v}_s K \ast \tilde{\eta}_s + \tilde{\eta}_s K \ast
		\tilde{v}_s, \nabla \varphi \rangle \mathd \nocomma s \bigg|^{\frac{1}{2}} \bigg)\xrightarrow{N \rightarrow \infty} 0, \nonumber
	\end{align}
	where \begin{align*}
		\tilde{  \mathcal{K}}_t^N (\varphi)&  \assign  \sqrt{N}  \langle \nabla \varphi, K \ast
		\tilde{\mu}_N (t) \tilde{\mu}_N (t) \rangle - \sqrt{N}  \langle \nabla \varphi, \tilde{v}^N _t K
		\ast \tilde{v}^N _t  \rangle ,
	\end{align*}
	and 
	\begin{align*}
		\tilde{\mu}_N& \assign \frac{1}{\sqrt{N}}\tilde{\eta}^N+\tilde{v}^N .
	\end{align*}
Moreover, there exists a subsequence $\{N_k\}_{k\geq 1}$ ( still denoted by $\{N\}_{N\in \mathbb{N}}$  for simplicity) such that 
\begin{align*}
	&\int^t_0\tilde{  \mathcal{K}}^N_s
	(\varphi) - \langle \tilde{v}_s K \ast \tilde{\eta}_s + \tilde{\eta}_s K \ast
	\tilde{v}_s, \nabla \varphi \rangle \mathd \nocomma s \xrightarrow{N\rightarrow\infty}0,\quad \forall t\in [0,T],\quad \tilde{	\mathbb{P}}-a.s..
\end{align*}
\end{lemma}

\begin{proof} Notice that 
	\begin{align*}
	\sqrt{N} \big(\tilde\mu_N K * \tilde\mu_N - \tilde{v}^N K * \tilde{v}^N  \big) &= \tilde{v}^N K * \tilde{\eta}^N + \tilde{\eta}^N K * \tilde{v}^N + \frac{1}{\sqrt{N}} \,  \tilde{\eta}^N K * \tilde{\eta}^N.
	\end{align*}
	Consequently,  for each $\varphi \in C^{\infty}$, we have
	\begin{align}
		&\sup_{t \in [0, T]} \big| \int^t_0 \mathcal{R}_T(\tilde{v}^N) \tilde{\mathcal{K}}^N_s (\varphi)- \mathcal{R}_T(\tilde{v})\langle
		\tilde{v}_s K \ast \tilde{\eta}_s + \tilde{\eta}_s K \ast \tilde{v}_s, \nabla \varphi
		\rangle \mathd \nocomma s \bigg|^{\frac{1}{2}}\nonumber\\& \leqslant \sqrt{J_1^N} (\varphi) + \sqrt{J_2^N}
		(\varphi), \label{limit no 1}
	\end{align}
	where
	
	\begin{align}
		J_1^N (\varphi) \assign & \sqrt{N} \mathcal{R}_T(\tilde{v}^N) \int^T_0 \big| \big\langle \nabla \varphi K \ast
		(\tilde{\mu}_N (t) - \tilde{v}^N_t), \tilde{\mu}_N (t) - \tilde{v}^N_t \big\rangle \big| \mathd t,
		\nonumber\\
		J_2^N (\varphi) \assign& \int^T_0 \big| \mathcal{R}_T(\tilde{v}^N) \big\langle \tilde{v}^N_t K \ast \tilde{\eta}^N_t +
		\tilde{\eta}^N_t K \ast \tilde{v}^N_t, \nabla \varphi \big\rangle 
		\\&-\mathcal{R}_T(\tilde{v}) \big\langle
		\tilde{v}_t K \ast \tilde{\eta}_t + \tilde{\eta}_t K \ast \tilde{v}_t, \nabla \varphi
		\big\rangle \big| \mathd t. \nonumber
	\end{align}
	On one hand, Lemma \ref{lemma uni mu2} yields the following estimate
	\begin{align*}
		\tilde{\mathbb{E}}J_1^N (\varphi) & \leqslant T \nocomma \sqrt{N} \sup_{t \in [0,
			T]} \tilde{\mathbb{E}} \big|\mathcal{R}_T(\tilde{v}^N) \langle \nabla \varphi K \ast (\tilde{\mu}_N (t) - \tilde{v}^N_t),
		\tilde{\mu}_N (t) - \tilde{v}^N_t \rangle \big|
		\\ &= T \sqrt{N} \sup_{t \in [0,
			T]} {\mathbb{E}} \big|\mathcal{R}_T(v) \langle \nabla \varphi K \ast ({\mu}_N (t) - v_t),
		{\mu}_N (t) - v_t \rangle\big|  
		\\&\leq T \sqrt{N} \sup_{t \in [0,
			T]} {\mathbb{E}} \big|\mathcal{R}_t(v) \langle \nabla \varphi K \ast ({\mu}_N (t) - v_t),
		{\mu}_N (t) - v_t \rangle\big|  \lesssim N^{- \frac{1}{2}} \xrightarrow{N \rightarrow \infty}
		0.
	\end{align*}
	With H\"older's inequality, we infer
$	\tilde{\mathbb{E}}\sqrt{J_1^N (\varphi)}\leq\[\tilde{\mathbb{E}}(J_1^N (\varphi))\]^{\frac{1}{2}}\xrightarrow{N \rightarrow \infty}0.$
	On the other hand, we have
	\begin{align}
		\tilde{\mathbb{E}}\sqrt{J_2^N (\varphi)} &\leqslant \tilde{\mathbb{E}} \[\sum_{i=1}^{4}\int^T_0 H_i 
		\mathd t\]^{\frac{1}{2}}. \label{limit no2}
	\end{align}
	where
	\begin{align*}
		H_1\assign& \big| \big\langle
		\tilde{v}_t K \ast \big(\mathcal{R}_T(\tilde{v}^N)\tilde{\eta}^N_t -  \mathcal{R}_T(\tilde{v})\tilde{\eta}_t\big), \nabla \varphi \big\rangle \big|,
		\\H_2\assign&\big|
		\big\langle \big(\mathcal{R}_T(\tilde{v}^N)\tilde{\eta}^N_t -  \mathcal{R}_T(\tilde{v})\tilde{\eta}_t\big) K \ast \tilde{v}_t, \nabla \varphi \big\rangle \big|,
		\\H_3\assign&\mathcal{R}_T(\tilde{v}^N)\big| \big\langle
		(\tilde{v}^N_t-\tilde{v}_t) K \ast \tilde{\eta}^N_t , \nabla \varphi \big\rangle \big|,
		\\H_4\assign&\mathcal{R}_T(\tilde{v}^N) \big|
		\big\langle \tilde{\eta}^N_t K \ast (\tilde{v}^N_t-\tilde{v}_t), \nabla \varphi \big\rangle \big|.
	\end{align*}

	For each $t \in [0, T]$, we deduce that for every $\alpha> 1$ that
	\begin{align*}
		H_1=& \big| \big\langle K(-\cdot) \ast (\tilde{v}_t \nabla \varphi), \mathcal{R}_T(\tilde{v}^N)\tilde{\eta}^N_t -  \mathcal{R}_T(\tilde{v})\tilde{\eta}_t \big\rangle \big|\\ \leqslant& \big\| \mathcal{R}_T(\tilde{v}^N)\tilde{\eta}^N_t -  \mathcal{R}_T(\tilde{v})\tilde{\eta}_t \big\|_{H^{- \alpha}}  \|
		K(-\cdot) \ast (\tilde{v}_t \nabla \varphi)\|_{H^{\alpha}}, \nonumber
	\end{align*}
	where
	\begin{align}
		K(-\cdot) * g (x) := \int K(y-x) g(y) \mathd y.\label{notation:k}
	\end{align}
	Recall that $K\lesssim\frac{1}{|x|}\in L^1(\mathbb{T}^2),$ 
	we apply 
	Lemma \ref{lemma convolution} with $p = p_1 = q = 2$ and Lemma \ref{lemma:infity} to get 
	\begin{align*}
		H_1& \lesssim \big\| \mathcal{R}_T(\tilde{v}^N)\tilde{\eta}^N_t -  \mathcal{R}_T(\tilde{v})\tilde{\eta}_t\big\|_{H^{- \alpha}}
		\cdot \|K\|_{L^1}  (\|
		\tilde{v}_t \|_{H^{\alpha}} \| \nabla \varphi \|_{L^{\infty}} +\|
		\tilde{v}_t \|_{L^{\infty}} \| \nabla \varphi \|_{H^{\alpha}}),
		\\ &\lesssim \big\| \mathcal{R}_T(\tilde{v}^N)\tilde{\eta}^N_t -  \mathcal{R}_T(\tilde{v})\tilde{\eta}_t\big\|_{H^{- \alpha}}
		\cdot \|K\|_{L^1}  (\|
		\tilde{v}_t \|_{H^{\alpha}} \| \nabla \varphi \|_{L^{\infty}} +\|
		\tilde{v}_t \|_{H^2} \| \nabla \varphi \|_{H^{\alpha}}),
		\nonumber
	\end{align*}
	where we use the Sobolev embedding theorem to get the last inequality.
	
		Similarly, we have 	\begin{align*}
		H_2\leqslant& \big\|\mathcal{R}_T(\tilde{v}^N)\tilde{\eta}^N_t-  \mathcal{R}_T(\tilde{v})\tilde{\eta}_t\big\|_{H^{- \alpha}}  \| \nabla \varphi
		\cdummy K \ast \tilde{v}_t \|_{H^{\alpha}}\nonumber
		\\\lesssim& \big\| \mathcal{R}_T(\tilde{v}^N)\tilde{\eta}^N_t -  \mathcal{R}_T(\tilde{v})\tilde{\eta}_t\big\|_{H^{- \alpha}}
		\cdot \|K\|_{L^1}  (\|
		\tilde{v}_t \|_{H^{\alpha}} \| \nabla \varphi \|_{L^{\infty}} +\|
		\tilde{v}_t \|_{H^{2}} \| \nabla \varphi \|_{H^{\alpha}}),
		\nonumber
	\end{align*}
	\begin{align*}
		H_3&=\mathcal{R}_T(\tilde{v}^N)\big| \big\langle K(-\cdot) \ast ((\tilde{v}^N_t-\tilde{v}_t) \nabla \varphi), \tilde{\eta}^N_t \big\rangle \big| 
		\\&\leqslant \mathcal{R}_T(\tilde{v}^N)\| \tilde{\eta}^N_t \|_{H^{- \alpha}}  \|
		K(-\cdot) \ast ((\tilde{v}^N_t-\tilde{v}_t) \nabla \varphi)\|_{H^{\alpha}}, 
		\\&\lesssim \mathcal{R}_T(\tilde{v}^N)\| \tilde{\eta}^N_t \|_{H^{- \alpha}} \cdot\|K\|_{L^1}  (\|
		\tilde{v}^N_t-\tilde{v}_t\|_{H^{\alpha}} \| \nabla \varphi \|_{L^{\infty}} +\|
		\tilde{v}^N_t-\tilde{v}_t \|_{H^{2}} \| \nabla \varphi \|_{H^{\alpha}}),
	\end{align*}
	and
	\begin{align*}
		H_4&\leq\mathcal{R}_T(\tilde{v}^N)\| \tilde{\eta}^N_t  \|_{H^{- \alpha}}  \| \nabla \varphi
		\cdummy K \ast (\tilde{v}^N_t-\tilde{v}_t) \|_{H^{\alpha}}.
		\\&\lesssim  \mathcal{R}_T(\tilde{v}^N)\| \tilde{\eta}^N_t \|_{H^{- \alpha}} \cdot\|K\|_{L^1}  (\|
		\tilde{v}^N_t-\tilde{v}_t\|_{H^{\alpha}} \| \nabla \varphi \|_{L^{\infty}} +\|
		\tilde{v}^N_t-\tilde{v}_t \|_{H^{2}} \| \nabla \varphi \|_{H^{\alpha}}) .
	\end{align*}

	Next, we substitute these estimates into equation {\eqref{limit no2}} with $\alpha=2$  and apply H\"older's inequality multiple times to obtain
	\begin{align}\label{eqt:interact2}
	\begin{split}
			\tilde{	\mathbb{E}}\sqrt{J_2^N (\varphi)} &\lesssim_{\varphi}\|K\|_{L^1}^{\frac{1}{2}}\tilde{	\mathbb{E}}\[\int_{0}^{T} \mathcal{R}_T(\tilde{v}^N)\| \tilde{\eta}^N_t \|_{H^{- 2}} \|\tilde{v}^N_t-\tilde{v}_t \|_{H^{2}} \mathd t\]^{\frac{1}{2}}
		\\&+\|K\|_{L^1}^{\frac{1}{2}}\tilde{	\mathbb{E}}\[\sup_{t \in [0, T]} \| \tilde{
			v}_t \|_{H^{2}}\int^T_0\big \| \mathcal{R}_T(\tilde{v}^N)\tilde{\eta}^N_t -  \mathcal{R}_T(\tilde{v})\tilde{\eta}_t\|_{H^{- 2}} \mathd t\]^{\frac{1}{2}}
		\\&\lesssim_{\varphi} \|K\|_{L^1}^{\frac{1}{2}}\tilde{	\mathbb{E}}\(\[\int_{0}^{T} \mathcal{R}_T(\tilde{v}^N)\| \tilde{\eta}^N_t \|^2_{H^{- 2}}\mathd t\]^{\frac{1}{4}} \[\int_{0}^{T}\|\tilde{v}^N_t-\tilde{v}_t \|^2_{H^{2}} \mathd t\]^{\frac{1}{4}}\)
		\\&+\|K\|_{L^1}^{\frac{1}{2}}
		\[\tilde{	\mathbb{E}}	{\sup_{t \in [0, T]} \| \tilde{
				v}_t \|_{H^{2}}}\]^{\frac{1}{2}}
		\[\tilde{	\mathbb{E}} \int^T_0 \big\| \mathcal{R}_T(\tilde{v}^N)\tilde{\eta}^N_t -  \mathcal{R}_T(\tilde{v})\tilde{\eta}_t\big\|_{H^{- 2}} \mathd t\]^{\frac{1}{2}}
		\\&\lesssim_{\varphi}\|K\|_{L^1}^{\frac{1}{2}}\[T\sup_{t\in[0,T]}\tilde{	\mathbb{E}} \mathcal{R}_T(\tilde{v}^N)\| \tilde{\eta}^N_t \|^2_{H^{- 2}}\]^{\frac{1}{4}} \[\tilde{	\mathbb{E}}\[\int_{0}^{T}\|\tilde{v}^N_t-\tilde{v}_t \|^2_{H^{2}} \mathd t\]^{\frac{1}{2}}\]^{\frac{1}{2}}
		\\&+ \|K\|_{L^1}^{\frac{1}{2}}
		\[\tilde{	\mathbb{E}}	{\sup_{t \in [0, T]} \| \tilde{
				v}_t \|_{H^{2}}}\]^{\frac{1}{2}}
		\[\tilde{	\mathbb{E}} \int^T_0 \big\| \mathcal{R}_T(\tilde{v}^N)\tilde{\eta}^N_t -  \mathcal{R}_T(\tilde{v})\tilde{\eta}_t\big\|_{H^{- 2}} \mathd t\]^{\frac{1}{2}}.
	\end{split}
	\end{align}
	Observe that
	\begin{align*}
		\tilde{	\mathbb{E}}\[\int_{0}^{T}\|\tilde{v}_t \|^2_{H^{2}} \mathd t\]\lesssim \sup_{N}\tilde{	\mathbb{E}}\[\int_{0}^{T}\|\tilde{v}^N_t\|^2_{H^{2}} \mathd t\]=\mathbb{E}\[\int_{0}^{T}\|v_t\|^2_{H^{2}} \mathd t\]<\infty,
	\end{align*}
	which establishes the uniform integrability of $$\[\int_{0}^{T}\|\tilde{v}^N_t-\tilde{v}_t \|^2_{H^{2}} \mathd t\]^{\frac{1}{2}}.$$ 
	Therefore,  the  convergence of $$\int_{0}^{T}\|\tilde{v}^N_t-\tilde{v}_t \|^2_{H^{2}} \mathd t,\quad \tilde{\mathbb{P}}-a.s.$$  together with Corollary \ref{lemma uni mu2}, leads to the convergence of the first term in \eqref{eqt:interact2}
	to $0.$ 
	Since $\sup_{t \in [0, T]} \| \cdot \|_{H^{2}}$ is lower semi-continuous on		$\mathcal{V}= C([0,T];L^2(\mathbb{T}^2))\cap L^2([0,T];H^4(\mathbb{T}^2)),$
	by Fatou's Lemma,	we have \begin{align}\label{eq:v-u-beta'}
		\tilde{	\mathbb{E}}	{\sup_{t \in [0, T]} \| \tilde{
				v}_t \|_{H^{2}}}\leq\liminf_{N\rightarrow\infty}	\tilde{	\mathbb{E}}	{\sup_{t \in [0, T]} \| \tilde{
				v}^N_t \|_{H^{2}}}=\mathbb{E}\sup_{t \in [0, T]} \| 
		v_t \|_{H^{2}}<\infty.
	\end{align}
	Thus, the second term in \eqref{eqt:interact2} converges to $0$ by Lemma \ref{prop converge},
	completing the proof of the first claim.  
	
	Furthermore, the first claim implies that there exists a subsequence $\{N_k\}_{k\geq 1}$ ( still denoted by $\{N\}_{N\in \mathbb{N}}$  for simplicity) such that 
	
	\begin{align*}
		&\int^t_0\mathcal{R}_T(\tilde{v}^N)\tilde{  \mathcal{K}}^N_s
		(\varphi) - \mathcal{R}_T(\tilde{v})\langle \tilde{v}_s K \ast \tilde{\eta}_s + \tilde{\eta}_s K \ast
		\tilde{v}_s, \nabla \varphi \rangle \mathd \nocomma s \xrightarrow{N\rightarrow\infty}0,\quad \forall t\in [0,T],\quad \tilde{	\mathbb{P}}-a.s..
	\end{align*}
	Since	$\mathcal{R}^{-1}_T(\tilde{v}^N)=m\exp\{\int_{0}^{T}m\|\tilde{v}^N_s\|^2_{H^4}\mathd s\}$ converges to $\mathcal{R}^{-1}_T(\tilde{v})=m\exp\{\int_{0}^{T}m\|\tilde{v}_s\|^2_{H^4}\mathd s\}, \tilde{\mathbb{P}}$-a.s.,  we then conclude that 
	\begin{align*}
		\int^t_0\tilde{  \mathcal{K}}^N_s
		(\varphi) - \langle \tilde{v}_s K \ast \tilde{\eta}_s + \tilde{\eta}_s K \ast
		\tilde{v}_s, \nabla \varphi \rangle \mathd \nocomma s &\xrightarrow{N\rightarrow\infty}0,\quad \forall t\in [0,T],\quad \tilde{	\mathbb{P}}-a.s..
	\end{align*}
	This completes the proof.
\end{proof}
	Now,  we are in the position to establish the following  existence result, using the  martingale approach  as in \cite{hofmanova2017weak,dabrock2021existence}.  
	\begin{theorem}\label{thm:existence}
	$\big(\tilde{\Omega}, \tilde{\mathcal{F}},(\tilde{\mathcal{F}}_t)_{t\in[0,T]},
	\tilde{\mathbb{P}},\(\tilde{\eta}_t,\tilde{\mathcal{M}}_t,\tilde{W}_t\)_{t\in[0,T]}\big)$  	 is a probabilistically weak solution to the fluctuation SPDE \eqref{LimitSpde}.
	\end{theorem}

Firstly,	we establish the convergence  of the drift terms. As a consequence of Proposition \ref{pro:skorokhod}, for every $\varphi \in C^{\infty},$ it holds $ \tilde{	\mathbb{P}}$-a.s.  that for $\forall t\in [0,T],$
\begin{align*}
	\left(\langle\tilde{\mathcal{M}}^N_t ,\varphi)\rangle, \langle \tilde{\eta}^N_t, \varphi \rangle , \langle \tilde{\eta}^N_0, \varphi \rangle ,
	\int^t_0 \langle  \Delta \varphi, \tilde{\eta}^N_s \rangle \mathd \nocomma s, \frac{1}{2}\int_0^t  \left\langle\sigma\cdot\nabla\(\sigma\cdot\nabla\varphi\),\tilde{\eta}^N_s\right\rangle \mathd s\right)
\end{align*}
converges to 
\begin{align*}
	\left(\langle \tilde{\mathcal{M}}_t ,\varphi\rangle,  \langle \tilde{\eta}_t, \varphi \rangle , \langle \tilde{\eta}_0, \varphi \rangle ,
	\int^t_0 \langle  \Delta \varphi, \tilde{\eta}_s \rangle \mathd \nocomma s, \frac{1}{2}\int_0^t  \left\langle\sigma\cdot\nabla\(\sigma\cdot\nabla\varphi\),\tilde{\eta}_s\right\rangle \mathd s\right). 
\end{align*}
Furthermore,  recall Lemma \ref{lemma limit no}, there exists a subsequence $\{N_k\}_{k\geq 1}$ ( still denoted by $\{N\}_{N\in \mathbb{N}}$  for simplicity) such that 
\begin{align*}
	\int^t_0\tilde{  \mathcal{K}}^N_s
	(\varphi) - \langle \tilde{v}_s K \ast \tilde{\eta}_s + \tilde{\eta}_s K \ast
	\tilde{v}_s, \nabla \varphi \rangle \mathd \nocomma s &\xrightarrow{N\rightarrow\infty}0,\quad \forall t\in [0,T],\quad \tilde{	\mathbb{P}}-a.s..
\end{align*}

We then identify $\tilde{v}$ is the unique probabilistically strong solution to the mean field limit equation \eqref{eqt:mean}.
\begin{proposition}\label{prop:w}
The process $\{\tilde{W}_t,t\in[0,T]\}$ is a 1-dimensional $(\tilde{\mathcal{F}}_t)_{t\in[0,T]}$ Brownian motion. 
Moreover, $\tilde{v}$ is the unique probabilistically strong solution to the mean field limit equation \eqref{eqt:mean}, in the sense defined by  Definition \ref{def:mean-strong} and \ref{def:uni-mean}.
\end{proposition}
			\begin{proof}		
				For $0\leq s<t \leq T,$ let $\gamma :\mathcal{H}_s\to \R$ be a bounded and continuous function, where
				\begin{align*}
					\mathcal{H}_s\assign &C([0,s]; L^2) \times \cap_{k\in\mathbb{N}}C([0,s]; H^{-3-\frac{1}{k}} )\\&\times \cap_{k\in\mathbb{N}}C([0,s]; H^{-2-\frac{1}{k}})\times C([0,s]; \mathbb{R}). 
				\end{align*}
					We will use the abbreviations
				\begin{align}\label{not:gamma}
					\begin{split}
						\gamma^N\assign & \gamma\left( v_{[0,s]},\eta^N_{[0,s]},
						\mathcal{M}^N_{[0,s]},W_{[0,s]} \right),\\
						\tilde{\gamma}^N \assign & \gamma\left( \tilde{v}^N_{[0,s]},\tilde{\eta}^N_{[0,s]},
						\tilde{\mathcal{M}}^N_{[0,s]},\tilde{W}^N_{[0,s]} \right), \\
						\tilde{\gamma} \assign & \gamma\left( \tilde{v}_{[0,s]},\tilde{\eta}_{[0,s]},
						\tilde{\mathcal{M}}_{[0,s]},\tilde{W}_{[0,s]} \right).
					\end{split}
				\end{align}
				Since the joint law of $(\tilde{v}^N,\tilde{\eta}^N,
				\tilde{\mathcal{M}}^N,\tilde{W}^N)$ coincides with the joint law of $(v,\eta^N,\mathcal{M}^N,W)$, we infer  that
			\begin{align}
					\label{eq:wiener}
			\begin{split}
					&\tilde{\mathbb{E}} \left( \tilde{\gamma}^N \left( \tilde{W}^N(t) - \tilde{W}^N(s) \right) \right)\\=&\mathbb{E} \left( \gamma^N \left( W(t) - W(s) \right) \right) =0,
				\\&\tilde{\mathbb{E}} \left( \tilde{\gamma}^N {\tilde{W}^N(t)}{\tilde{W}^N(t)}\right)-\tilde{\mathbb{E}}\left( \tilde{\gamma}^N {\tilde{W}^N(s)}{\tilde{W}^N(s)}\right)\\=& {\mathbb{E}} \left( \gamma^N {W(t)}{{W}(t)} \right)  - {\mathbb{E}} \left(\gamma^N {{W}(s)}{{W}(s)} \right) =t-s.
			\end{split}
			\end{align}
Using  Burkholder-Davis-Gundy inequality for $W,$ we obtain the following uniform bound
			\begin{align*}
				\tilde{\mathbb{E}} | \tilde{W}^N(t) |^3 = \mathbb{E} | W(t) |^3 &\leq C t^{\frac{3}{2}},\quad \forall N\in \mathbb{N},
			\end{align*}
		which provides the uniform integrability of the above terms.
	We thus can pass to the limit in the equations \eqref{eq:wiener} and infer
				\begin{align}
			\label{eq:wiener-limit}
				\begin{split}
				&\tilde{\mathbb{E}} \left( \tilde{\gamma} \left( \tilde{W}(t) - \tilde{W}(s) \right) \right)=0 ,
				\\&\tilde{\mathbb{E}} \left( \tilde{\gamma} {\tilde{W}(t)}{\tilde{W}(t)} \right)  - \tilde{\mathbb{E}} \left(\tilde{\gamma} {\tilde{W}(s)}{\tilde{W}(s)} \right)=\left(t-s\right) . 
			\end{split} 
				\end{align}
			This implies that $\{\tilde{W}_t,t\in [0,T]\}$ is a  square-integrable $(\tilde{\mathcal{F}}_t)_{t\in[0,T]}$-martingale with $(\tilde{\mathcal{F}}_t)_{t\in[0,T]}$-quadratic variation   $\[W,W\]_{t}=t.$
		  By the L\'evy martingale characterization theorem,  we conclude that $\tilde{W}$ is a $(\tilde{\mathcal{F}}_t)_{t\in[0,T]}$ 1-dimensional Brownian motion. 
		  
		  Furthermore, as a consequence of Proposition \ref{pro:skorokhod},  it holds $\tilde{\mathbb{P}}$-a.s. that $(\tilde{v}^N,\tilde{W}^N)$ with the same law $\mathcal{L}(v,W),$ where $v$ is the unique probabilistically strong solution to the mean field limit equation \eqref{eqt:mean} on the previous stochastic basis $(\Omega,\mathcal{F},(\mathcal{F}^W_t)_{t\in [0,T]},\mathbb{P},W),$ converges to $(\tilde{v},\tilde{W})$ in the space $C([0,T];L^2)\cap L^2([0,T];H^4), \tilde{\mathbb{P}}$-a.s.. This implies that $\mathcal{L}(\tilde{v},\tilde{W})=\mathcal{L}(v,W).$ Since we have established that $\tilde{W}$ is $(\tilde{\mathcal{F}}_t)_{t\in[0,T]}$-1-dimensional Brownian motion and given the well-posedness of the mean field limit equation \eqref{eqt:mean} in Lemma \ref{thm:spde}, we conclude that $\tilde{v}$ is the unique probabilistically strong solution to \eqref{eqt:mean}.  
			\end{proof}
		
		We now proceed to handle the transport noise term $\sigma\cdot \nabla \tilde{\eta}_t \mathd \tilde{W}_t,$  using the approach in \cite{hofmanova2017weak}.
		For $\varphi \in C^{\infty},$ let's define
		\begin{align*}
			\tilde{Z}_{\cdot}\assign &\langle \tilde{\eta}_{\cdot}, \varphi \rangle - \langle  \tilde{\eta}_0, \varphi \rangle -
			\int^{\cdot}_0 \langle  \Delta \varphi,  \tilde{\eta}_s \rangle \mathd \nocomma s -
			\int^{\cdot}_0 \langle \nabla \varphi, v_s K \ast  \tilde{\eta}_s \rangle \mathd
			\nocomma s - \int^{\cdot}_0 \langle \nabla \varphi,  \tilde{\eta}_s K \ast v_s
			\rangle \mathd \nocomma s\\&- \mathcal{M}_{\cdot}(\varphi)-\frac{1}{2}\int_0^{\cdot} \left\langle\sigma\cdot\nabla\(\sigma\cdot\nabla\varphi\), \tilde{\eta}_s\right\rangle \mathd s.
		\end{align*}
		(We do the same constructions to define $\tilde{Z}^N,Z^N$ for $\tilde{\eta}^N,\eta^N$).
		\begin{proposition}\label{prop:transport}
			The processes   \begin{align*}
				\tilde{Z},\quad \tilde{Z}^2-\int_{0}^{t}\left\langle \sigma\cdot \nabla\varphi, \tilde{\eta}_s\right\rangle^2 \mathd s,\quad \tilde{Z}\tilde{W}-\int_{0}^{t}\left\langle \sigma\cdot \nabla\varphi, \tilde{\eta}_s\right\rangle \mathd s,
			\end{align*} indexed by $t\in[0,T],$ are $(\tilde{\mathcal{F}}_t)_{t\in[0,T]}-$local-martingale. 
		\end{proposition}
	\begin{proof}
	We recall the following notations. For $0\leq s<t \leq T,$ $\gamma :\mathcal{H}_s\to \R$ denote a bounded and continuous function, where
	\begin{align*}
		\mathcal{H}_s\assign &C([0,s]; L^2) \times \cap_{k\in\mathbb{N}}C([0,s]; H^{-3-\frac{1}{k}} )\\&\times \cap_{k\in\mathbb{N}}C([0,s]; H^{-2-\frac{1}{k}})\times C([0,s]; \mathbb{R}). 
	\end{align*}
	We introduce the following abbreviations.
	\begin{align*}
		\begin{split}
			\gamma^N\assign & \gamma\left( v_{[0,s]},\eta^N_{[0,s]},
			\mathcal{M}^N_{[0,s]},W_{[0,s]} \right),\\
			\tilde{\gamma}^N \assign & \gamma\left( \tilde{v}^N_{[0,s]},\tilde{\eta}^N_{[0,s]},
			\tilde{\mathcal{M}}^N_{[0,s]},\tilde{W}^N_{[0,s]} \right), \\
			\tilde{\gamma} \assign & \gamma\left( \tilde{v}_{[0,s]},\tilde{\eta}_{[0,s]},
			\tilde{\mathcal{M}}_{[0,s]},\tilde{W}_{[0,s]} \right).
		\end{split}
	\end{align*}
		Let $M>0$ and define 
		\begin{align*}
			\vartheta_M\assign C([0,T];R)\rightarrow[0,T], \quad f\rightarrow \inf\{t>0; |f(t)|\geq M\}
		\end{align*}
		(with the convention $\inf \emptyset=T$).
		Choosing  $\alpha_0=\frac{7}{2}$ and  noting that $\tilde{\eta}^N$ belongs to $C([0,T];H^{-\alpha_0}),$ we then deduce that for every $N\in \mathbb{N},$ $\vartheta_M(\tilde{I}^N)$ defines an $(\tilde{\mathcal{F}}_t)_{t\in[0,T]}$-stopping time and the blow up does not occur in a finite time, i.e.
		\begin{align}\label{stoptime-T}
			\sup_{M>0}\vartheta_M(\tilde{I}^N)=T\quad \tilde{\mathbb{P}}-a.s.,
		\end{align}
		where  $\tilde{I}^N(t)\assign\sup_{s \in [0, t]}\|\tilde{\eta}^N\|_{H^{-\alpha_0}}.$ The same is valid for the case $\tilde{I}(t)\assign\sup_{s \in [0, t]}\|\tilde{\eta}\|_{H^{-\alpha_0}}.$
		The stopping times $\vartheta_M(\tilde{I})$	will play the role of a localizing sequence for the processes \begin{align*}
		\tilde{Z},\quad \tilde{Z}^2-\int_{0}^{t}\left\langle \sigma\cdot \nabla\varphi, \tilde{\eta}_s\right\rangle^2 \mathd s,\quad \tilde{Z}\tilde{W}-\int_{0}^{t}\left\langle \sigma\cdot \nabla\varphi, \tilde{\eta}_s\right\rangle \mathd s.
	\end{align*} 
Due to the observation made in  \cite[Lemma 3.5, Lemma 3.6]{hofmanova2013weak}, there exists a sequence $\{M_n\}\rightarrow\infty$ such that \begin{align*}
			\tilde{\mathbb{P}}\(\vartheta_{M_n}(\cdot)\text{ is continuous at $\tilde{I}$}\)=1.
		\end{align*}
		Consequencely, we establish the convergence of stopping times, that is, for fixed $n\in \mathbb{N},$
		\begin{align*}
			\vartheta_{M_n}(\tilde{I}^N) \xrightarrow{N\rightarrow\infty}\vartheta_{M_n}(\tilde{I}),\quad \tilde{\mathbb{P}}-a.s..
		\end{align*}
		
		 Since   the joint law of $(\tilde{v}^N,\tilde{\eta}^N,
		\tilde{\mathcal{M}}^N,\tilde{W}^N)$ coincides with the joint law of $(v,\eta^N,\mathcal{M}^N,W),$ we then have for every $n\in \mathbb{N}$ and $0\leq s<t \leq T,$ 
		\begin{align*}
			&\tilde{	\mathbb{E}}\[\tilde{\gamma}^N \tilde{Z}^N(t\wedge\vartheta_{M_n}(\tilde{I}^N))\]=\tilde{	\mathbb{E}}\[\tilde{\gamma}^N \tilde{Z}^N(s\wedge\vartheta_{M_n}(\tilde{I}^N))\],
			\\&\tilde{	\mathbb{E}}\[\tilde{\gamma}^N\((\tilde{Z}^N(t\wedge\vartheta_{M_n}(\tilde{I}^N))^2- \int_0^{t\wedge\vartheta_{M_n}(\tilde{I}^N)}  \left\langle \sigma\cdot \nabla\varphi,\tilde{\eta}^N_r\right\rangle ^2 \mathd r\) \]\\=&\tilde{	\mathbb{E}}\[\tilde{\gamma}^N\((\tilde{Z}^N(s\wedge\tau_{M_n}(\tilde{I}^N)))^2- \int_0^{s\wedge\tau_{M_n}(\tilde{I}^N)} \left\langle \sigma\cdot \nabla\varphi,\tilde{\eta}^N_r\right\rangle ^2 \mathd r\) \],	
		\end{align*}
		and
		\begin{align*}
			&\tilde{	\mathbb{E}}\[\tilde{\gamma}^N\(\tilde{W}^N(t\wedge\vartheta_{M_n}(\tilde{I}^N))\tilde{Z}^N(t\wedge\vartheta_{M_n}(\tilde{I}^N))-\int_{0}^{t\wedge\vartheta_{M_n}(\tilde{I}^N)}\left\langle \sigma\cdot \nabla\varphi,\tilde{\eta}^N_r\right\rangle \mathd r\)\]\\=&\tilde{	\mathbb{E}}\[\tilde{\gamma}^N\(\tilde{W}^N(s\wedge\vartheta_{M_n}(\tilde{I}^N))\tilde{Z}^N(s\wedge\vartheta_{M_n}(\tilde{I}^N))-\int_{0}^{s\wedge\vartheta_{M_n}(\tilde{I}^N)}\left\langle \sigma\cdot \nabla\varphi,\tilde{\eta}^N_r\right\rangle \mathd r\)\].
		\end{align*}
		The Burkholder-Davis-Gundy's inequality for $\tilde{Z}^N(t\wedge\vartheta_{M_n}(\tilde{I}^N))$ yields the uniform bound
		\begin{align*}
			&\tilde{\mathbb{E}}\big|\tilde{Z}^N(t\wedge\vartheta_{M_n}(\tilde{I}^N))\big|^4=\mathbb{E}\big|\int_{0}^{t\wedge\vartheta_{M_n}(I^N)}\left\langle \sigma\cdot \nabla\varphi,\tilde{\eta}^N_r\right\rangle \mathd W_r\big|^4\\&\leq \mathbb{E}\big|\int_{0}^{t\wedge\vartheta_{M_n}(I^N)}\left\langle \sigma\cdot \nabla\varphi,\tilde{\eta}^N_r\right\rangle^2 \mathd r\big|^2\leq C_{M_n},\quad \forall N\in \mathbb{N},
		\end{align*}
		which provide the necessary uniform integrability. 
		We thus can pass the limit in equations and infer  \begin{align*}
			&\tilde{	\mathbb{E}}\[\tilde{\gamma}\tilde{Z}(t\wedge\vartheta_{M_n}(\tilde{I}))\]=\tilde{	\mathbb{E}}\[\tilde{\gamma}\tilde{Z}(s\wedge\vartheta_{M_n}(\tilde{I}))\],
			\\&\tilde{	\mathbb{E}}\[\tilde{\gamma}\(\tilde{Z}^2(t\wedge\vartheta_{M_n}(\tilde{I}))- \int_0^{t\wedge\vartheta_{M_n}(\tilde{I})}  \left\langle \sigma\cdot \nabla\varphi,\tilde{\eta}_r\right\rangle ^2 \mathd r\) \]\\=&\tilde{	\mathbb{E}}\[\tilde{\gamma}\(\tilde{Z}^2(s\wedge\vartheta_{M_n}(\tilde{I}))- \int_0^{s\wedge\vartheta_{M_n}(\tilde{I})} \left\langle \sigma\cdot \nabla\varphi,\tilde{\eta}_r\right\rangle ^2 \mathd r\) \],	
		\end{align*}
		and
		\begin{align*}
			&\tilde{	\mathbb{E}}\[\tilde{\gamma}\(\tilde{W}(t\wedge\vartheta_{M_n}(\tilde{I}))\tilde{Z}(t\wedge\vartheta_{M_n}(\tilde{I}))-\int_{0}^{t\wedge\vartheta_{M_n}(\tilde{I})}\left\langle \sigma\cdot \nabla\varphi,\tilde{\eta}_r\right\rangle \mathd r\)\]\\=&\tilde{	\mathbb{E}}\[\tilde{\gamma}\(\tilde{W}(s\wedge\vartheta_{M_n}(\tilde{I}))\tilde{Z}(s\wedge\vartheta_{M_n}(\tilde{I}))-\int_{0}^{s\wedge\vartheta_{M_n}(\tilde{I})}\left\langle \sigma\cdot \nabla\varphi,\tilde{\eta}_r\right\rangle \mathd r\)\].
		\end{align*}
		Therefore, for every $n\in \mathbb{N},$ 
		\begin{align*}
			&\tilde{Z}(\cdot\wedge\vartheta_{M_n}(\tilde{I})),
			\\&\tilde{Z}^2(\cdot\wedge\vartheta_{M_n}(\tilde{I}))- \int_0^{\cdot\wedge\vartheta_{M_n}(\tilde{I})}  \left\langle \sigma\cdot \nabla\varphi,\tilde{\eta}_r\right\rangle ^2 \mathd r,
			\\&\tilde{W}(\cdot \wedge\vartheta_{M_n}(\tilde{I}))\tilde{Z}(\cdot\wedge\vartheta_{M_n}(\tilde{I}))-\int_{0}^{\cdot\wedge\vartheta_{M_n}(\tilde{I})}\left\langle \sigma\cdot \nabla\varphi,\tilde{\eta}_r\right\rangle \mathd r,
		\end{align*}
		are $(\tilde{\mathcal{F}}_t)_{t\in[0,T]}-$martingales, which completes the proof.
	\end{proof}
	\begin{proof}[Proof of Theorem 4.5]
		Having Proposition \ref{prop:w} and Proposition \ref{prop:transport} in hand, we can directly calculate \begin{align*}
			\[\tilde{Z}(\cdot\wedge\vartheta_{M_n}(\tilde{I}))-\int_{0}^{\cdot\wedge\vartheta_{M_n}(\tilde{I})}\left\langle \sigma\cdot \nabla\varphi,\tilde{\eta}_r\right\rangle \mathd W_r,\tilde{Z}(\cdot\wedge\vartheta_{M_n}(\tilde{I}))-\int_{0}^{\cdot\wedge\vartheta_{M_n}(\tilde{I})}\left\langle \sigma\cdot \nabla\varphi,\tilde{\eta}_r\right\rangle \mathd W_r\]_t=0
		\end{align*}
		for every $n\in \mathbb{N}.$  Let's $n\rightarrow\infty,$ the fluctuation SPDE \eqref{eqt-limitspde} holds. Combining the results of Lemma \ref{lem:unilaw} and Proposition  \ref{prop:w},  we complete the proof.
	\end{proof}
\subsection{Uniqueness}\label{sec:unique}
In this subsection, we demonstrate  pathwise uniqueness for  the fluctuation SPDE \eqref{LimitSpde}  and complete the proof of our main results, Theorem \ref{thm:main1} and Theorem \ref{thm:main}.
\begin{theorem}\label{thm:uni}
	Pathwise uniqueness holds true for the fluctuation SPDE \eqref{LimitSpde} in the sense of  Deifinition \eqref{def:uni}.
\end{theorem}
	\begin{proof}
		Let ${W_t, t \in [0,T]}$ be a  1-dimensional Brownian motion and stochastic process $(\mathcal{M}_t)_{t\in [0,T]}$ takes values in $\bigcap_{k\in \mathbb{N}}C ([0, T];H^{-2-\frac{1}{k}}(\mathbb{T}^d)),$ both defined on the same stochastic basis $(\Omega,\mathcal{F}, (\mathcal{G}_t)_{t\in [0,T]},\mathbb{P}).$ Suppose further that $$\bigg(\Omega,\mathcal{F},(\mathcal{G}_t)_{t\in [0,T]},\mathbb{P},\(  \eta^i_t,\mathcal{M}_t,W_t\)_{t\in [0,T]}\bigg)_{i=1,2}$$  are  probabilistically weak solution to \eqref{LimitSpde} with the same initial data $\eta_0.$
We define $ \bar{\eta}\assign\eta^1-\eta^2$  and obtain that $\forall \varphi\in C^{\infty}(\mathbb{T}^2),$
	\begin{align*}
		\langle \bar{\eta}_t, \varphi \rangle =&	\int^t_0 \langle  \Delta \varphi, \bar{\eta}_s \rangle \mathd \nocomma s +
		\int^t_0 \langle \nabla \varphi, v_s K \ast \bar{\eta}_s \rangle \mathd
		\nocomma s + \int^t_0 \langle \nabla \varphi, \bar{\eta}_s K \ast v_s
		\rangle \mathd \nocomma s \nonumber\\
		&+\frac{1}{2}\int_0^t  \left\langle\sigma\cdot\nabla\(\sigma\cdot\nabla\varphi\),\bar{\eta}_s\right\rangle \mathd s+\int_{0}^{t}\left\langle \sigma\cdot \nabla\varphi,\bar{\eta}_s\right\rangle \mathd W_s,\quad \forall t\in [0,T],\mathbb{P}-a.s..
	\end{align*}
Then, we evolve $|\langle \bar{\eta}_{\cdot}, \varphi \rangle|^2$ for each $e_k$ of  the Fourier basis $\{e_k\assign e^{\sqrt{-1}k\cdot x}, k\in \mathbb{Z}^2\},$ using It\^o's formula.
	\begin{align*}
		\mathd \mid \langle\bar{\eta}_t,e_k\rangle\mid^2
		&=\langle\bar{\eta}_t,e_{k}\rangle \[\langle  \Delta e_{-k}, \bar{\eta}_t \rangle +\langle \nabla e_{-k}, v_t K \ast \bar{\eta}_t \rangle\\&+ \langle \nabla e_{-k}, \bar{\eta}_t K \ast v_t
		\rangle +\frac{1}{2}\left\langle\sigma\cdot\nabla\(\sigma\cdot\nabla e_{-k}\),\bar{\eta}_t\right\rangle\]\mathd t
		\\&+\langle\bar{\eta}_t,e_{-k}\rangle \[\langle  \Delta e_{k}, \bar{\eta}_t \rangle +\langle \nabla e_{k}, v_t K \ast \bar{\eta}_t \rangle\\&+ \langle \nabla e_{k}, \bar{\eta}_t K \ast v_t
		\rangle +\frac{1}{2}\left\langle\sigma\cdot\nabla\(\sigma\cdot\nabla e_{k}\),\bar{\eta}_t\right\rangle\]\mathd t
		\\&+\langle\bar{\eta}_t,e_{-k}\rangle \left\langle \sigma\cdot \nabla e_{k},\bar{\eta}_t\right\rangle \mathd W_t+\langle\bar{\eta}_t,e_{k}\rangle \left\langle \sigma\cdot \nabla e_{-k},\bar{\eta}_t\right\rangle \mathd W_t
		\\&+\left\langle \sigma\cdot \nabla e_{-k},\bar{\eta}_t\right\rangle\left\langle \sigma\cdot \nabla e_{k},\bar{\eta}_t\right\rangle \mathd t.
	\end{align*}
 For fixed $3>\alpha>2,$ we now sum  $(1+|k|^2)^{-\alpha-1}\mid \langle\bar{\eta}_t,e_k\rangle\mid^2$ over $k\in \mathbb{Z}^2,$ and obtain 
	\begin{align*}
		\|\bar{\eta}_t\|_{H^{-\alpha-1}}^2	&=\sum_{i=1}^{3}J_i(t)+L_t,
	\end{align*}
	where \begin{align*}
		J_1(t)&\assign-2\int_{0}^{t}\sum_{k \in \mathbb{Z}^2}|k|^2(1+|k|^2)^{-\alpha-1}\mid \langle\bar{\eta}_s,e_k\rangle\mid^2\mathd s,
		\\J_2(t)&\assign\sum_{k \in \mathbb{Z}^2}(1+|k|^2)^{-\alpha-1}\int_{0}^{t}\langle \bar{\eta}_s, e_{-
			k} \rangle  \bigg[\sqrt{- 1} k \cdot\langle K \ast \bar{\eta}_sv_s,e_k\rangle\\&+ \sqrt{- 1} k \cdot\langle K \ast v_s \bar{\eta}_s,e_k\rangle\bigg] + \langle \bar{\eta}_s, e_k \rangle  \bigg[-\sqrt{- 1} k\cdot\langle K \ast \bar{\eta}_sv_s,e_{-k}\rangle \\&- \sqrt{- 1} k \cdot\langle K \ast v_s \bar{\eta}_s,e_{-k}\rangle\bigg]\mathd s,
		\\J_3(t)&\assign\int_{0}^{t}\|\sigma\cdot\nabla\bar{\eta}_s\|_{H^{-\alpha-1}}^2+\left\langle \sigma\cdot \nabla(\sigma\cdot\nabla\bar{\eta}_s),\bar{\eta}_s\right\rangle_{H^{-\alpha-1}}\mathd s,
	\end{align*}
	and $L_t\assign\int_{0}^{t}\left\langle \sigma\cdot \nabla\bar{\eta}_s,\bar{\eta}_s\right\rangle_{H^{-\alpha-1}}\mathd W_s$ is a continuous local martingale.
	
	Applying Young's inequality, we find  for every $\eps>0,$  there exists a constant $C_{\eps}$ such that
	\begin{align*} 
		\int_{0}^{t}\sqrt{- 1} k\cdot\langle \bar{\eta}_s, e_{-
			k} \rangle   \langle K \ast \bar{\eta}_sv_s,e_k\rangle\mathd s&\leq\eps\int_{0}^{t}|k|^2|\langle \bar{\eta}_s, e_{-
			k} \rangle |^2\mathd s+C_{\eps}\int_{0}^{t} |\langle K \ast \bar{\eta}_sv_s,e_k\rangle|^2\mathd s.
	\end{align*}
Consequently, we have
	\begin{align*}
		&\int_{0}^{t}	\sum_{k \in \mathbb{Z}^2}(1+|k|^2)^{-\alpha-1}\sqrt{- 1} k\cdot\langle \bar{\eta}_s, e_{-
			k} \rangle \langle K \ast \bar{\eta}_sv_s,e_k\rangle\mathd s
		\\&\leq \eps\int_{0}^{t}\|\eta_s\|^2_{H^{-\alpha}}\mathd s+C_\eps \int_{0}^{t}\| K \ast \bar{\eta}_sv_s\|_{H^{-\alpha-1}}^2\mathd s.
	\end{align*}
	By Lemma \ref{lemma embedding}, Lemma  \ref{lemma triebel} and \cite[Theorem A.1.3]{breit2018stochastically}, for $2<\alpha<3,$
	we know 
	\begin{align*}
		\| K \ast \bar{\eta}_sv_s\|_{H^{-\alpha-1}}&\leq\|K \ast \bar{\eta}_sv_s\|_{H^{-\alpha}}
		\\&\lesssim \|K \ast \bar{\eta}_s\|_{H^{-\alpha}}\|v_s\|_{H^{4}}
		\\&=\|\nabla^{\perp} (-\Delta)^{-1}\bar{\eta}_s\|_{H^{-\alpha}}\|v_s\|_{H^{4}}
		\\&\lesssim\|\bar{\eta}_s\|_{H^{-\alpha-1}}\|v_s\|_{H^{4}}.
	\end{align*}
Here we used the fact   Biot-Savart kernel $K=\nabla^{\perp}G,$ where $G$ is the Green function of $-\Delta$ on torus, as given in \cite{flandoli2011full}.
	We thus conclude that \begin{align*}
		\int_{0}^{t}	\sum_{k \in \mathbb{Z}^d}(1+|k|^2)^{-\alpha-1}\sqrt{- 1} k\cdot \langle \bar{\eta}_s, e_{-
			k} \rangle \langle K \ast \bar{\eta}_sv_s,e_k\rangle\mathd s
		&\leq \eps\int_{0}^{t}\|\eta_s\|^2_{H^{-\alpha}}\mathd s+C_\eps \int_{0}^{t} \|v_s\|^2_{H^{4}}\|\bar{\eta}\|_{H^{-\alpha-1}}^2 \mathd s.
	\end{align*}
The other three terms in $J_2(t)$ can be controlled in a similar way. We then obtain that   for every $\eps>0,$  there exists a constant $C_{\eps}$ such that for $2<\alpha<3,$
	\begin{align*}
		J_2(t)&\leq \eps\int_{0}^{t}\|\eta_s\|^2_{H^{-\alpha}}\mathd s+C_{\eps} \int_{0}^{t} \|v_s\|^2_{H^{4}}\|\bar{\eta}\|_{H^{-\alpha-1}}^2 \mathd s.
	\end{align*}
	 Choosing $\eps<2,$  we conclude that  for $2<\alpha<3,$
	\begin{align*}
		J_1(t)+J_2(t)&=-2\int_{0}^{t}\sum_{k \in \mathbb{Z}^2}(1+|k|^2)^{-\alpha}\mid \langle\bar{\eta}_s,e_k\rangle\mid^2\mathd s+2\int_{0}^{t}\sum_{k \in \mathbb{Z}^2}(1+|k|^2)^{-\alpha-1}\mid \langle\bar{\eta}_s,e_k\rangle\mid^2\mathd s+J_2(t)
		\\&\leq(\eps-2)\int_{0}^{t}\|\eta_s\|^2_{H^{-\alpha}}\mathd s+C_{\eps} \int_{0}^{t} \|v_s\|^2_{H^{4}}\|\bar{\eta}\|_{H^{-\alpha-1}}^2 \mathd s+ 2\int_{0}^{t} \|\bar{\eta}\|_{H^{-\alpha-1}}^2 \mathd s
		\\&\lesssim\int_{0}^{t} (\|v_s\|^2_{H^{4}}+1)\|\bar{\eta}\|_{H^{-\alpha-1}}^2 \mathd s.
	\end{align*}
	Furthermore, applying \cite[Lemma 2.3 ]{krylov2015filtering} about commutator estimate yields 
	\begin{align*}
		J_3(t)\leq C_{\sigma} \int_{0}^{t} \|\bar{\eta}\|_{H^{-\alpha-1}}^2 \mathd s.
	\end{align*}
	Consequently, we finally have \begin{align*}
		\|\bar{\eta}_t\|_{H^{-\alpha-1}}^2\lesssim \int_{0}^{t}(\|v_s\|^2_{H^{4}}+1)\|\bar{\eta}\|_{H^{-\alpha-1}}^2 \mathd s+L_t.
	\end{align*}
	Using stochastic Gronwall's inequality in \cite[Corollary 5.4]{geiss2021sharp}, we have $\sup_{t \in [0, T]}\|\bar{\eta}_t\|_{H^{-\alpha-1}}^2=0\quad \mathbb{P}$-a.s.. The proof is then completed.
\end{proof}

\begin{proof}[Proof of Theorem \ref{thm:main}]	
	Having established the existence of probabilistically weak solution to the fluctuation SPDE \eqref{LimitSpde} and the  pathwise uniqueness for the fluctuation SPDE \eqref{LimitSpde}  given in Theorem \ref{thm:existence} and Theorem \ref{thm:uni} in hand, we apply the general Yamada-Watanabe theorem  {\cite[Theorem
		1.5]{kurtz2014yamada}} to conclude the well-posedness of \eqref{LimitSpde}.
We have shown in Lemma \ref{tightness} that the sequence of laws of $\{\eta^N \}_{N \in \mathbb{N}}$
is tight, and in Theorem \ref{thm:existence} that every limiting point
is a probabilistically weak solution to \eqref{LimitSpde}.  Through the well-posedness of \eqref{LimitSpde} and general Yamada-Watanabe theorem  {\cite[Theorem
	1.5]{kurtz2014yamada}}, we then conclude  that every limiting point is the unique probabilistically strong
solution to \eqref{LimitSpde} and the law of every  limiting point is unique. This establishes the convergence of the fluctuation measures $(\eta^N)_{N\geq 1}$ to  the fluctuation SPDE \eqref{LimitSpde}, completing the proof.

\end{proof}
\subsubsection*{\textbf{Acknowledgement:}}The authors would like to thank Rongchan Zhu, Zhenfu Wang and Lin L\"u for quite useful discussion and advice.

	\bibliographystyle{alpha}
		\bibliography{Fluctuation}

\newcommand{\etalchar}[1]{$^{#1}$}
\begin{thebibliography}{GLBM24}

\bibitem[BCD11]{chemin2011fourier}
Hajer Bahouri, Jean-Yves Chemin, and Rapha{\"e}l Danchin.
\newblock {\em Fourier analysis and nonlinear partial differential equations},
  volume 343.
\newblock Springer Science \& Business Media, 2011.

\bibitem[BD24]{bernou2024uniform}
Armand Bernou and Mitia Duerinckx.
\newblock Uniform-in-time estimates on the size of chaos for interacting
  brownian particles.
\newblock {\em arXiv preprint arXiv:2405.19306}, 2024.

\bibitem[BFH18]{breit2018stochastically}
Dominic Breit, Eduard Feireisl, and Martina Hofmanov{\'a}.
\newblock {\em Stochastically forced compressible fluid flows}, volume~3.
\newblock Walter de Gruyter GmbH \& Co KG, 2018.

\bibitem[BH77]{braun1977vlasov}
Werner Braun and Klaus Hepp.
\newblock The vlasov dynamics and its fluctuations in the 1/n limit of
  interacting classical particles.
\newblock {\em Communications in mathematical physics}, 56(2):101--113, 1977.

\bibitem[BJW23]{bresch2023mean}
Didier Bresch, Pierre-Emmanuel Jabin, and Zhenfu Wang.
\newblock Mean field limit and quantitative estimates with singular attractive
  kernels.
\newblock {\em Duke Mathematical Journal}, 172(13):2591--2641, 2023.

\bibitem[BW16]{budhiraja2016some}
Amarjit Budhiraja and Ruoyu Wu.
\newblock Some fluctuation results for weakly interacting multi-type particle
  systems.
\newblock {\em Stochastic Processes and their Applications}, 126(8):2253--2296,
  2016.

\bibitem[CF16a]{chen2016fluctuation}
Zhen-Qing Chen and Wai-Tong~Louis Fan.
\newblock Fluctuation limit for interacting diffusions with partial
  annihilations through membranes.
\newblock {\em Journal of Statistical Physics}, 164(4):890--936, 2016.

\bibitem[CF16b]{coghi2016propagation}
Michele Coghi and Franco Flandoli.
\newblock Propagation of chaos for interacting particles subject to
  environmental noise.
\newblock 2016.

\bibitem[CFG{\etalchar{+}}24]{carrillo2024relative}
Jos{\'e}~Antonio Carrillo, Xuanrui Feng, Shuchen Guo, Pierre-Emmanuel Jabin,
  and Zhenfu Wang.
\newblock Relative entropy method for particle approximation of the landau
  equation for maxwellian molecules.
\newblock {\em arXiv preprint arXiv:2408.15035}, 2024.

\bibitem[Cha24]{chaintron2024quasi}
Louis-Pierre Chaintron.
\newblock Quasi-continuity method for mean-field diffusions: large deviations
  and central limit theorem.
\newblock {\em arXiv preprint arXiv:2410.04935}, 2024.

\bibitem[Che17]{chevallier2017fluctuations}
Julien Chevallier.
\newblock Fluctuations for mean-field interacting age-dependent hawkes
  processes.
\newblock {\em Electronic Journal of Probability}, 22, 2017.

\bibitem[CHJ24]{chen2024fluctuations}
Li~Chen, Alexandra Holzinger, and Ansgar J{\"u}ngel.
\newblock Fluctuations around the mean-field limit for attractive riesz
  potentials in the moderate regime.
\newblock {\em arXiv preprint arXiv:2405.15128}, 2024.

\bibitem[DD19]{durrett2019probability}
Richard Durrett and Rick Durrett.
\newblock {\em Probability: theory and examples}.
\newblock Cambridge university press, 2019.

\bibitem[DHR21]{dabrock2021existence}
Nils Dabrock, Martina Hofmanov{\'a}, and Matthias R{\"o}ger.
\newblock Existence of martingale solutions and large-time behavior for a
  stochastic mean curvature flow of graphs.
\newblock {\em Probability Theory and Related Fields}, 179:407--449, 2021.

\bibitem[Due16]{duerinckx2016mean}
Mitia Duerinckx.
\newblock Mean-field limits for some riesz interaction gradient flows.
\newblock {\em SIAM Journal on Mathematical Analysis}, 48(3):2269--2300, 2016.

\bibitem[Due21]{duerinckx2021size}
Mitia Duerinckx.
\newblock On the size of chaos via glauber calculus in the classical mean-field
  dynamics.
\newblock {\em Communications in Mathematical Physics}, 382(1):613--653, 2021.

\bibitem[FGP11]{flandoli2011full}
Franco Flandoli, Massimiliano Gubinelli, and Enrico Priola.
\newblock Full well-posedness of point vortex dynamics corresponding to
  stochastic 2d euler equations.
\newblock {\em Stochastic Processes and their Applications}, 121(7):1445--1463,
  2011.

\bibitem[FHM14]{fournier2014propagation}
Nicolas Fournier, Maxime Hauray, and St{\'e}phane Mischler.
\newblock Propagation of chaos for the 2d viscous vortex model.
\newblock {\em Journal of the European Mathematical Society}, 16(7):1423--1466,
  2014.

\bibitem[FL21]{flandoli2021mean}
Franco Flandoli and Dejun Luo.
\newblock Mean field limit of point vortices with environmental noises to
  deterministic 2d navier-stokes equations.
\newblock {\em arXiv preprint arXiv:2103.01497}, 2021.

\bibitem[Fla11]{flandoli2011random}
Franco Flandoli.
\newblock {\em Random Perturbation of PDEs and Fluid Dynamic Models: {\'E}cole
  d’{\'e}t{\'e} de Probabilit{\'e}s de Saint-Flour XL--2010}, volume 2015.
\newblock Springer Science \& Business Media, 2011.

\bibitem[FM97]{fernandez1997hilbertian}
Begona Fernandez and Sylvie M{\'e}l{\'e}ard.
\newblock A hilbertian approach for fluctuations on the mckean-vlasov model.
\newblock {\em Stochastic processes and their applications}, 71(1):33--53,
  1997.

\bibitem[FW23]{feng2023quantitative}
Xuanrui Feng and Zhenfu Wang.
\newblock Quantitative propagation of chaos for 2d viscous vortex model on the
  whole space.
\newblock {\em arXiv preprint arXiv:2310.05156}, 2023.

\bibitem[Gei21]{geiss2021sharp}
Sarah Geiss.
\newblock Sharp convex generalizations of stochastic gronwall inequalities.
\newblock {\em arXiv preprint arXiv:2112.05047}, 2021.

\bibitem[GL23]{guo2023scaling}
Shuchen Guo and Dejun Luo.
\newblock Scaling limit of moderately interacting particle systems with
  singular interaction and environmental noise.
\newblock {\em The Annals of Applied Probability}, 33(3):2066--2102, 2023.

\bibitem[GLBM24]{guillin2024uniform}
Arnaud Guillin, Pierre Le~Bris, and Pierre Monmarch{\'e}.
\newblock Uniform in time propagation of chaos for the 2d vortex model and
  other singular stochastic systems.
\newblock {\em Journal of the European Mathematical Society}, 2024.

\bibitem[HRvR17]{hofmanova2017weak}
Martina Hofmanov{\'a}, Matthias R{\"o}ger, and Max von Renesse.
\newblock Weak solutions for a stochastic mean curvature flow of
  two-dimensional graphs.
\newblock {\em Probability Theory and Related Fields}, 168:373--408, 2017.

\bibitem[HS13]{hofmanova2013weak}
Martina Hofmanov{\'a} and Jan Seidler.
\newblock On weak solutions of stochastic differential equations ii.
\newblock {\em Stochastic Analysis and Applications}, 31(4):663--670, 2013.

\bibitem[It{\^o}83]{ito1983distribution}
Kiyosi It{\^o}.
\newblock Distribution-valued processes arising from independent brownian
  motions.
\newblock {\em Mathematische Zeitschrift}, 182:17--33, 1983.

\bibitem[JM98]{jourdain1998propagation}
Benjamin Jourdain and Sylvie M{\'e}l{\'e}ard.
\newblock Propagation of chaos and fluctuations for a moderate model with
  smooth initial data.
\newblock In {\em Annales de l'Institut Henri Poincare (B) Probability and
  Statistics}, volume~34, pages 727--766. Elsevier, 1998.

\bibitem[JW16]{jabin2016mean}
Pierre-Emmanuel Jabin and Zhenfu Wang.
\newblock Mean field limit and propagation of chaos for vlasov systems with
  bounded forces.
\newblock {\em Journal of Functional Analysis}, 271(12):3588--3627, 2016.

\bibitem[JW18]{jabin2018quantitative}
Pierre-Emmanuel Jabin and Zhenfu Wang.
\newblock Quantitative estimates of propagation of chaos for stochastic systems
  with $w^{-1,\infty}$ kernels.
\newblock {\em Inventiones mathematicae}, 214(1):523--591, 2018.

\bibitem[Kel17]{kelley2017topology}
John~L Kelley.
\newblock {\em General topology}.
\newblock Courier Dover Publications, 2017.

\bibitem[Kry15]{krylov2015filtering}
Nicolai~V Krylov.
\newblock Hypoellipticity for filtering problems of partially observable
  diffusion processes.
\newblock {\em Probability Theory and Related Fields}, 161(3-4):687--718, 2015.

\bibitem[KS21]{kuhn2021convolution}
Franziska Kuhn and Rene~L. Schilling.
\newblock Convolution inequalities for besov and triebel lizorkin spaces, and
  applications to convolution semigroups.
\newblock 2021.

\bibitem[Kur14]{kurtz2014yamada}
Thomas Kurtz.
\newblock Weak and strong solutions of general stochastic models.
\newblock {\em Electronic Communications in Probability}, 19, 2014.

\bibitem[KX99]{kurtz1999particle}
Thomas~G Kurtz and Jie Xiong.
\newblock Particle representations for a class of nonlinear spdes.
\newblock {\em Stochastic Processes and their Applications}, 83(1):103--126,
  1999.

\bibitem[KX04]{kurtz2004stochastic}
Thomas~G Kurtz and Jie Xiong.
\newblock A stochastic evolution equation arising from the fluctuations of a
  class of interacting particle systems.
\newblock 2004.

\bibitem[Lac15]{lacker2015stochastic}
Daniel Lacker.
\newblock {\em Stochastic differential mean field game theory}.
\newblock PhD thesis, Princeton University, 2015.

\bibitem[Lac23]{lacker2023hierarchies}
Daniel Lacker.
\newblock Hierarchies, entropy, and quantitative propagation of chaos for mean
  field diffusions.
\newblock {\em Probability and Mathematical Physics}, 4(2):377--432, 2023.

\bibitem[Lan09]{lancellotti2009fluctuations}
Carlo Lancellotti.
\newblock On the fluctuations about the vlasov limit for n-particle systems
  with mean-field interactions.
\newblock {\em Journal of Statistical Physics}, 136(4):643--665, 2009.

\bibitem[LS16]{lucon2016transition}
Eric Lucon and Wilhelm Stannat.
\newblock Transition from gaussian to non-gaussian fluctuations for mean-field
  diffusions in spatial interaction.
\newblock {\em The Annals of Applied Probability}, 26(6):3840--3909, 2016.

\bibitem[MW17]{mourrat2017global}
Jean-Christophe Mourrat and Hendrik Weber.
\newblock Global well-posedness of the dynamic phi4 model in the plane.
\newblock {\em The Annals of Probability}, 45(4):2398--2476, 2017.

\bibitem[Nik24]{nikolaev2024quantitative}
Paul Nikolaev.
\newblock Quantitative relative entropy estimates for interacting particle
  systems with common noise.
\newblock {\em arXiv preprint arXiv:2407.01217}, 2024.

\bibitem[NRS21]{nguyen2021mean}
Quoc~Hung Nguyen, Matthew Rosenzweig, and Sylvia Serfaty.
\newblock Mean-field limits of riesz-type singular flows.
\newblock {\em arXiv preprint arXiv:2107.02592}, 2021.

\bibitem[Oel87]{oelschlager1987fluctuation}
Karl Oelschl{\"a}ger.
\newblock A fluctuation theorem for moderately interacting diffusion processes.
\newblock {\em Probability theory and related fields}, 74(4):591--616, 1987.

\bibitem[Osa86]{osada1986propagation}
Hirofumi Osada.
\newblock Propagation of chaos for the two dimensional navier-stokes equation.
\newblock {\em Proceedings of the Japan Academy, Series A, Mathematical
  Sciences}, 62(1):8--11, 1986.

\bibitem[Ros20]{rosenzweig2020mean}
Matthew Rosenzweig.
\newblock The mean-field limit of stochastic point vortex systems with
  multiplicative noise.
\newblock {\em arXiv preprint arXiv:2011.12180}, 2020.

\bibitem[Ser20]{serfaty2020mean}
Sylvia Serfaty.
\newblock Mean field limit for coulomb-type flows.
\newblock {\em Duke Mathematical Journal}, 169(15):2887--2935, 2020.

\bibitem[SZ24]{shao2024quantitative}
Yufei Shao and Xianliang Zhao.
\newblock Quantitative particle approximations of stochastic 2d navier-stokes
  equation.
\newblock {\em arXiv preprint arXiv:2402.02336}, 2024.

\bibitem[Szn84]{sznitman1984nonlinear}
Alain-Sol Sznitman.
\newblock Nonlinear reflecting diffusion process, and the propagation of chaos
  and fluctuations associated.
\newblock {\em Journal of functional analysis}, 56(3):311--336, 1984.

\bibitem[Szn91]{sznitman1991topics}
Alain-Sol Sznitman.
\newblock Topics in propagation of chaos.
\newblock In {\em Ecole d'{\'e}t{\'e} de probabilit{\'e}s de Saint-Flour
  XIX—1989}, pages 165--251. Springer, 1991.

\bibitem[Tan84]{tanaka1984limit}
Hiroshi Tanaka.
\newblock Limit theorems for certain diffusion processes with interaction.
\newblock In {\em North-Holland Mathematical Library}, volume~32, pages
  469--488. Elsevier, 1984.

\bibitem[TH81]{tanaka1981central}
Hiroshi Tanaka and Masuyuki Hitsuda.
\newblock Central limit theorem for a simple diffusion model of interacting
  particles.
\newblock {\em Hiroshima Mathematical Journal}, 11(2):415--423, 1981.

\bibitem[Tri06]{triebel2006theory}
Hans Triebel.
\newblock Theory of function spaces. iii,.
\newblock {\em Monographs in Mathematics}, 2006.

\bibitem[Wan24]{wang2024sharp}
Songbo Wang.
\newblock Sharp local propagation of chaos for mean field particles with
  $w^{-1,\infty}$ kernels.
\newblock {\em arXiv preprint arXiv:2403.13161}, 2024.

\bibitem[WZZ23]{wang2023gaussian}
Zhenfu Wang, Xianliang Zhao, and Rongchan Zhu.
\newblock Gaussian fluctuations for interacting particle systems with singular
  kernels.
\newblock {\em Archive for Rational Mechanics and Analysis}, 247(5):101, 2023.

\end{thebibliography}
															
														\end{document}